\documentclass[11pt]{article}

\usepackage{amssymb,amsmath,amsfonts,amsthm}

\usepackage{caption2}

\renewcommand{\subsection}{\subsubsection}

\renewcommand\appendix{\par
  \setcounter{section}{0}
   \renewcommand\thesection{Appendix \Alph{section}.}
 }

\newtheorem{theorem}{Theorem}[section]

\newtheorem{lemma}{Lemma}[section]
\newtheorem{proposition}{Proposition}[section]
\newtheorem{definition}{Definition}[section]

\newtheorem{remark}{Remark}[section]

\newcommand{\nt}{|\hspace{-0.7pt}|\hspace{-0.7pt}|}
\newcommand{\nl}{\langle\hspace*{-2pt}\langle}
\newcommand{\nr}{\rangle\hspace*{-2pt}\rangle}

\def\nt{|\hspace{-0.7pt}|\hspace{-0.7pt}|}

\def\div{{\rm div}\, }

\begin{document}
\title{\bf Local existence of MHD contact discontinuities}
%\subtitle{Do you have a subtitle?\\ If so, write it here}
\author{{\bf Alessandro Morando}\\
DICATAM, Sezione di Matematica, Universit\`a di Brescia \\ Via Valotti, 9, 25133 Brescia, Italy\\
E-mail: alessandro.morando@unibs.it
\and
{\bf Yuri Trakhinin}\\
Sobolev Institute of Mathematics, Koptyug av. 4, 630090 Novosibirsk, Russia\\
and\\
Novosibirsk State University, Pirogova str. 2, 630090 Novosibirsk, Russia\\
E-mail: trakhin@math.nsc.ru
\and
{\bf Paola Trebeschi}\\
DICATAM, Sezione di Matematica, Universit\`a di Brescia \\ Via Valotti, 9, 25133 Brescia, Italy\\
E-mail: paola.trebeschi@unibs.it
}

\date{
%Received: date / Accepted: date
}
% The correct dates will be entered by Springer
%
% Add name of the expert who has communicated your paper
%\communicated{name}
%
\maketitle
\begin{abstract}
We prove the local-in-time existence of solutions with a contact discontinuity of the equations of ideal compressible magnetohydrodynamics (MHD) for 2D planar flows provided that the Rayleigh-Taylor sign condition $[\partial p/\partial N]<0$ on the jump of the normal derivative of the pressure is satisfied
at each point of the initial discontinuity. MHD contact discontinuities are characteristic discontinuities with no flow across the discontinuity for which the pressure, the magnetic field and the velocity are continuous whereas the density and the entropy may have a jump. This paper is a natural completion of
our previous analysis ({\sc Morando, Trakhinin, Trebeschi} in J Differential Equations 258:2531--2571, 2015) where the well-posedness in Sobolev spaces of the linearized problem was proved under the Rayleigh-Taylor sign condition satisfied at each point of the unperturbed discontinuity. The proof of the resolution of the nonlinear problem given in the present paper follows from a suitable tame a priori estimate in Sobolev spaces for the linearized equations and a Nash-Moser iteration.
\end{abstract}

\section{Introduction}
\label{s1}

\subsection{Free boundary problem for MHD contact discontinuities}

\label{intro}
We consider the equations of ideal compressible MHD:
\begin{equation}\label{1}
\left\{
\begin{array}{l}
 \partial_t\rho  +{\rm div}\, (\rho {v} )=0,\\[6pt]
 \partial_t(\rho {v} ) +{\rm div}\,(\rho{v}\otimes{v} -{H}\otimes{H} ) +
{\nabla}q=0, \\[6pt]
 \partial_t{H} -{\nabla}\times ({v} {\times}{H})=0,\\[6pt]
 \partial_t\bigl( \rho e +{\textstyle \frac{1}{2}}|{H}|^2\bigr)+
{\rm div}\, \bigl((\rho e +p){v} +{H}{\times}({v}{\times}{H})\bigr)=0,
\end{array}
\right.
\end{equation}
where $\rho$ denotes density, $v$ plasma velocity, $H $ magnetic field, $p=p(\rho,S )$ pressure, $q =p+\frac{1}{2}|{H} |^2$ total pressure, $S$ entropy, $e=E+\frac{1}{2}|{v}|^2$ total energy, and  $E=E(\rho,S )$ internal energy. With a state equation of gas, $\rho=\rho(p ,S)$, and the first principle of thermodynamics, \eqref{1} is a closed system for the unknown $ U =U (t, x )=(p, v,H, S)$.

Unlike \cite{Cont1}, in this paper we do not first consider the general 3D case, and from the outset we restrict ourselves to {\it 2D planar} MHD flows. This means that the flow is $x_3$-invariant ($U=U(t,x_1,x_2)$), but the velocity and the magnetic field are shearless ($v_3=H_3=0$).\footnote{It follows from the 4th and 7th scalar equations of system \eqref{1} for $x_3$-invariant flows that $v_3|_{t=0}=H_3|_{t=0}=0$ implies $v_3=H_3=0$ for all $t>0$. That is, the restriction that the velocity and the magnetic field are shearless at a first moment guarantees that 2D flows are planar.} In other words, without loss of generality we may assume that the space variables, the velocity and the magnetic field have only two components: $x=(x_1,x_2)\in \mathbb{R}^2$, $v=(v_1,v_2)\in \mathbb{R}^2$, $H=(H_1,H_2)\in \mathbb{R}^2$. Moreover, we assume that the plasma obeys the state equation of a polytropic gas
\begin{equation}
\rho(p,S)= A p^{\frac{1}{\gamma}} e^{-\frac{S}{\gamma}}, \qquad A>0,\quad \gamma>1.
\label{pg}
\end{equation}

Taking into account the divergence constraint
\begin{equation}
{\rm div}\, {H} =0
\label{2}
\end{equation}
on the initial data ${U} (0,{x} )={U}_0({x})$, we easily symmetrize the system of conservation laws \eqref{1} by rewriting it in the nonconservative form
\begin{equation}
\left\{
\begin{array}{l}
{\displaystyle\frac{1}{\gamma p}\,\frac{{\rm d} p}{{\rm d}t} +{\rm div}\,{v} =0,\qquad
\rho\, \frac{{\rm d}v}{{\rm d}t}-({H}\cdot\nabla ){H}+{\nabla}q  =0 ,}\\[9pt]
{\displaystyle\frac{{\rm d}{H}}{{\rm d}t} - ({H} \cdot\nabla ){v} +
{H}\,{\rm div}\,{v}=0},\qquad
{\displaystyle\frac{{\rm d} S}{{\rm d} t} =0},
\end{array}\right. \label{3}
\end{equation}
where ${\rm d} /{\rm d} t =\partial_t+({v} \cdot{\nabla} )$. Equations (\ref{3}) form the symmetric system
\begin{equation}
\label{4}
A_0(U )\partial_tU+A_1(U )\partial_1U+A_2(U )\partial_2U=0
\end{equation}
with $A_0= {\rm diag} (1/(\gamma p) ,\rho ,\rho ,1,1,1)$ and
\[
A_1=\left( \begin{array}{cccccc} \frac{v_1}{\gamma p}& 1 & 0 & 0 & 0 & 0\\[6pt]
1 & \rho v_1 & 0 & 0& H_2& 0 \\
0& 0& \rho v_1 & 0& -H_1& 0 \\
0& 0& 0& v_1 & 0& 0\\
0& H_2& - H_1& 0& v_1 & 0\\
0& 0& 0& 0& 0& v_1
\end{array} \right),\quad
A_2=\left( \begin{array}{cccccc} \frac{v_2}{\gamma p}& 0 & 1& 0 & 0 & 0\\[6pt]
0 & \rho v_2 & 0 & - H_2& 0& 0 \\
1& 0& \rho v_2 & H_1& 0& 0 \\
0& -H_2& H_1& v_2 & 0& 0\\
0& 0& 0& 0& v_2 & 0\\
0& 0& 0& 0& 0& v_2
\end{array} \right).
\]
System \eqref{4} is hyperbolic if   $A_0>0$, i.e.,
\begin{equation}
p >0 \label{5}
\end{equation}
(in view of \eqref{pg}, the hyperbolicity condition \eqref{5} implies $\rho >0$).

Let
\begin{equation}
\Gamma (t)=\{ x_1=\varphi (t,x_2)\}
\label{discont}
\end{equation}
be a curve of strong discontinuity for the conservation laws (\ref{1}), i.e., we are interested in solutions of (\ref{1}) that are smooth on either side of $\Gamma (t)$. To be weak solutions of (\ref{1}) such piecewise smooth solutions should satisfy the MHD Rankine-Hugoniot conditions (see, e.g., \cite{LL,Cont1}). According to the classification of strong discontinuities in MHD \cite{LL}, for {\it contact discontinuities} there is no plasma flow across the discontinuity and the magnetic field on both its sides is nowhere tangent to the discontinuity. In view of these requirements, the Rankine-Hugoniot conditions imply the boundary conditions \cite{LL}
\begin{equation}
[p]=0,\quad [v]=0,\quad [H]=0,\quad \partial_t\varphi =v^+_N\quad \mbox{on}\ \Gamma (t),
\label{bcond}
\end{equation}
where $[g]=g^+|_{\Gamma}-g^-|_{\Gamma}$ denotes the jump of $g$, with $g^{\pm}:=g$ in the domains
\[
\Omega^{\pm}(t)=\{\pm (x_1- \varphi (t,x_2))>0\},
\]
and $v^\pm_N=v_1^\pm-v_2^\pm\partial_2\varphi$. Note that the entropy and, hence, the density (see \eqref{pg}) may undergo any jump: $[S]\neq 0$, $[\rho]\neq 0$. Recall also that, according to the definition of contact discontinuities, we have the requirement
\begin{equation}
H^\pm_N\neq 0 \quad \mbox{on}\ \Gamma (t),
\label{mf}
\end{equation}
where $H^\pm_N=H_1^\pm-H_2^\pm\partial_2\varphi$.

As is noted in \cite{Goed}, the boundary conditions \eqref{bcond} are most typical for astrophysical plasmas. Contact discontinuities are usually observed in the solar wind, behind astrophysical shock waves bounding supernova remnants or due to the interaction of multiple shock waves driven by fast coronal mass ejections. However, it is not evident that the formally introduced piecewise smooth solutions to the MHD equations satisfying the boundary conditions \eqref{bcond} must necessarily exist, at least locally in time, for any initial data.  Our final goal is to find conditions on the initial data
\begin{equation}
{U}^{\pm} (0,{x})={U}_0^{\pm}({x}),\quad {x}\in \Omega^{\pm} (0),\quad \varphi (0,{x}_2)=\varphi _0({x}_2),\quad {x}_2\in\mathbb{R},\label{indat}
\end{equation}
providing the existence and uniqueness of a smooth solution $(U^+,U^-,\varphi )$ on some time interval $[0,T]$ to the free boundary problem \eqref{4}, \eqref{bcond}, \eqref{indat}. These conditions will be additional ones to \eqref{5} and \eqref{mf} satisfied at $t=0$.

The study of the linearized problem associated to the nonlinear problem \eqref{4}, \eqref{bcond}, \eqref{indat} is a necessary step towards the proof of the local-in-time existence of MHD contact discontinuities. In \cite{Cont1} we have managed to prove the well-posedness in Sobolev spaces of the linearized variable coefficients problem provided that the {\it Rayleigh-Taylor sign condition}
\begin{equation}
\left[\frac{\partial p}{\partial N}\right]\leq -\epsilon <0
\label{RT}
\end{equation}
on the jump of the normal derivative of the pressure is satisfied at each point of the unperturbed contact discontinuity. It is amazing that the classical condition \eqref{RT} naturally appeared in our energy method as the condition sufficient for the well-posedness of the linearized problem. It is worth noting that, unlike the condition $[\partial q/\partial N] <0$ considered in \cite{Tjde} for the plasma-vacuum interface problem, the magnetic field does not enter \eqref{RT}. That is, for MHD contact discontinuities condition \eqref{RT} appears in its classical (purely hydrodynamical) form as a condition for the pressure $p$ but not for the total pressure $q=p+\frac{1}{2}|{H} |^2$.

The main goal of the present paper is the proof of the local-in-time existence and uniqueness of a smooth solution $(U^+,U^-,\varphi )$ of the original nonlinear free boundary problem \eqref{4}, \eqref{bcond}, \eqref{indat} provided that the initial data \eqref{indat} satisfy the Rayleigh-Taylor sign condition \eqref{RT} together with \eqref{5} and \eqref{mf} (as well as appropriate compatibility conditions).

\subsection{Reduced problem in a fixed domain}

The function $\varphi (t,x_2)$ determining the curve $\Gamma$ of a contact discontinuity is one of the unknowns of
the free boundary problem \eqref{4}, \eqref{bcond}, \eqref{indat}. To reduce this problem to that in a fixed domain we straighten the curve $\Gamma$ by using the same simplest change of independent variables as in \cite{T09,Tcpam}. That is, the unknowns $U^+$ and $U^-$ being smooth in $\Omega^{\pm}(t)$ are replaced by the vector-functions
\begin{equation}
\widetilde{U}^{\pm}(t,x ):= {U}^{\pm}(t,\Phi^{\pm} (t,x),x_2)
\label{change}
\end{equation}
which are smooth in the half-plane $\mathbb{R}^2_+=\{x_1>0,\ x_2\in \mathbb{R}\}$,
where
\begin{equation}
\Phi^{\pm}(t,x ):= \pm x_1+\Psi^{\pm}(t,x ),\quad \Psi^{\pm}(t,x ):= \chi (\pm x_1)\varphi (t,x_2),
\label{change2}
\end{equation}
and $\chi\in C^{\infty}_0(\mathbb{R})$ equals to 1 on $[-1,1]$, and $\|\chi'\|_{L_{\infty}(\mathbb{R})}<1/2$. Here, as in \cite{Met}, we use the cut-off function $\chi$ to avoid assumptions about compact support of the initial data in our existence theorem. The change of variables \eqref{change} is admissible if $\partial_1\Phi^{\pm}\neq 0$. The latter is guaranteed, namely, the inequalities $\partial_1\Phi^+> 0$ and $\partial_1\Phi^-< 0$ are fulfilled, if we consider solutions for which
\begin{equation}
\|\varphi\|_{L_{\infty}([0,T]\times\mathbb{R})}\leq 1.
\label{fi}
\end{equation}
This holds if, without loss of generality, we consider the initial data satisfying $\|\varphi_0\|_{L_{\infty}(\mathbb{R})}\leq 1/2$, and the time $T$ in our existence theorem is sufficiently small.

Dropping for convenience tildes in $\widetilde{U}^{\pm}$, we reduce \eqref{4}, \eqref{bcond}, \eqref{indat} to the initial-boundary value problem
\begin{equation}
A_0(U^\pm )\partial_tU^\pm +\widetilde{A}_1(U^\pm ,\Psi^\pm )\partial_1U^\pm +A_2(U^\pm )\partial_2U^\pm =0\quad\mbox{in}\ [0,T]\times \mathbb{R}^2_+,\label{11}
\end{equation}
\begin{equation}
[p]=0,\quad [v]=0,\quad [H]=0,\quad \partial_t\varphi =v^+_N\quad\mbox{on}\ [0,T]\times\{x_1=0\}\times\mathbb{R},\label{12}
\end{equation}
\begin{equation}
U^+|_{t=0}=U^+_0,\quad U^-|_{t=0}=U^-_0\quad\mbox{in}\ \mathbb{R}^2_+,
\qquad \varphi |_{t=0}=\varphi_0\quad \mbox{in}\ \mathbb{R},\label{13}
\end{equation}
where
\[
\widetilde{A}_1(U^{\pm},\Psi^{\pm})=\frac{1}{\partial_1\Phi^{\pm}}\Bigl(
A_1(U^{\pm})-A_0(U^{\pm})\partial_t\Psi^{\pm}-A_2(U^{\pm})\partial_2\Psi^{\pm}\Bigr)
\]
($\partial_1\Phi^{\pm}=\pm 1 +\partial_1\Psi^{\pm}$), and in (\ref{12}) we use the notation $[g]:=g^+|_{x_1=0}-g^-|_{x_1=0}$ for any pair of values $g^+$ and $g^-$.

We are interested in smooth solutions $(U^+,U^-,\varphi )$ of problem \eqref{11}--\eqref{13}, to be exact, we are going to prove their existence under conditions \eqref{5}, \eqref{mf} and \eqref{RT} on the initial data:
\begin{equation}
p^\pm \geq \bar{p} >0,
\label{5.1}
\end{equation}
\begin{equation}
|H_N^\pm|_{x_1=0}|\geq \kappa >0
\label{mf.1}
\end{equation}
($\bar{p}$ and $\kappa$ are positive constants),
\begin{equation}
[\partial_1p]\geq \epsilon >0,
\label{RT1}
\end{equation}
where \eqref{RT1} is the Rayleigh-Taylor sign condition \eqref{RT} written for the straightened discontinuity (with the equation $x_1=0$) and due to the fact that we have transformed the domains $\Omega^\pm (t)$ into the same half-plane $\mathbb{R}^2_+$ (but not into the different half-planes $\mathbb{R}^2_+$ and $\mathbb{R}^2_-$) the jump of a normal derivative is defined as follows:
\begin{equation}
[\partial_1a]:=\partial_1a^+_{|x_1=0}+\partial_1a^-_{|x_1=0}.
\label{norm_jump}
\end{equation}
At the same time, if we want to have smooth solutions belonging to Sobolev spaces, then not $U^+$ and $U^-$ themselves but corresponding functions shifted to some smooth bounded functions $\bar{U}^+$ and $\bar{U}^-$ should belong to Sobolev spaces. Indeed, a function belonging to a Sobolev space and defined on the unbounded domain $\mathbb{R}^2_+$ must vanish at infinity, but conditions \eqref{5.1}, \eqref{mf.1} and \eqref{RT1} cannot be satisfied for $U^+$ and $U^-$ vanishing at infinity.

As the shifts $\bar{U}^\pm$  one can consider magnetohydrostatic (MHS) equilibria
\[
\bar{U}^\pm (x) =(\bar{p}^\pm (x),0,\bar{H}^\pm (x), \bar{S}^\pm (x)),
\]
which are smooth bounded solutions of the systems $\nabla^\pm \bar q^\pm =(\bar H^\pm\cdot \nabla^\pm )\bar H^\pm$, with $\nabla^\pm:= (\pm\partial_1,\partial_2)$, satisfying the physical condition \eqref{5.1} in the whole half-plane $\overline{\mathbb{R}^2_+}$ and conditions \eqref{mf.1} and \eqref{RT1} on its boundary $x_1=0$. Moreover, $\bar{p}^+(0,x_2)=\bar{p}^-(0,x_2)$,  $\bar{H}^+(0,x_2)=\bar{H}^-(0,x_2)$, and $\bar{S}^\pm (x)$ are arbitrary functions, with $\bar{S}^+(0,x_2)\neq \bar{S}^-(0,x_2)$ (the last condition, i.e., $[\bar{S}]\neq 0$ implies $[\bar{\rho}]\neq 0$). The problem of existence of such equilibria needs a separate study. Some results on the existence of general 2D MHS equilibria in the half-plane (when \eqref{RT1} is not assumed to be satisfied) can be found in \cite{AA}  and references therein where the cases with and without gravity were studied. However, if gravity is taken into account, in Appendix A we present rather simple MHS equilibria satisfying conditions \eqref{5.1}, \eqref{mf.1} and \eqref{RT1} when the contact discontinuity is located between two perfectly conducting rigid walls.

Alternatively, instead of the unbounded curve \eqref{discont} of contact discontinuity we could consider a closed curve $F(t,x)=0$ without self-intersections. In this case there appears no problem with the satisfaction of conditions \eqref{mf.1} and \eqref{RT1} by solutions belonging to Sobolev spaces in unbounded domains. However, to avoid using local coordinate charts necessary for such a geometry, and for the sake of simplicity, as in \cite{ST}, we can still consider the free boundary in the form of a graph \eqref{discont}, but we pose periodic boundary conditions in the tangential direction. More precisely, let
\[
D=\{x\in\mathbb{R}^2\,|\, x_1\in\mathbb{R},\ x_2\in \mathbb{T}\}
\]
be the original space domain occupied by plasma, where $\mathbb{T}$ denotes the 1-torus (the unit circle), which can be thought as the unit segment with periodic boundary conditions. We also set
\[
\Gamma (t)=\{x\in \mathbb{R}\times \mathbb{T},\ x_1=\varphi (t,x_2)\},\qquad t\in [0,T].
\]
Then we still make the change of variables \eqref{change} which reduces our free boundary problem to that in the fixed domain
\[
\Omega=\{x_1>0,\ x_2\in\mathbb{T}\}
\]
with the straightened contact discontinuity
\[
\partial\Omega=\{x_1=0,\ x_2\in\mathbb{T}\}.
\]

For the sake of simplicity, in this paper we prefer to consider periodic boundary conditions in the $x_2$-direction. Since $\Omega$ is an unbounded domain, for satisfying the hyperbolicity condition \eqref{5.1} as $x_1\rightarrow \infty$ we make the change of unknowns
\begin{equation}
\breve{U}^\pm =U^\pm -\bar{U}^\pm ,
\label{shift}
\end{equation}
where $\bar{U}^\pm =(\bar{p},0,0,\bar{S}^\pm )$, the constants $\bar{S}^\pm$ are such that $\bar{S}^+\neq\bar{S}^-$, and $\bar{p}$ is a positive constant from \eqref{5.1}. Under the change of unknowns \eqref{shift} the boundary conditions \eqref{12} stay unchanged whereas in the MHD systems \eqref{11} we should make the shift of the arguments $U^\pm$ of the matrix functions by the constant vectors $\bar{U}^\pm$. Dropping for convenience the breve accents in $\breve{U}^\pm$, we get the following initial-boundary value problem in the space-time domain $[0,T]\times\Omega$:
\begin{equation}
\mathbb{L}(U^+,\Psi^+)=0,\quad \mathbb{L}(U^-,\Psi^-)=0\quad\mbox{in}\ [0,T]\times \Omega,\label{11.1}
\end{equation}
\begin{equation}
\mathbb{B}(U^+,U^-,\varphi )=0\quad\mbox{on}\ [0,T]\times\partial\Omega,\label{12.1}
\end{equation}
\begin{equation}
U^+|_{t=0}=U^+_0,\quad U^-|_{t=0}=U^-_0\quad\mbox{in}\ \Omega,
\qquad \varphi |_{t=0}=\varphi_0\quad \mbox{on}\ \partial\Omega,\label{13.1}
\end{equation}
where $\mathbb{L}(U^\pm,\Psi^\pm)=L(U^\pm,\Psi^\pm)U^\pm$,
\[
L(U^\pm,\Psi^\pm)=A_0(U^\pm +\bar{U}^\pm)\partial_t +\widetilde{A}_1(U^\pm +\bar{U}^\pm,\Psi^\pm)\partial_1+A_2(U^\pm  +\bar{U}^\pm)\partial_2,
\]
and \eqref{12.1} is the compact form of the boundary conditions
\begin{equation}
[p]=0,\quad [v]=0,\quad [H_{\tau}]=0,\quad \partial_t\varphi-v_{N}^+|_{x_1=0}=0,
\label{12'}
\end{equation}
with ${H}^{\pm}_{\tau}= H_1^\pm \partial_2\Psi^{\pm}+H_2^{\pm}$. Moreover, in view of \eqref{shift}, the hyperbolicity conditions for systems \eqref{11.1} can be now written, for example, as  (cf. \eqref{5.1})
\begin{equation}
p^\pm > -\bar{p}/4
\label{5.1'}
\end{equation}
whereas conditions \eqref{mf.1} and \eqref{RT1} stay unchanged for the new (shifted) unknowns.

Note that the continuity of the magnetic field $[H]=0$ is equivalent to $[H_N]=0$ and $[H_\tau ]=0$. However, as was proved in \cite{Cont1}, the condition $[H_N]=0$ coming from the constraint equation \eqref{2} is not a real boundary condition and must be regarded as a restriction (boundary constraint) on the initial data \eqref{13.1} (and this is why we did not include it into \eqref{12.1}). More precisely, we have the following proposition.

\begin{proposition}[\cite{Cont1}]
Let the initial data \eqref{13.1} satisfy
\begin{equation}
\div h^+=0,\quad \div h^-=0
\label{14}	
\end{equation}
and the boundary condition
\begin{equation}
[H_{N}]=0,
\label{15}
\end{equation}
where $h^{\pm}=(H_{N}^{\pm},H_2^{\pm}\partial_1\Phi^{\pm})$. If problem \eqref{11.1}--\eqref{13.1} has a sufficiently smooth solution, then this solution satisfies \eqref{14} and \eqref{15} for all $t\in [0,T]$.
\label{p1}
\end{proposition}

Equations \eqref{14} are just the constraint equation \eqref{2} written in  the straightened variables. Using \eqref{14}, one can show that system (\ref{1}) in the straightened variables is equivalent to \eqref{11.1}. Clearly, a counterpart of Proposition \ref{p1} is true for solutions of system \eqref{1} with a curve of contact discontinuity.

As was noted in \cite{Cont1}, to prove the existence of solutions to problem \eqref{11.1}--\eqref{13.1} we need to know a certain a priori information about these solutions. This information is contained in the following proposition.

\begin{proposition}[\cite{Cont1}]
Assume that problem \eqref{11.1}--\eqref{13.1} (with the initial data satisfying \eqref{14} and \eqref{15}) has a sufficiently smooth solution $(U^+, U^-,\varphi)$ on a time interval $[0,T]$. Then the normal derivatives $\partial_1U^{\pm}$ satisfy the jump condition
\begin{equation}
[\partial_1v]=0
\label{jc2}
\end{equation}
for all $t\in [0,T]$ (recall that the jump of a normal derivative is defined in \eqref{norm_jump}).
\label{p2}
\end{proposition}

\begin{remark}
{\rm
Strictly speaking, Propositions \ref{p1} and \ref{p2} were proved in \cite{Cont1} for problem \eqref{11}--\eqref{13} in the half-plane $\mathbb{R}^2_+$. But, clearly, these propositions as well as all of the results obtained in \cite{Cont1} for the linearized problem stay valid for problem \eqref{11.1}--\eqref{13.1} and its linearization. In particular, while obtaining energy equalities for the linearized problem by standard arguments of the energy method for symmetric hyperbolic systems, the boundary integrals over the opposite sides of the semi-strip $\{x_1>0\}\times\mathbb{S}$ (where $\mathbb{S}$ is a unit $x_2$-segment) cancel out due to periodic boundary conditions in the $x_2$-direction.  Below we will not comment the role of periodic boundary conditions anymore.
}
\label{r1}
\end{remark}

\subsection{Main result and discussion}

We are now in a position to state the main result of the present paper that is the local-in-time existence theorem for problem \eqref{11.1}--\eqref{13.1}. Clearly, this theorem implies a corresponding theorem for the original free boundary problem \eqref{4}, \eqref{bcond}, \eqref{indat}.

\begin{theorem} Let $m\in\mathbb{N}$ and $m\geq 6$. Suppose the initial data \eqref{13.1}, with
\[
\left((U^+_0,U^-_0),\varphi_0\right)\in H^{m+17/2}(\Omega)\times H^{m+17/2}(\partial\Omega),
\]
satisfy the hyperbolicity condition \eqref{5.1'} and the divergence constraints \eqref{14} for all $x\in\Omega$. Let the initial data satisfy requirement \eqref{mf.1}, the Rayleigh-Taylor sign condition \eqref{RT1} and the boundary constraint \eqref{15} for all $x\in \partial\Omega$. Assume also that the initial data are compatible up to order $m+8$ in the sense of Definition \ref{d1}. Then there exists a sufficiently short time $T>0$ such that problem \eqref{11.1}--\eqref{13.1} has a unique solution
\[
\left((U^+,U^-),\varphi\right)\in H^{m}([0,T]\times\Omega)\times H^{m}([0,T]\times\partial\Omega).
\]
\label{t1}
\end{theorem}

As usual, we will construct solutions to the nonlinear problem by considering a sequence of
linearized problems. However, since for MHD contact discontinuties the Kreiss-Lopatinski condition is satisfied only in a weak sense (see \cite{Cont1}), there appears a loss of derivatives phenomena and, therefore, the standard fixed-point argument is inapplicable for our case. As, for example, in \cite{Al,CS,ST,T09,Tcpam}, we overcome this principal difficulty by solving our nonlinear problem by a suitable Nash-Moser-type iteration scheme (see, e.g., \cite{Herm,Sec16} and references therein).

The main tool for proving the convergence of the Nash-Moser iteration scheme is a so-called tame estimate for the linearized problem (see Section \ref{s3} and, e.g.,  \cite{Al,Sec16}). In \cite{Cont1}, the  basic a priori estimate in $H^1$ for the linearized problem was obtained by the energy method provided that the Rayleigh-Taylor sign condition \eqref{RT} is satisfied at each point of the unperturbed discontinuity. This estimate is a basis for deriving the tame estimate in $H^s$  (see Section \ref{s3}) and implies \textit{uniqueness} of a solution to the nonlinear problem \eqref{11.1}--\eqref{13.1} that can be proved by standard argument (see, e.g., \cite{Ticvs}). With this short remark, we shall no longer discuss the problem of uniqueness in this paper.

As was noted in \cite{Cont1}, the  basic a priori estimate obtained there for the linearized problem for the 2D planar case cannot be directly extended to the 3D case because of a principal difficulty connected with the appearance of additional boundary terms in energy integrals. That is, the general 3D case is still an open problem for MHD contact discontinuities.

The classical condition \eqref{RT} is sufficient for the local-in-time existence of MHD contact discontinuities. Our hypothesis is that it is also necessary for the well-posedness of problem \eqref{11.1}--\eqref{13.1} and its violation leads to ill-posedness associated with Rayleigh-Taylor instability. The proof of this hypothesis is also an interesting open problem for future research. In this connection, it is worth noting that the Rayleigh-Taylor instability of contact discontinuities was earlier detected in numerical MHD simulations of astrophysical plasmas as fingers near the contact discontinuity in the contour maps of density (see \cite{FangZhang} and references therein).

The plan of the rest of the paper is the following. In Section \ref{s2}, we formulate the linearized problem and recall the well-posedness result for it obtained in \cite{Cont1}. In Section \ref{s3}, for the linearized problem we derive the a priori tame estimate mentioned above. In Section \ref{s4}, we specify compatibility conditions for the initial data and, by constructing an approximate solution, reduce problem \eqref{11.1}--\eqref{13.1} to that with zero initial data. At last, in Section \ref{s5} we solve the reduced problem by a suitable Nash-Moser-type iteration scheme.

\section{Linearized problem associated to (\ref{11.1})--(\ref{13.1})}
\label{s2}

\setcounter{subsubsection}{0}
\subsection{The basic state}

Consider
\begin{equation}
\begin{split}
& \Omega_T:= (-\infty, T]\times\Omega,\quad \partial\Omega_T:=(-\infty ,T]\times\partial\Omega ,\\
 & \Omega_T^+:= [0, T]\times\Omega,\quad \partial\Omega_T^+:=[0 ,T]\times\partial\Omega .
\end{split}
\label{OmegaT}
\end{equation}
Let the basic state
\begin{equation}
(\widehat{U}^+(t,x ),\widehat{U}^-(t,x ),\hat{\varphi}(t,{x}'))
\label{a21}
\end{equation}
upon which we perform linearization be a given sufficiently smooth vector-function with $\widehat{U}^{\pm}=(\hat{p}^{\pm},\hat{v}^{\pm},\widehat{H}^{\pm},\widehat{S}^{\pm})$ and
\begin{equation}
\|\widehat{U}^+\|_{W^2_{\infty}(\Omega_T)}+
\|\widehat{U}^-\|_{W^2_{\infty}(\Omega_T)}
+\|\hat{\varphi}\|_{W^2_{\infty}(\partial\Omega_T)} \leq K,
\label{a22}
\end{equation}
where $K>0$ is a constant, and below we will also use the notations
\[
\widehat{\Phi}^{\pm}(t,x )=\pm x_1 +\widehat{\Psi}^{\pm}(t,x ),\quad
\widehat{\Psi}^{\pm}(t,x )=\chi(\pm x_1)\hat{\varphi}(t,x'),
\]
i.e., all of the ``hat'' values are determined like corresponding values for $(U^+,U^-, \varphi)$, e.g.,
\[
\hat{v}_{N}^{\pm}=\hat{v}_1^{\pm}- \hat{v}_2^{\pm}\partial_2\widehat{\Psi}^{\pm},\quad
\widehat{H}_{N}^{\pm}=\widehat{H}_1^{\pm}- \widehat{H}_2^{\pm}\partial_2\widehat{\Psi}^{\pm},\quad
\widehat{H}^{\pm}_{\tau}= \widehat H_1^\pm \partial_2\widehat\Psi^{\pm}+\widehat H^{\pm}_2.
\]
Moreover, without loss of generality we assume that $\|\hat{\varphi}\|_{L_{\infty}(\partial\Omega_T)}<1$ (see \eqref{fi}). This implies
$\partial_1\widehat{\Phi}^+\geq 1/2$ and $\partial_1\widehat{\Phi}^-\leq -  1/2$.

\begin{remark}
{\rm
In \cite{Cont1}, unlike \eqref{a22}, we assumed that the norm $\|\hat{\varphi}\|_{W^3_{\infty}(\partial\Omega_T)}$ should be bounded. In fact, this assumption in \cite{Cont1} could be relaxed and replaced with that in \eqref{a22} because the third-order derivatives of $\widehat{\Psi}^{\pm}$ whose boundedness was necessary for the derivation of an priori estimate for the linearized problem had the form $\partial_1\partial_t^{\alpha_0}\partial_2^{\alpha_2}\widehat{\Psi}^{\pm}$ (with $\alpha_0+\alpha_2=2$). But,
$\partial_1\partial_t^{\alpha_0}\partial_2^{\alpha_2}\widehat{\Psi}^{\pm}=\pm \chi'\partial_t^{\alpha_0}\partial_2^{\alpha_2}\hat{\varphi}$ and the boundedness of the norm $\|\hat{\varphi}\|_{W^2_{\infty}(\partial\Omega_T)}$ was enough.
}
\label{r2}
\end{remark}

We assume that the basic state defined in ${\Omega_T}$ satisfies the ``relaxed'' hyperbolicity conditions \eqref{5.1'},
\begin{equation}
\hat{p}^{\pm} \geq -\bar{p}/2\quad \mbox{in}\ \Omega_T^+,  \label{a5}
\end{equation}
the boundary conditions \eqref{12'} together with the boundary constraint \eqref{15},
\begin{equation}
[\hat{p}]=0,\quad [\hat{v}]=0,\quad [\widehat{H}]=0, \quad \partial_t\hat{\varphi}-\hat{v}_{N}^+|_{x_1=0}=0\quad \mbox{on}\ \partial\Omega_T,
\label{a12'}
\end{equation}
the ``relaxed'' requirement \eqref{mf.1},
\begin{equation}
|\widehat{H}_N^{\pm}|_{x_1=0}|\geq {\kappa}/2 >0\quad \mbox{on}\ \partial\Omega_T^+,
\label{cdass}
\end{equation}
and the jump condition \eqref{jc2},
\begin{equation}
[\partial_1\hat{v}]=0\quad \mbox{on}\ \partial\Omega_T.
\label{jc1'}
\end{equation}

% $\hat{w}^{\pm}=\hat{u}^{\pm}-(\partial_t\widehat{\Psi}^{\pm},0)$ and

% $\hat{u}^{\pm}=(\hat{v}_N^{\pm},\hat{v}_2^{\pm}\partial_1\widehat{\Phi}^{\pm})$.

\begin{remark}
{\rm
In \cite{Cont1}, for the linearized problem equations for the perturbations of the magnetic fields $H^{\pm}$ associated to  \eqref{14} and \eqref{15} were deduced. Exactly as in \cite{T09}, to do this it is not enough that \eqref{14} and \eqref{15} for the basic state hold.  We need also that the equations for the unperturbed magnetic fields $\widehat{H}^{\pm}$ themselves
are fulfilled, i.e., the fourth and fifth equations of systems in \eqref{11.1} must be assumed to be satisfied for the basic state.
In this paper, unlike \cite{Cont1,T09}, we do not deduce the linear equations associated to \eqref{14}
but just directly obtain estimates for the linearized divergences $\div h^{\pm}$. Moreover, we {\it do not use} the linearized version of \eqref{15} (actually, it was also not used in \cite{Cont1}).
Therefore, here, unlike \cite{Cont1}, we do not assume that the basic state satisfies the fourth and fifth equations of systems in \eqref{11.1} as well as the divergence constraints \eqref{14} (following from \eqref{11.1} under their fulfillment at $t=0$). Note also that assumptions \eqref{a5}--\eqref{jc1'} are constraints on the basic state which are automatically satisfied if the basic state is an exact local-in-time solution of problem \eqref{11.1}--\eqref{13.1}. We used these assumptions in \cite{Cont1} and will really need them while deriving the tame a priori estimate in Section \ref{s3}, which is the main tool for proving the convergence of the Nash-Moser iteration scheme.
}
\label{r3}
\end{remark}

\subsection{The linearized equations}

The linearized equations for \eqref{11.1}, \eqref{12.1} read:
\[
\mathbb{L}'(\widehat{U}^{\pm},\widehat{\Psi}^{\pm})(\delta U^{\pm},\delta\Psi^{\pm}):=
\frac{d}{d\varepsilon}\mathbb{L}(U_{\varepsilon}^{\pm},\Psi_{\varepsilon}^{\pm})|_{\varepsilon =0}={f}^{\pm}
\quad \mbox{in}\ \Omega_T,
\]
\[
\mathbb{B}'(\widehat{U}^+,\widehat{U}^-,\hat{\varphi})(\delta U^+,\delta U^-,\delta \varphi):=
\frac{d}{d\varepsilon}\mathbb{B}(U_{\varepsilon}^+,U_{\varepsilon}^-,\varphi_{\varepsilon})|_{\varepsilon =0}={g}
\quad \mbox{on}\ \partial\Omega_T
\]
where $U_{\varepsilon}^{\pm}=\widehat{U}^{\pm}+ \varepsilon\,\delta U^{\pm}$,
$\varphi_{\varepsilon}=\hat{\varphi}+ \varepsilon\,\delta \varphi$, and
\[
\Psi_{\varepsilon}^{\pm}(t,{x} ):=\chi (\pm x_1)\varphi _{\varepsilon}(t,{x}'),\quad
\Phi_{\varepsilon}^{\pm}(t,{x} ):=\pm x_1+\Psi_{\varepsilon}^{\pm}(t,{x} ),
\]
\[
\delta\Psi^{\pm}(t,{x} ):=\chi (\pm x_1)\delta \varphi (t,{x} ).
\]
Here we introduce the source terms
\[
{f}^{\pm}(t,{x} )=(f_1^{\pm}(t,{x} ),\ldots ,f_6^{\pm}(t,{x} ))\quad \mbox{and}\quad {g}(t,{x}' )=(g_1(t,{x}' ),\ldots ,g_5(t,{x}' ))
\]
to make the interior equations and the boundary conditions inhomogeneous.

We easily compute the exact form of the linearized equations (below we drop $\delta$):
\[
\mathbb{L}'(\widehat{{U}}^{\pm},\widehat{\Psi}^{\pm})({U}^{\pm},\Psi^{\pm})\\
%\]
%\[
=
L(\widehat{{U}}^{\pm},\widehat{\Psi}^{\pm}){U}^{\pm} +{\cal C}(\widehat{{U}}^{\pm},\widehat{\Psi}^{\pm})
{U}^{\pm} -  \bigl\{L(\widehat{{U}}^{\pm},\widehat{\Psi}^{\pm})\Psi^{\pm}\bigr\}\frac{\partial_1\widehat{U}{\pm}}{\partial_1\widehat{\Phi}^{\pm}},
\]
\[
\mathbb{B}'(\widehat{{U}}^+,\widehat{{U}}^-,\hat{\varphi})({U}^+,{U}^-,\varphi)=
\left(
\begin{array}{c}
p^+-p^-\\[3pt]
v^+-v^-\\[3pt]
H_{\tau}^+-H_{\tau}^-\\[3pt]
\partial_t\varphi +\hat{v}_2^+\partial_2\varphi  -v_{N}^+
\end{array}
\right),
\]
where $v_{N}^{\pm}=v_1^{\pm}-v_2^{\pm}\partial_2\widehat{\Psi}^\pm$, ${H}^{\pm}_{\tau}= H_1^\pm \partial_{2}\widehat{\Psi}^{\pm}+H^{\pm}_{2}$ and the matrix
${\cal C}(\widehat{{U}}^{\pm},\widehat{\Psi}^{\pm})$ is determined as follows:
\[
{\cal C}(\widehat{{U}}^{\pm},\widehat{\Psi}^{\pm}){Y}
= ({Y} ,\nabla_yA_0(\widehat{{U}}^{\pm} ))\partial_t\widehat{{U}}^{\pm}
 +({Y} ,\nabla_y\widetilde{A}_1(\widehat{\bf U}^{\pm},\widehat{\Psi}^{\pm}))\partial_1\widehat{{U}}^{\pm}
+ ({Y} ,\nabla_yA_2(\widehat{{U}}^{\pm} ))\partial_2\widehat{{U}}^{\pm},
\]
\[
({Y} ,\nabla_y A(\widehat{{U}}^{\pm})):=\sum_{i=1}^6y_i\left.\left(\frac{\partial A ({Y} )}{
\partial y_i}\right|_{{Y} =\widehat{{U}}^{\pm}}\right),\quad {Y} =(y_1,\ldots ,y_6).
\]

The differential operator $\mathbb{L}'(\widehat{U}^{\pm},\widehat{\Psi}^{\pm})$ is a first order operator in
$\Psi^{\pm}$. Following \cite{Al}, we overcome this potential difficulty by introducing the ``good unknown''
\begin{equation}
\dot{U}^\pm:=U^\pm -\frac{\Psi^\pm}{\partial_1\widehat{\Phi}^\pm}\,\partial_1\widehat{U}^\pm .
\label{b23}
\end{equation}
Omitting detailed calculations, we rewrite the linearized interior equations in terms of the new unknowns $\dot{U}^+$ and $\dot{U}^-$ in \eqref{b23}:
\begin{equation}
L(\widehat{U}^{\pm},\widehat{\Psi}^{\pm})\dot{U}^{\pm} +{\cal C}(\widehat{U}^{\pm},\widehat{\Psi}^{\pm})
\dot{U}^{\pm} - \frac{\Psi^{\pm}}{\partial_1\widehat{\Phi}^{\pm}}\,\partial_1\bigl\{\mathbb{L}
(\widehat{U}^{\pm},\widehat{\Psi}^{\pm})\bigr\}={f}^{\pm}.
\label{b24}
\end{equation}
As in \cite{Al,ST,T09,Tcpam}, we drop the zeroth-order terms in $\Psi^+$  and $\Psi^-$ in \eqref{b24} and consider the effective linear operators
\begin{equation}
\begin{split}
\mathbb{L}'_e(\widehat{U}^{\pm},\widehat{\Psi}^{\pm})\dot{U}^{\pm} :=& L(\widehat{U}^{\pm},\widehat{\Psi}^{\pm}) \dot{U}^{\pm}+\mathcal{C}(\widehat{U}^{\pm},\widehat{\Psi}^{\pm})\dot{U}^{\pm}
\\
= & A_0(\widehat{U}^{\pm}+\bar{U}^{\pm})\partial_t\dot{U}^{\pm}
+\widetilde{A}_1(\widehat{U}^{\pm}+\bar{U}^{\pm},\widehat{\Psi}^{\pm})
\partial_1\dot{U}^{\pm}\\
 &+A_2(\widehat{U}^{\pm}+\bar{U}^{\pm})\partial_2\dot{U}^{\pm}+\mathcal{C}(\widehat{U}^{\pm},\widehat{\Psi}^{\pm})\dot{U}^{\pm}.
\end{split}
\label{24}
\end{equation}
In the subsequent nonlinear analysis the dropped terms in \eqref{b24} will be considered as additional error terms at each Nash-Moser iteration step (see Section \ref{s5}).

Regarding the boundary differential operator $\mathbb{B}'$, in terms of unknowns \eqref{b23} it reads
\begin{equation}
\mathbb{B}'_e(\widehat{U},\hat{\varphi})(\dot{U},\varphi ):=
\mathbb{B}'(\widehat{U},\hat{\varphi})(U^+,U^-,\varphi )=
\left(
\begin{array}{c}
\dot{p}^+-\dot{p}^- + \varphi [\partial_1\hat{p}]\\[3pt]
\dot{v}_1^+-\dot{v}_1^-\\[3pt]
\dot{v}_2^+-\dot{v}_2^-\\[3pt]
\dot{H}_{\tau}^+-\dot{H}_{\tau}^-+ \varphi [\partial_1\widehat{H}_{\tau}]\\[3pt]
\partial_t\varphi +\hat{v}_2^+\partial_2\varphi -\dot{v}_{N}^+ - \varphi \partial_1\hat{v}_N^+
\end{array}
\right),
\label{25}
\end{equation}
where
\[
\widehat{U}=(\widehat{U}^+,\widehat{U}^-),\quad \dot{U}=(\dot{U}^+,\dot{U}^-),
\quad \dot{v}_{N}^{\pm}=
\dot{v}_1^{\pm}-\dot{v}_2^{\pm}\partial_2\widehat{\Psi}^{\pm}, \quad
\dot{H}_{\tau}^{\pm}=\dot{H}_1^{\pm}\partial_2\widehat{\Psi}^{\pm} +\dot{H}_2^{\pm}.
\]
We used the important condition $[\partial_1\hat{v}]=0$ for the basic state (see \eqref{jc1'}) while writing down the second and third lines in the boundary operator in \eqref{25}.
Introducing the notation
\begin{equation}
\mathbb{L}'_e(\widehat{U},{\widehat{\Psi}})\dot{U}:=
\left(
\begin{array}{c}
\mathbb{L}'_e(\widehat{U}^+,\widehat{\Psi}^+)\dot{U}^+\\[3pt]
\mathbb{L}'_e(\widehat{U}^-,\widehat{\Psi}^-)\dot{U}^-
\end{array}
\right),
\label{27}
\end{equation}
with $\widehat{\Psi}=(\widehat{\Psi}^+,\widehat{\Psi}^-)$, we write down the linear problem for $(\dot{U},\varphi)$:
\begin{align}
\mathbb{L}'_e(\widehat{U},\widehat{{\Psi}})\dot{U}={f}\quad &\mbox{in}\ \Omega_T, \label{28}\\[3pt]
\mathbb{B}'_e(\widehat{U},\hat{\varphi})(\dot{U},\varphi)={g}\quad &\mbox{on}\ \partial\Omega_T,\label{29}\\[3pt]
(\dot{U},\varphi)=0\quad &\mbox{for}\ t<0,\label{30}
\end{align}
where $f=(f^+,f^-)$.

We are now in a position to recall the main result of \cite{Cont1} which is the well-posedness of the linearized problem \eqref{28}--\eqref{30} under the Rayleigh-Taylor sign condition satisfied for the basic state.

\begin{theorem}[\cite{Cont1}]
Let assumptions \eqref{a22}--\eqref{jc1'} be fulfilled for the basic state \eqref{a21}. Let also the basic state satisfies   the Rayleigh-Taylor sign condition \eqref{RT1}:
\begin{equation}
[\partial_1\hat{p} ]\geq \epsilon /2 >0\quad \mbox{on}\ \partial\Omega_T^+,
\label{RTL}
\end{equation}
where $[\partial_1\hat{p} ]=\partial_1\hat{p}^+_{|x_1=0} +\partial_1\hat{p}^-_{|x_1=0}$ (see \eqref{norm_jump}). Then, for all $f \in H^1(\Omega_T)$ and $g\in H^{3/2}(\partial\Omega_T)$ which vanish in the past, problem \eqref{28}--\eqref{30} has a unique solution $(\dot{U},\varphi )\in H^1(\Omega_T)\times H^1(\partial\Omega_T)$. Moreover, this solution obeys the a priori estimate
\begin{equation}
\|\dot{U} \|_{H^{1}(\Omega_T)}+\|\varphi\|_{H^1(\partial\Omega_T)} \leq C\left\{\|f \|_{H^{1}(\Omega_T)}+ \|g\|_{H^{3/2}(\partial\Omega_T)}\right\},
\label{main_estL}
\end{equation}
where $C=C(K,\bar{p},{\kappa},\epsilon,T)>0$ is a constant independent of the data $f$ and $g$.
\label{t1L}
\end{theorem}

Note that for the basic state we write the half of $\epsilon$ in \eqref{RTL} just for technical convenience of subsequent arguments.\footnote{Condition \eqref{RT} being fulfilled for the approximate solution from Section \ref{s4} in the ``relaxed'' form \eqref{RTL} will be satisfied for the so-called modified state from Section \ref{s5} for a sufficiently small time interval.}

\section{Tame estimate}
\label{s3}

\setcounter{subsubsection}{0}
\subsection{Tame a priori estimate for problem (\ref{28})--(\ref{30})}

We are going to derive a tame a priori estimate in $H^s$ for problem \eqref{28}--\eqref{30}, with $s$ large enough.
This tame estimate (see Theorem \ref{t3.1} below) being, roughly speaking, linear in high norms (that are
multiplied by low norms) is with no loss of derivatives from
$f$, with the loss of one derivative from $g$, and with a fixed loss of derivatives with
respect to the coefficients, i.e., with respect to the basic state
(\ref{a21}). Although problem \eqref{28}--\eqref{30} is a hyperbolic problem with characteristic boundary that implies a natural loss of control on derivatives in the normal direction we manage to compensate this loss. This was achieved in \cite{Cont1} for $H^1$, and also here we derive higher-order estimates in usual Sobolev spaces  $H^s$ by estimating missing normal derivatives through equations satisfied by the divergences $\div\dot{h}^{\pm}$ (where $\dot{h}^{\pm}=(\dot{H}_{N}^{\pm},\dot{H}_2^{\pm}\partial_1\widehat{\Phi}^{\pm})$) and by using a ``decoupled'' character of the equations for $\dot{S}^{\pm}$ (they are connected with the rest equations only by lower-order terms).

\begin{theorem}
Let $T>0$ and $s\in \mathbb{N}$, with $s\geq 3$. Assume that the basic state $(\widehat{U} ,\hat{\varphi})\in
H^{s+3}(\Omega_T )\times H^{s+3}(\partial\Omega_T)$ satisfies assumptions \eqref{a22}--\eqref{jc1'},  the Rayleigh-Taylor sign condition \eqref{RTL} and
\begin{equation}
\|\widehat{U}\|_{H^6(\Omega_T )} +\|\hat{\varphi} \|_{H^{6}(\partial\Omega_T)}\leq \widehat{K},
\label{37}
\end{equation}
where $\widehat{K}>0$ is a constant. Let also the data $(f ,g)\in
H^{s}(\Omega_T )\times H^{s+1}(\partial\Omega_T)$ vanish in the past. Then there exists a positive constant $K_0$ that does not depend on $s$ and $T$ and there exists a constant $C(K_0) >0$ such that, if $\widehat{K}\leq K_0$, then there exists a unique solution $(\dot{U} ,\varphi)\in H^{s}(\Omega_T )\times H^{s}(\partial\Omega_T)$ to problem \eqref{28}--\eqref{30} that obeys the tame a priori  estimate
\begin{equation}
\begin{split}
\|\dot{U}\|_{H^s(\Omega_T )}+\|\varphi\|_{H^{s}(\partial\Omega_T)}\leq & C(K_0)\Bigl\{
\|f\|_{H^{s}(\Omega_T )}+ \|g \|_{H^{s+1}(\partial\Omega_T)} \\
 &+\bigl( \|f\|_{H^{3}(\Omega_T )}+ \| g\|_{H^{4}(\partial\Omega_T)} \bigr)\bigl(
\|\widehat{U}\|_{H^{s+3}(\Omega_T )}+\|\hat{\varphi}\|_{H^{s+3}(\partial\Omega_T)}\bigr)\Bigr\}
\end{split}
\label{38}
\end{equation}
for a sufficiently short time $T$ (the constant $C(K_0)$ depends also on the fixed constants $\bar{p}$, ${\kappa}$ and $\epsilon$ from \eqref{5.1}--\eqref{RT1}).
\label{t3.1}
\end{theorem}

\begin{remark}
{\rm
Theorem \ref{t3.1} looks similar to that in \cite{Tcpam} for the free boundary problem for the compressible Euler equations with a vacuum boundary condition. As \cite{Tcpam}, and unlike \cite{Cont1} (cf. \eqref{main_estL}), here we prefer to work with integer indices of Sobolev spaces, i.e., we derive a little bit roughened version of the tame estimate where, in particular, we loose one but not ``half'' derivative from $g$. We do so just for technical convenience because an additional gain of ``half'' derivative in the local-in-time existence theorem is not really principal (e.g., from the physical point of view).
}
\label{r4}
\end{remark}

\subsection{Reduction to homogeneous boundary conditions}

Technically, it is much more convenient to derive first a tame estimate for a reduced linearized problem with homogeneous boundary conditions (with $g=0$) and then get estimate \eqref{38} as its consequence. Using the classical argument, we subtract from the solution a more regular function $\widetilde{U}=(\widetilde{U}^+,\widetilde{U}^-)\in H^{s+1}(\Omega_T)$ satisfying the boundary conditions \eqref{29}. Then, the new unknown
\begin{equation}
U^{\natural}=(U^{+\natural},U^{-\natural})=\dot{U}-\widetilde{U} ,\label{a87'}
\end{equation}
with
\begin{equation}
\|\widetilde{U} \|_{H^{s+1}(\Omega_T)}\leq C\|g \|_{H^{s+1/2}(\partial\Omega_T)},
\label{tildU}
\end{equation}
satisfies problem \eqref{28}--\eqref{30} with $f =F =(F^+,F^-)$, where
\begin{equation}
F^\pm =f^\pm-\widehat{A}_0^{\pm}\partial_t\widetilde{U}^{\pm} -\widehat{A}_1^{\pm}\partial_1\widetilde{U}^{\pm}  -\widehat{A}_2^{\pm}\partial_2\widetilde{U}^{\pm}-\widehat{\mathcal{C}}^{\pm}\widetilde{U}^\pm
\label{a87''}
\end{equation}
and
\[
\widehat{A}^{\pm}_{\alpha}:=A_{\alpha}(\widehat{U}^{\pm}+\bar{U}^{\pm}), \quad \alpha=0,2,\quad \widehat{A}^{\pm}_1:=\widetilde{A}_1(\widehat{U}^{\pm}+\bar{U}^{\pm},\widehat{\Psi}^{\pm}),\quad
\widehat{\mathcal{C}}^{\pm}:=\mathcal{C}(\widehat{U}^{\pm},\widehat{\Psi}^{\pm}).
\]
Moreover, here and later on $C$ is a positive constant that can change from line to line, and it may depend on other constants, in particular, in \eqref{tildU} the constant $C$ depends on $s$ (sometimes, as in \eqref{38}, we show the dependence of $C$ from another constants).
Below we will use the roughened version of \eqref{tildU} (see Remark \ref{r4})
\begin{equation}
\|\widetilde{U} \|_{H^{s+1}(\Omega_T)}\leq C\|g \|_{H^{s+1}(\partial\Omega_T)}.
\label{tildU'}
\end{equation}

Dropping for convenience the indices $^{\natural}$ in \eqref{a87'}, we get our reduced linearized problem:
\begin{align}
\widehat{A}^{\pm}_0\partial_t{U}^{\pm} +\widehat{A}^{\pm}_1\partial_1{U}^{\pm}+\widehat{A}^{\pm}_2\partial_2{U}^{\pm} +\widehat{\mathcal{C}}^{\pm}{U}^{\pm} =F^{\pm}\qquad &\mbox{in}\ \Omega_T,\label{b48b}\\[3pt]
 {[}p{]}=- \varphi {[}\partial_1\hat{p}{]},\qquad & \label{b50b.1}\\[3pt]
  {[}v{]}=0,\qquad & \label{b50b.2}\\[3pt]
  {[}H_{\tau}{]}=-\varphi {[}\partial_1\widehat{H}_{\tau}{]},\qquad & \label{b50b.3} \\[3pt]
{v}_{N}^+ = \partial_0\varphi - \varphi \partial_1\hat{v}_N^+   \qquad &\mbox{on}\ \partial\Omega_T
\label{b50b.4} \\[3pt]
({U},\varphi )=0\qquad &\mbox{for}\ t<0,\label{b51b}
\end{align}
where
\begin{equation}
\partial_0:=  \partial_t +\hat{v}_2^+\partial_2 \qquad \mbox{in}\ \Omega_T.
\label{d0}
\end{equation}
It is not a big mistake to call $\partial_0$ the material derivative because on the boundary $\partial_0$ coincides with the material derivative $\partial_t + (\hat{w}^\pm \cdot\nabla )$ (in the reference frame related to the discontinuity), where
\[
\hat{w}^{\pm}=\hat{u}^{\pm}-(\partial_t\widehat{\Psi}^{\pm},0), \quad \hat{u}^{\pm}=(\hat{v}_N^{\pm},\hat{v}_2^{\pm}\partial_1\widehat{\Phi}^{\pm}).
\]

From now on we concentrate on the proof of a tame estimate for the reduced problem \eqref{b48b}--\eqref{b51b}. Namely, we will prove the following lemma.

\begin{lemma}
Let $T>0$ and $s\in \mathbb{N}$, with $s\geq 3$. Assume that the basic state $(\widehat{U} ,\hat{\varphi})\in
H^{s+3}(\Omega_T )\times H^{s+3}(\partial\Omega_T)$ satisfies assumptions \eqref{a22}--\eqref{jc1'}, the Rayleigh-Taylor sign condition \eqref{RTL} and inequality \eqref{37}.
Let also that $F\in H^{s}(\Omega_T )$ vanishes in the past. Then there exists a positive constant $K_0$ that does not depend on $s$ and $T$ and there exists a constant $C(K_0) >0$ such that, if $\widehat{K}\leq K_0$, then there exists a unique solution $(U ,\varphi)\in H^{s}(\Omega_T )\times H^{s}(\partial\Omega_T)$ to problem \eqref{b48b}--\eqref{b51b} that obeys the tame a priori estimate
\begin{equation}
\begin{split}
\|{U} & \|_{H^s(\Omega_T )}+  \|\varphi\|_{H^{s}(\partial\Omega_T)}\\ &\leq  C(K_0)\Bigl\{
\|F\|_{H^{s}(\Omega_T )} +\|F\|_{H^{3}(\Omega_T )}\bigl(
\|\widehat{U}\|_{H^{s+3}(\Omega_T )}+\|\hat{\varphi}\|_{H^{s+3}(\partial\Omega_T)}\bigr)\Bigr\}
\end{split}
\label{38'}
\end{equation}
for a sufficiently short time $T$ (the constant $C(K_0)$ depends also on the fixed constants $\bar{p}$, ${\kappa}$ and $\epsilon$ from \eqref{5.1}--\eqref{RT1}).
\label{l3.1}
\end{lemma}

Taking into account \eqref{a87'}, \eqref{a87''} and \eqref{tildU'}, we will then show that estimate \eqref{38'} implies \eqref{38}.

\subsection{An equivalent formulation of problem \eqref{b48b}--\eqref{b51b}}

While using the energy method for problem \eqref{b48b}--\eqref{b51b}, for writing down the quadratic forms with the boundary matrices $\widehat{A}_1^{\pm}$ on the boundary $x_1=0$ we will use their exact form \cite{Cont1} (following from the last condition in \eqref{a12'})
\begin{equation}
\widehat{A}_1^{\pm}|_{x_1=0} = \pm \begin{pmatrix} 0 & 1 & -\partial_2\hat{\varphi} & 0& 0 & 0\\[6pt]
1 &0 &0 &\widehat{H}_2^\pm\partial_2\hat{\varphi} & \widehat{H}_2^\pm & 0 \\[6pt]
-\partial_2\hat{\varphi} & 0 & 0 & -\widehat{H}_1^\pm\partial_2\hat{\varphi} & -\widehat{H}_1^\pm & 0 \\[6pt]
0 & \widehat{H}_2^\pm\partial_2\hat{\varphi} & -\widehat{H}_1^\pm\partial_2\hat{\varphi} & 0& 0& 0 \\[6pt]
0& \widehat{H}_2^\pm & -\widehat{H}_1^\pm & 0 & 0 & 0 \\[6pt]
0 & 0 &0 &0 &0 &0
\end{pmatrix}_{|x_1=0}.
\label{A1tilde2D}
\end{equation}
In view of \eqref{cdass}, the matrices $\widehat{A}_1^{\pm}|_{x_1=0}$ are of constant rank 4 and, more precisely, they have two positive, two negative and two zero eigenvalues \cite{Cont1}. That is, \eqref{b48b}--\eqref{b51b} is a hyperbolic problem with \textit{characteristic} boundary of constant multiplicity, and since one of the boundary conditions is needed for determining the function $\varphi $, the correct number of boundary conditions is five that is the case in \eqref{b50b.1}--\eqref{b50b.4}.

Let us set
\begin{equation}
V^{\pm}=(V_1^{\pm},\ldots ,V_6^{\pm})=({p}^{\pm},{v}_N^{\pm},{v}_2^{\pm},H_N^{\pm},H_{\tau}^{\pm},{S}^{\pm}),
\label{nun}
\end{equation}
where
\[
{v}_N^{\pm}={v}_1^{\pm}-{v}_2^{\pm}\partial_2\widehat{\Psi}^{\pm},\quad H_N^{\pm}={H}_1^{\pm}- {H}_2^{\pm}\partial_2\widehat{\Psi}^{\pm},\quad {H}_{\tau}^{\pm}={H}_1^{\pm}\partial_2\widehat{\Psi}^{\pm} +{H}_2^{\pm}.
\]
We have ${U}^{\pm}=J^{\pm}V^{\pm}$, with
\begin{equation}
J^{\pm}=
\begin{pmatrix}
1 &  0 & 0 & 0 &  0 & 0\\
0 &  1 & \partial_2\widehat{\Psi}^{\pm} & 0 &  0 & 0\\
0 &  0 & 1 & 0 &  0 &0\\[3pt]
0 & 0 & 0 & {\displaystyle\frac{1}{1+(\partial_2\widehat{\Psi}^{\pm})^2}} & \displaystyle\frac{\partial_2\widehat{\Psi}^{\pm}}{1+(\partial_2\widehat{\Psi}^{\pm})^2}& 0 \\[12pt]
0 & 0 & 0 & \displaystyle -\frac{\partial_2\widehat{\Psi}^{\pm}}{1+(\partial_2\widehat{\Psi}^{\pm})^2} & \displaystyle\frac{1}{1+(\partial_2\widehat{\Psi}^{\pm})^2}& 0 \\[9pt]
0 & 0 & 0 & 0 & 0& 1
\end{pmatrix}.
\label{J}
\end{equation}
Then, systems \eqref{b48b} are equivalently rewritten as
\begin{equation}
\mathcal{A}_0^{\pm}\partial_t{V}^{\pm}+ \mathcal{A}_1^{\pm}\partial_1{V}^{\pm} + \mathcal{A}_2^{\pm}\partial_2{V}^{\pm}+\mathcal{A}_3^{\pm}{V}^{\pm} =\mathcal{F}^{\pm} \qquad \mbox{in}\ \Omega_T, \label{c29}
\end{equation}
where
\begin{align*}
\mathcal{A}_{\alpha}^{\pm} &=(J^{\pm})^{\textsf{T}}\widehat{A}_{\alpha}^{\pm}J^{\pm}\quad (\alpha =\overline{0,2}),\qquad \mathcal{F}^{\pm}=(J^{\pm})^{\textsf{T}}F^{\pm},
\\
\mathcal{A}_3^{\pm} &=(J^{\pm})^{\textsf{T}}\bigl( \widehat{A}_0^{\pm}\partial_t{J}^{\pm}+ \widehat{A}_1^{\pm}\partial_1{J}^{\pm} + \widehat{A}_2^{\pm}\partial_2{J}^{\pm}+\widehat{\mathcal{C}}J^{\pm}\bigr).
\end{align*}
The boundary matrices $\mathcal{A}_1^{\pm}$ in systems (\ref{c29}) have the form
\begin{equation}
\mathcal{A}_1^{\pm}=\mathcal{A}^{\pm}_{(1)}+\mathcal{A}^{\pm}_{(0)},\quad
\mathcal{A}^{\pm}_{(1)}=\pm\begin{pmatrix}
0 & 1 & 0 & 0 & 0 & 0\\
1 & 0 & 0 & 0 & \widehat{H}_2^{\pm} & 0\\
0 & 0 & 0 & 0 & - \widehat{H}_N^{\pm} & 0\\
0 & 0 & 0 & 0 & 0 & 0\\
0 & \widehat{H}_2^{\pm} & - \widehat{H}_N^{\pm} & 0 & 0 & 0\\
0 & 0 & 0 & 0 & 0 & 0
\end{pmatrix},
\quad \mathcal{A}^{\pm}_{(0)}|_{x_1=0}=0,
\label{c30}
\end{equation}
The explicit form of $\mathcal{A}_{(0)}^{\pm}$ is of no interest. Since $\widehat H_N^{\pm}|_{x_1=0}\neq 0$ (see \eqref{cdass}),
\begin{equation}
V_n^{\pm} = (V_1^{\pm},V_2^{\pm},V_3^{\pm},V_5^{\pm})
\label{V_n}
\end{equation}
are the ``noncharacteristic'' parts of the vectors $V^{\pm}$. In  terms of the components of the vectors $V_n^{\pm}$ the boundary conditions \eqref{b50b.1}--\eqref{b50b.4} are rewritten as
\begin{align}
 {[}V_1{]}=- \varphi {[}\partial_1\hat{p}{]},\qquad & \label{b50b.1'}\\[3pt]
  {[}V_2{]}={[}V_3{]}=0,\qquad & \label{b50b.2'}\\[3pt]
  {[}V_5{]}=-\varphi {[}\partial_1\widehat{H}_{\tau}{]},\qquad & \label{b50b.3'} \\[3pt]
\partial_0\varphi=V_2^+ + \varphi \partial_1\hat{v}_N^+   \qquad &\mbox{on}\ \partial\Omega_T.
\label{b50b.4'}
\end{align}

\subsection{Estimation of tangential derivatives containing the $x_2$-derivative}

Since arguments of the energy method below are quite standard, we will omit detailed calculations. By applying to systems \eqref{c29} the operator $\mathcal{D}_2^{\alpha}:=\partial_t^{\alpha_0}\partial_2^{\alpha_2}$, with $\alpha_2\neq 0$ and $|\alpha |=
|(\alpha_0,\alpha_2)|\leq s$ (i.e., $\mathcal{D}_2^{\alpha}=\partial_2\partial_{\rm tan}^{\beta}$ with $\partial_{\rm tan}^{\beta}:=\partial_t^{\beta_0}\partial_2^{\beta_2}$  and $|\beta |=|(\beta_0,\beta_2)|=|(\alpha_0,\alpha_2-1)|\leq s-1$), one gets
\begin{equation}
\sum\limits_{\pm}\int\limits_{\Omega}(\mathcal{ A}_0^{\pm}\mathcal{D}_2^{\alpha}V^{\pm},\mathcal{D}_2^{\alpha}V^{\pm} ) {d}x + 2\int\limits_{\partial\Omega_t}{Q}_2{d}x_2d\tau=\mathcal{ R},\label{c39}
\end{equation}
where
\begin{align*}
{Q}_2= &-\frac{1}{2}\sum\limits_{\pm}\bigl(\mathcal{A}_1^{\pm}\mathcal{D}_2^{\alpha}V^\pm , \mathcal{D}_2^{\alpha}V^\pm \bigr)\bigr|_{x_1=0} \\ = &
\left.\left\{-\mathcal{D}_2^{\alpha}V_2^+[\mathcal{D}_2^{\alpha}V_1]+
\left(\widehat{H}_N^+\mathcal{D}_2^{\alpha}V_3^+-\widehat{H}_2^+\mathcal{D}_2^{\alpha}V_2^+\right)[\mathcal{D}_2^{\alpha}V_5]\right\}\right|_{x_1=0}  ,
\end{align*}
\begin{align*}
\mathcal{ R}= \sum\limits_{\pm}\int\limits_{\Omega_t}\Bigl(
\Bigl\{ &\left(\partial_t\mathcal{ A}_0^{\pm}+ \partial_1\mathcal{ A}_1^{\pm}+\partial_2\mathcal{ A}_2^{\pm}\right)\mathcal{D}^{\alpha}_2V^{\pm} - 2\left[\mathcal{D}^{\alpha}_2 ,\mathcal{ A}_0^{\pm}\right]\partial_tV^{\pm}
- 2\left[\mathcal{D}^{\alpha}_2 ,\mathcal{ A}_1^{\pm}\right]\partial_1V^{\pm} \\ & - 2\left[\mathcal{D}^{\alpha}_2 ,\mathcal{ A}_2^{\pm}\right]\partial_2V^{\pm}
 - 2\mathcal{D}^{\alpha}_2(\mathcal{ A}_3^{\pm}V^{\pm} )+2\mathcal{D}^{\alpha}_2\mathcal{ F}^{\pm}\Bigr\},\mathcal{D}^{\alpha}_2V^{\pm}\Bigr){d}x {d}\tau
\end{align*}
and we use the notation of commutator: $[a,b]c:=a(bc)-b(ac)$. While writing down the quadratic form ${Q}_2$ we took into account the boundary conditions \eqref{b50b.2'}. Using the Moser-type calculus inequalities
\begin{align}
\| uv\|_{H^s(\Omega_T)}\leq & C\left( \|u\|_{H^s(\Omega_T)}\|v\|_{L_{\infty}(\Omega_T)}
+\|u\|_{L_{\infty}(\Omega_T)}\|v\|_{H^s(\Omega_T)}\right),\label{c40}\\
\| b(u)\|_{H^s(\Omega_T)}\leq & C(M)\|u\|_{H^s(\Omega_T)},\label{c41}
\end{align}
where the function $b$ is a $C^{\infty}$ function of $u$ with $b(0)=0$, and $M$ is such a positive constant that
$\|u\|_{L_{\infty}(\Omega_T)}\leq M$, we estimate the right-hand side in \eqref{c39}:
\begin{equation}
\begin{split}
\mathcal{ R}\leq C(K)\Bigl\{ & \|V\|^2_{H^s(\Omega_t)}+\|F\|^2_{H^s(\Omega_T)}\\ & +\left(\|{U}\|^2_{W^1_{\infty}(\Omega_T)}+\|F\|^2_{L_{\infty}(\Omega_T)}\right)\left( 1+\|{\rm coeff}\|^2_{s+2}\right)\Bigr\},
\label{c42}
\end{split}
\end{equation}
where $V=(V^+,V^-)$, $F=(F^+,F^-)$, and $\|{\rm coeff}\|_{m}:=\|\widehat{U}\|_{H^{m}(\Omega_T)}+\|\hat{\varphi}\|_{H^{m}(\partial\Omega_T)}$. More precisely, here and below we also use the following refinement of \eqref{c40}
\begin{equation}
\sum\limits_{|\mu |+|\nu |=s}\|\partial^{\mu }u\,\partial^{\nu}v\|_{L^2(\Omega_T)}\leq C\left( \|u\|_{H^s(\Omega_T)}\|v\|_{L_{\infty}(\Omega_T)}
+\|u\|_{L_{\infty}(\Omega_T)}\|v\|_{H^s(\Omega_T)}\right)
\label{c40'}
\end{equation}
which implies
\[
\begin{split}
\bigl\|[\partial^{\mu },u]v\bigr\|_{L^2(\Omega_T)} &\leq
C\sum\limits_{|\gamma |+|\nu |=s,\ |\gamma|\neq 0}\|\partial^{\gamma }u\,\partial^{\nu}v\|_{L^2(\Omega_T)} \\ & \leq C\left( \|u\|_{H^s(\Omega_T)}\|v\|_{L_{\infty}(\Omega_T)}
+\|u\|_{W^1_{\infty}(\Omega_T)}\|v\|_{H^{s-1}(\Omega_T)}\right), \qquad |\mu |\leq s
\end{split}
\]
($\partial^{\mu}=\partial_t^{\mu_0}\partial_1^{\mu_1}\partial_2^{\mu_2}$, etc.). Note also that before the usage of \eqref{c41} we decompose the matrix functions $\mathcal{A}_{\beta}^{\pm}$ (with $\beta =\overline{0,3}$) as $\mathcal{A}_{\beta}^{\pm}(u)=\mathcal{B}_{\beta}^{\pm}(u)+\mathcal{C}_{\beta}^{\pm}$, where $\mathcal{B}_{\beta}^{\pm}(0)=0$ and  the matrix $\mathcal{C}_{\beta}^{\pm}$ is a constant matrix.

Using the boundary conditions \eqref{b50b.1'} and \eqref{b50b.3'}, we calculate ${Q}_2$:
\begin{equation}
%\begin{split}
{Q}_2=\left.\left\{[\partial_1\hat{p}]\mathcal{D}_2^{\alpha} \varphi \,\mathcal{D}_2^{\alpha}V_2^+
-[\partial_1\widehat{H}_{\tau}]\mathcal{D}_2^{\alpha} \varphi \left(\widehat{H}_N^+\mathcal{D}_2^{\alpha}V_3^+-\widehat{H}_2^+\mathcal{D}_2^{\alpha}V_2^+\right)
\right\}\right|_{x_1=0} + \widetilde{Q}_2,
%\end{split}
\label{Q_2}
\end{equation}
where $\widetilde{Q}_2$ is, in some sense, a sum of lower-order terms containing commutators appeared after the differentiation of \eqref{b50b.1'} and \eqref{b50b.4'}:
\[
\widetilde{Q}_2=\left.\left\{\mathcal{D}_2^{\alpha}V_2^+\bigl[\mathcal{D}_2^{\alpha},[\partial_1\hat{p}]\bigr]\varphi -\left(\widehat{H}_N^+\mathcal{D}_2^{\alpha}V_3^+-\widehat{H}_2^+\mathcal{D}_2^{\alpha}V_2^+\right)
\bigl[\mathcal{D}_2^{\alpha},[\partial_1\widehat{H}_{\tau}]\bigr]\varphi \right\}\right|_{x_1=0}.
\]
Recalling that $\mathcal{D}_2^{\alpha}=\partial_2\partial_{\rm tan}^{\beta}$ with $|\beta |\leq s-1$ and using the boundary condition \eqref{b50b.4'}, we rewrite \eqref{Q_2} as
\begin{equation}
{Q}_2=
\left.\left\{
[\partial_1\hat{p}]\mathcal{D}_2^{\alpha} \varphi \,\mathcal{D}_2^{\alpha}\left(\partial_0\varphi - \varphi \partial_1\hat{v}_N^+\right)
-[\partial_1\widehat{H}_{\tau}]\mathcal{D}_2^{\alpha} \varphi \,R_{\beta}\right\}\right|_{x_1=0} +\widetilde{Q}_2,
\label{Q_2'}
\end{equation}
where
\begin{equation}
R_{\beta}=\left.\left\{\partial_{\rm tan}^{\beta}R-\bigl[\partial_{\rm tan}^{\beta},\widehat{H}_N^+\bigr]\partial_2V_3^+ +\bigl[\partial_{\rm tan}^{\beta},\widehat{H}_2^+\bigr]\partial_2V_2^+\right\}\right|_{x_1=0},
\label{Rbeta'}
\end{equation}
\begin{equation}
\begin{split}
\widetilde{Q}_2= \partial_2\Bigl\{ &\partial_{\rm tan}^{\beta}V_2^+\bigl[\mathcal{D}_2^{\alpha},[\partial_1\hat{p}]\bigr]\varphi -\left(\widehat{H}_N^+\partial_{\rm tan}^{\beta}V_3^+-\widehat{H}_2^+\partial_{\rm tan}^{\beta}V_2^+\right)
\bigl[\mathcal{D}_2^{\alpha},[\partial_1\widehat{H}_{\tau}]\bigr]\varphi \Bigr\}\Bigr|_{x_1=0}\\
 -\Bigl\{ &\partial_{\rm tan}^{\beta}V_2^+\partial_2\left(\bigl[\mathcal{D}_2^{\alpha},[\partial_1\hat{p}]\bigr]\varphi\right) -\partial_{\rm tan}^{\beta}V_3^+\partial_2\left(\widehat{H}_N^+\bigl[\mathcal{D}_2^{\alpha},[\partial_1\widehat{H}_{\tau}]\bigr]\varphi\right)  \\ &  +\partial_{\rm tan}^{\beta}V_2^+\partial_2\left(\widehat{H}_2^+\bigl[\mathcal{D}_2^{\alpha},[\partial_1\widehat{H}_{\tau}]\bigr]\varphi \right) \Bigr\}\Bigr|_{x_1=0}
\end{split}
\label{Q2tilde}
\end{equation}
and
\[
R=\left.\left\{\widehat{H}_N^+\partial_2V_3^+-\widehat{H}_2^+\partial_2V_2^+\right\}\right|_{x_1=0}.
\]

To treat the quadratic form $Q_2$ we use not only the boundary conditions but also the interior equations considered on the boundary.
Multiplying  the equations for $H^+$ contained in \eqref{b48b} (i.e., the 4th and 5th equations in \eqref{b48b} for the superscript $+$) by the vector $(1,-\partial_2\widehat{\Psi}^+)$, considering the result at $x_1=0$ and taking \eqref{b50b.4} into account, we obtain
\begin{equation}
R = \left.\left\{-\partial_0V_4^+  +\widetilde{R}+F_N^+\right\}\right|_{x_1=0},
\label{eqH}
\end{equation}
where $F_N^+=F_4^+-F_5^+\partial_2\widehat{\Psi}^+$ and $\widetilde{R}$ is a sum of lower-order terms:
\begin{equation}
\widetilde{R}=\left.\left\{\bigl(\partial_2\widehat{H}_2^+\bigr)\,V_2^+ -\bigl(\partial_2\widehat{H}_N^+\bigr)\,V_3^+ -(\partial_2\hat{v}_2^+)\,V_4^+ \right\}\right|_{x_1=0}.
\label{Rlow}
\end{equation}
Recalling the definition of the material derivative $\partial_0$ in \eqref{d0}, from \eqref{Rbeta'} and \eqref{eqH} we obtain
\begin{equation}
\begin{split}
R_{\beta}=\Bigl\{ &-\partial_0\partial_{\rm tan}^{\beta}V_4^+ -\bigl[\partial_{\rm tan}^{\beta},\hat{v}_2^+\bigr]\partial_2V_4^+ +\partial_{\rm tan}^{\beta}\widetilde{R} \\ & -\bigl[\partial_{\rm tan}^{\beta},\widehat{H}_N^+\bigr]\partial_2V_3^+ +\bigl[\partial_{\rm tan}^{\beta},\widehat{H}_2^+\bigr]\partial_2V_2^+ +\partial_{\rm tan}^{\beta}F_N^+ \Bigr\}\Bigr|_{x_1=0}.
\end{split}
\label{Rbeta}
\end{equation}

Taking into account  \eqref{d0}, it follows from \eqref{Q_2'} and  \eqref{Rbeta} that
\begin{equation}
{Q}_2=
\underbrace{\frac{1}{2}\,\partial_t\left\{[\partial_1\hat{p}](\mathcal{D}_2^{\alpha} \varphi )^2\right\}}
+\underline{[\partial_1\widehat{H}_{\tau}](\mathcal{D}_2^{\alpha} \varphi) (\partial_0\partial_{\rm tan}^{\beta}V_4^+)|_{x_1=0} }+P_2,
\label{Q2"}
\end{equation}
where the underbraced term in \eqref{Q2"} is the {\it  most important} one because under the Rayleigh-Taylor sign condition \eqref{RTL} it gives us the control on the $L^2$ norm of $\mathcal{D}_2^{\alpha} \varphi$ (see below); the underlined term in \eqref{Q2"} needs an additional treatment whereas the rest terms collected in the quadratic form $P_2$ either disappear after the integration over the domain $\partial\Omega_t$ (because they have the form $\partial_2\{\cdots\}$) or can be absorbed in the right-hand side of a future energy inequality by using the trace theorem, etc. The structure of terms of $P_2$ rather than their explicit form is important, but here for the reader's convenience we write down $P_2$ explicitly:
\begin{equation}
\begin{split}
P_2=&\widetilde{Q}_2-\frac{1}{2}\,[\partial_t\partial_1\hat{p}](\mathcal{D}_2^{\alpha} \varphi )^2+\frac{1}{2}\,\partial_2\left\{[\partial_1\hat{p}]\hat{v}^+_{2}|_{x_1=0}\,
(\mathcal{D}_2^{\alpha} \varphi )^2 \right\}\\ &
-\frac{1}{2}\,\partial_2\left.\left([\partial_1\hat{p}]\hat{v}^+_{2}\right)\right|_{x_1=0}(\mathcal{D}_2^{\alpha} \varphi )^2
+[\partial_1\hat{p}]\mathcal{D}_2^{\alpha}\varphi \left.\left([\mathcal{D}_2^{\alpha},\hat{v}_2^+]\partial_2\varphi
-\mathcal{D}_2^{\alpha}(\varphi\partial_1\hat{v}_N^+)\right)\right|_{x_1=0}+\\
& +[\partial_1\widehat{H}_{\tau}]\mathcal{D}_2^{\alpha} \varphi \Bigl( \bigl[\partial_{\rm tan}^{\beta},\hat{v}_2^+\bigr]\partial_2V_4^+ -\partial_{\rm tan}^{\beta}\widetilde{R} +\bigl[\partial_{\rm tan}^{\beta},\widehat{H}_N^+\bigr]\partial_2V_3^+ \\ & \qquad\qquad\qquad\quad -\bigl[\partial_{\rm tan}^{\beta},\widehat{H}_2^+\bigr]\partial_2V_2^+ -\partial_{\rm tan}^{\beta}F_N^+ \Bigr)\Bigr|_{x_1=0},
\end{split}
\label{P2}
\end{equation}
where $\widetilde{Q}_2$ and $\widetilde{R}$ are given by \eqref{Q2tilde} and \eqref{Rlow} respectively.

We first estimate the integral of $P_2$ containing, in some sense, lower-order terms. As an example, we estimate the last term in the right-hand side of \eqref{Q2tilde} which is just a typical one giving a biggest loss of derivatives from the coefficients in the final a priori estimate \eqref{38'}:
\[
\begin{split}
\int\limits_{\partial\Omega_t} &
\partial_{\rm tan}^{\beta}   V_{2|x_1=0}^+\,\partial_2\left(\widehat{H}_{2|x_1=0}^+\bigl[\mathcal{D}_2^{\alpha},[\partial_1\widehat{H}_{\tau}]\bigr]\varphi \right) {d}x_2d\tau  \\ & \leq \|V_{|x_1=0}\|^2_{H^{s-1}(\partial\Omega_t)}+C(K)\sum_{|\omega|=1}\sum_{|\mu|+|\nu|\leq s }
\left\|\partial_{\rm tan}^{\mu}\bigl( \partial_{\rm tan}^{\omega}[\partial_1\widehat{H}_{\tau}]\bigr)\,\partial_{\rm tan}^{\nu}\varphi\right\|^2_{L^2(\partial\Omega_t)}\\
& \leq C(K)\left\{ \|V\|^2_{H^{s}(\Omega_t)}+\|\varphi\|^2_{H^s(\partial\Omega_t)}+\|\varphi\|^2_{L_{\infty}(\partial\Omega_T)}\left(
\|\hat{\varphi}\|^2_{H^{s+2}(\partial\Omega_T)}
+\|\partial_1\widehat{U}|_{x_1=0}\|^2_{H^{s+1}(\partial\Omega_T)}\right)\right\}
\\
& \leq C(K)\left\{ \|V\|^2_{H^{s}(\Omega_t)}+\|\varphi\|^2_{H^s(\partial\Omega_t)}+\|\varphi\|^2_{L_{\infty}(\partial\Omega_T)}
\|{\rm coeff}\|^2_{s+3}\right\}.
\end{split}
\]
Here we used the trace theorem and the counterpart of the calculus inequality \eqref{c40'} for the domain $\partial\Omega_t$. Estimating analogously the rest terms in \eqref{P2}, we obtain
\begin{equation}
-2\int\limits_{\partial\Omega_t}P_2{d}x_2d\tau \leq C(K)\mathcal{ M}(t),
\label{P2est}
\end{equation}
where
\begin{equation}
\begin{split}
\mathcal{ M}(t) &=  \mathcal{ N}(T)+\int\limits_0^t\mathcal{ I}(\tau )\,{d}\tau,\qquad
\mathcal{ I}(t)=\nt V(t)\nt^2_{H^s(\Omega )}+
\nt \varphi (t)\nt^2_{H^s(\partial\Omega )},
\\
\mathcal{ N}(T) & = \|F\|^2_{H^s(\Omega_T)}
+\left(\|{U}\|^2_{W^1_{\infty}(\Omega_T)}+\|\varphi\|^2_{W^1_{\infty}(\partial\Omega_T)}+\|F\|^2_{L_{\infty}(\Omega_T)}\right)\left( 1+\|{\rm coeff}\|^2_{s+3}\right),
\end{split}
\label{It}
\end{equation}
with
\[
\nt u(t)\nt^2_{H^m(D)}:=
\sum\limits_{j=0}^m\|\partial_t^ju(t)\|^2_{H^{m-j}(D)}\qquad (D=\Omega\quad\mbox{or}\quad D=\partial\Omega ).
\]
Since only the biggest loss of derivatives from the coefficients will play the role for obtaining the final tame estimate, we have roughened inequality \eqref{P2est} by choosing the biggest loss.

Using the Rayleigh-Taylor sign condition \eqref{RTL}, it follows from \eqref{c39}, \eqref{c42}, \eqref{Q2"} and \eqref{P2est} that
\begin{equation}
\sum\limits_{\pm}\int\limits_{\Omega}(\mathcal{ A}_0^{\pm}\mathcal{D}_2^{\alpha}V^{\pm},\mathcal{D}_2^{\alpha}V^{\pm} ) {d}x + \frac{\epsilon}{2} \| \mathcal{D}_2^{\alpha} \varphi (t) \|^2_{L^2(\partial\Omega )}\leq \mathcal{K}(t)+C(K)\mathcal{ M}(t),\label{c39'}
\end{equation}
where
\[
\mathcal{K}(t)=-2\int\limits_{\partial\Omega_t}[\partial_1\widehat{H}_{\tau}](\mathcal{D}_2^{\alpha} \varphi) (\partial_0\partial_{\rm tan}^{\beta}V_4^+)|_{x_1=0}\,{d}x_2d\tau .
\]
We now need to estimate the boundary integral $\mathcal K(t)$ connected with the underlined term in \eqref{Q2"}. To this end we first integrate by parts and use the boundary condition \eqref{b50b.4'}, and then we apply the same arguments as those towards the proof of \eqref{P2est}:
\[
\begin{split}
\mathcal{K}(t)=&2\int\limits_{\partial\Omega_t}\Bigl\{\Bigl(\partial_0[\partial_1\widehat{H}_{\tau}]\,\mathcal{D}_2^{\alpha}
\varphi-[\partial_1\widehat{H}_{\tau}]\,[\mathcal{D}_2^{\alpha} ,\hat{v}_2^+]\partial_2\varphi \\ &\quad\qquad\; +[\partial_1\widehat{H}_{\tau}]\left\{\mathcal{D}_2^{\alpha} (\partial_0\varphi)+(\partial_2\hat{v}_2^+)\mathcal{D}_2^{\alpha}\varphi\right\}\Bigr)\partial_{\rm tan}^{\beta}V_4^+ \Bigr\}\Bigr|_{x_1=0}dx_2d\tau  +\mathcal{L}(t)\\ =& 2\int\limits_{\partial\Omega_t}\Bigl\{\Bigl(\partial_0[\partial_1\widehat{H}_{\tau}]\,\mathcal{D}_2^{\alpha}
\varphi-[\partial_1\widehat{H}_{\tau}] \left\{[\mathcal{D}_2^{\alpha} ,\hat{v}_2^+]\partial_2\varphi -\mathcal{D}_2^{\alpha}(\varphi\partial_1\hat{v}_N^+)\right\}\\
&
\quad\qquad\; +[\partial_1\widehat{H}_{\tau}](\partial_2\hat{v}_2^+)\mathcal{D}_2^{\alpha}\varphi\Bigr) \partial_{\rm tan}^{\beta}V_4^+\Bigr\}\Bigr|_{x_1=0}dx_2d\tau
 +\widetilde{\mathcal{K}}(t)+\mathcal{L}(t)\\
\leq & C(K)\mathcal{ M}(t)+\widetilde{\mathcal{K}}(t)+\mathcal{L}(t),
\end{split}
\]
where
\[
\widetilde{\mathcal{K}}(t)= 2\int\limits_{\partial\Omega_t}[\partial_1\widehat{H}_{\tau}](\mathcal{D}_2^{\alpha} V_2^+\partial_{\rm tan}^{\beta}V_4^+)|_{x_1=0}\,dx_2d\tau
\]
and
\[
\mathcal{L}(t)=-2\int\limits_{\partial\Omega}[\partial_1\widehat{H}_{\tau}]\mathcal{D}_2^{\alpha} \varphi\, \partial_{\rm tan}^{\beta}V^+_{4|x_1=0}\,{d}x_2.
\]

For estimating the boundary integral $\widetilde{\mathcal{K}}(t)$ we pass to a volume integral and then integrate by parts:
\begin{equation}
\begin{split}
\widetilde{\mathcal{K}}(t)= &\int\limits_{\Omega_t}\partial_1\bigl\{\hat{c}\,\mathcal{D}_2^{\alpha} V_2^+\partial_{\rm tan}^{\beta}V_4^+\bigr\}\,dxd\tau
=\int\limits_{\Omega_t}\partial_1\bigl\{\hat{c}\,\partial_2(\partial_{\rm tan}^{\beta}V_2^+)\,\partial_{\rm tan}^{\beta}V_4^+\bigr\}\,dxd\tau \\
= &\int\limits_{\Omega_t}\bigl\{\partial_1\hat{c}\;\mathcal{D}_2^{\alpha}V_2^+\partial_{\rm tan}^{\beta}V_4^+
+\hat{c}\;\mathcal{D}_2^{\alpha}V_2^+\,\partial_{\rm tan}^{\beta}\partial_1V_4^+ \\
&\qquad -\partial_2\hat{c}\;\partial_{\rm tan}^{\beta} \partial_1V_2^+\;\partial_{\rm tan}^{\beta}V_4^+
-\hat{c}\; \partial_{\rm tan}^{\beta}\partial_1V_2^+\;\mathcal{D}_2^{\alpha}V_4^+\bigr\}\,dx d\tau\\
\leq & C(K)\|V\|^2_{H^s(\Omega_t)}\leq  C(K)\mathcal{ M}(t),
\end{split}
\label{Ktild}
\end{equation}
where $\hat{c}=-2(\partial_1\widehat{H}_{\tau}^++\partial_1\widehat{H}_{\tau}^-)$. This yields
\begin{equation}
\mathcal{K}(t)\leq C(K)\mathcal{ M}(t)+\mathcal{L}(t).
\label{tKt}
\end{equation}

By using the Young inequality and the passage to a volume integral we now estimate the boundary integral $\mathcal L(t)$:
\begin{equation}
\begin{split}
\mathcal L(t) \leq & \tilde{\varepsilon}C(K)\| \mathcal{D}_2^{\alpha} \varphi (t) \|^2_{L^2(\partial\Omega )}+\frac{1}{\tilde{\varepsilon}}
\int\limits_{\partial\Omega} (\partial_{\rm tan}^{\beta}V^+_4)^2_{|x_1=0}\,{d}x_2\\
= &\tilde{\varepsilon}C(K)\| \mathcal{D}_2^{\alpha} \varphi (t) \|^2_{L^2(\partial\Omega )}-\frac{2}{\tilde{\varepsilon}}
\int\limits_{\Omega} \partial_{\rm tan}^{\beta}V^+_4\,\partial_{\rm tan}^{\beta}\partial_1V^+_4{d}x \\
 \leq &\tilde{\varepsilon}C(K)\| \mathcal{D}_2^{\alpha} \varphi (t) \|^2_{L^2(\partial\Omega )}+\frac{1}{\tilde{\varepsilon}}
 \left\{ \frac{1}{\tilde{\varepsilon}^2} \nt V_4^+(t)\nt^2_{H^{s-1}(\Omega)}+\tilde{\varepsilon}^2\nt V_4^+(t)\nt^2_{H^{s}(\Omega)}\right\},
\end{split}
\label{tLt}
\end{equation}
where the constant $C=C(K)$ does not depend on a small positive constant $\tilde{\varepsilon}$ whose choice will be made below.
Taking into account the elementary inequality
\begin{equation}
\nt u(t)\nt^2_{H^{s-1}(D )}\leq \int\limits_{0}^{t}\nt u(\tau )\nt^2_{H^s(D )}d\tau =\|u\|^2_{H^s([0,t]\times D)}
\label{elin}
\end{equation}
($D=\Omega$ or $D=\partial\Omega $), from \eqref{c39'}, \eqref{tKt} and \eqref{tLt} we deduce
\begin{equation}
\begin{split}
\sum\limits_{\pm}\int\limits_{\Omega} &(\mathcal{ A}_0^{\pm}\mathcal{D}_2^{\alpha}V^{\pm},\mathcal{D}_2^{\alpha}V^{\pm} ) {d}x + \frac{\epsilon}{2} \| \mathcal{D}_2^{\alpha} \varphi (t) \|^2_{L^2(\partial\Omega )} \\  & \leq C(K)\left\{\mathcal{ M}(t) +
\tilde{\varepsilon}\| \mathcal{D}_2^{\alpha} \varphi (t) \|^2_{L^2(\partial\Omega )}\right\}+\tilde{\varepsilon}\nt V(t)\nt^2_{H^{s}(\Omega)}+\frac{1}{\tilde{\varepsilon}^3} \|V\|^2_{H^{s}(\Omega_t)}.
\end{split}
\label{c39"}
\end{equation}

Assuming for the present moment that $\tilde{\varepsilon}=\epsilon\hat{\varepsilon}$ and $\hat{\varepsilon}\leq \hat{\varepsilon}_1=1 /(2C_1)$, where $C_1:=C(K)$ is the  constant from \eqref{c39"}, but leaving the final choice of  $\hat{\varepsilon}$ for the future, from \eqref{c39"} we get
\[
\begin{split}
\sum\limits_{\pm}\int\limits_{\Omega} &(\mathcal{ A}_0^{\pm}\mathcal{D}_2^{\alpha}V^{\pm},\mathcal{D}_2^{\alpha}V^{\pm} ) {d}x + \frac{\epsilon}{2} \| \mathcal{D}_2^{\alpha} \varphi (t) \|^2_{L^2(\partial\Omega )} \\  & \leq C(K)\mathcal{ M}(t) +\tilde{\varepsilon}\nt V(t)\nt^2_{H^{s}(\Omega)}+\frac{1}{\tilde{\varepsilon}^3} \,\|V\|^2_{H^{s}(\Omega_t)}
\end{split}
\]
that implies
\begin{equation}
\begin{split}
\sum\limits_{\pm}\int\limits_{\Omega} &(\mathcal{ A}_0^{\pm}\mathcal{D}_2^{\alpha}V^{\pm},\mathcal{D}_2^{\alpha}V^{\pm} ) {d}x + \| \mathcal{D}_2^{\alpha} \varphi (t) \|^2_{L^2(\partial\Omega )} \\  & \leq C(K)\mathcal{ M}(t) +2\hat{\varepsilon}\nt V(t)\nt^2_{H^{s}(\Omega)}+\frac{2}{\epsilon^4\hat{\varepsilon}^3} \,\|V\|^2_{H^{s}(\Omega_t)},
\end{split}
\label{c39^}
\end{equation}
where the constant $C(K)=C(K,\epsilon )$ depends on $\epsilon$ and without loss of generality we suppose that $\epsilon \leq 4$.

\subsection{Estimation of tangential derivatives containing the  material derivative}

Applying to systems \eqref{c29} the operator $\mathcal{D}_0^{\beta}:=\partial_0\partial_{\rm tan}^{\beta}$, with
$|\beta |\leq s-1$, and using, as in \eqref{c39}, standard arguments of the energy method, we obtain
\begin{equation}
\sum\limits_{\pm}\int\limits_{\Omega}(\mathcal{ A}_0^{\pm}\mathcal{D}_0^{\beta}V^{\pm},\mathcal{D}_0^{\beta}V^{\pm} ) {d}x + 2\int\limits_{\partial\Omega_t}{Q}_0{d}x_2d\tau=\widetilde{\mathcal{R}},\label{c39.0}
\end{equation}
\begin{equation}
\begin{split}
{Q}_0= &-\frac{1}{2}\sum\limits_{\pm}\bigl(\mathcal{A}_1^{\pm}\mathcal{D}_0^{\beta}V^\pm , \mathcal{D}_0^{\beta}V^\pm \bigr)\bigr|_{x_1=0} \\ = &
\left.\left\{-\mathcal{D}_0^{\beta}V_2^+[\mathcal{D}_0^{\beta}V_1]+
\left(\widehat{H}_N^+\mathcal{D}_0^{\beta}V_3^+-\widehat{H}_2^+\mathcal{D}_0^{\beta}V_2^+\right)[\mathcal{D}_0^{\beta}V_5]\right\}\right|_{x_1=0}  ,
\end{split}
\label{Q0}
\end{equation}
\begin{align*}
\widetilde{\mathcal{R}}= \sum\limits_{\pm}\int\limits_{\Omega_t}\Bigl(
\Bigl\{ &\left(\partial_t\mathcal{ A}_0^{\pm}+ \partial_1\mathcal{ A}_1^{\pm}+\partial_2\mathcal{ A}_2^{\pm}\right)\mathcal{D}^{\beta}_0V^{\pm} - 2\left[\mathcal{D}^{\beta}_0 ,\mathcal{ A}_0^{\pm}\right]\partial_tV^{\pm}
- 2\left[\mathcal{D}^{\beta}_0 ,\mathcal{ A}_1^{\pm}\right]\partial_1V^{\pm} \\ & - 2\left[\mathcal{D}^{\beta}_0 ,\mathcal{ A}_2^{\pm}\right]\partial_2V^{\pm}
 - 2\mathcal{D}^{\beta}_0(\mathcal{ A}_3^{\pm}V^{\pm} )+2\mathcal{D}^{\beta}_0\mathcal{ F}^{\pm} \\
 & +2\left(\partial_t\hat{v}_2^+\,\mathcal{ A}_0^{\pm}+ \partial_1\hat{v}_2^+\,\mathcal{ A}_1^{\pm}+\partial_2\hat{v}_2^+\,\mathcal{ A}_2^{\pm}\right)\partial_2\partial^{\beta}_{\rm tan}V^{\pm}\Bigr\},\mathcal{D}^{\beta}_0V^{\pm}\Bigr){d}x {d}\tau .
\end{align*}
Using, as in \eqref{c42}, the Moser-type calculus inequalities we estimate $\widetilde{\mathcal{R}}$ and roughening the result, from \eqref{c39.0} we obtain
\begin{equation}
\sum\limits_{\pm}\int\limits_{\Omega}(\mathcal{ A}_0^{\pm}\mathcal{D}_0^{\beta}V^{\pm},\mathcal{D}_0^{\beta}V^{\pm} ) {d}x + 2\int\limits_{\partial\Omega_t}{Q}_0{d}x_2d\tau\leq C(K)\mathcal{ M}(t).\label{c39.0'}
\end{equation}

By applying the differential operator $\mathcal{D}_0^{\beta}$ to the boundary conditions \eqref{b50b.1'} and \eqref{b50b.3'} and the differential operator $\partial_{\rm tan}^{\beta}$ to the boundary condition \eqref{b50b.4'}, one gets
\begin{align}
{[} \mathcal{D}_0^{\beta}V_1{]}=- {[}\partial_1\hat{p}{]}\mathcal{D}_0^{\beta}\varphi-\bigl[\mathcal{D}_0^{\beta},{[}\partial_1\hat{p}{]}\bigr]\varphi,\qquad & \label{b50b.1^}\\[3pt]
    {[} \mathcal{D}_0^{\beta}V_5{]}=- {[}\partial_1\widehat{H}_{\tau}{]}\mathcal{D}_0^{\beta}\varphi -\bigl[\mathcal{D}_0^{\beta},{[}\partial_1\widehat{H}_{\tau}{]}\bigr]\varphi,\qquad & \label{b50b.3^} \\[3pt]
\mathcal{D}_0^{\beta}\varphi=\partial_{\rm tan}^{\beta}V_2^+ - [\partial_{\rm tan}^{\beta},\hat{v}_2^+]\partial_2\varphi +\partial_{\rm tan}^{\beta}(\varphi \partial_1\hat{v}_N^+ )  \qquad &\mbox{on}\ \partial\Omega_T.
\label{b50b.4^}
\end{align}
Substituting \eqref{b50b.4^} into \eqref{b50b.1^} and \eqref{b50b.3^}, it follows from \eqref{Q0} that
\begin{equation}
\begin{split}
{Q}_0=
\biggl\{&\left(\partial_{\rm tan}^{\beta}V_2^+ - [\partial_{\rm tan}^{\beta},\hat{v}_2^+]\partial_2\varphi +\partial_{\rm tan}^{\beta}(\varphi \partial_1\hat{v}_N^+ ) \right) \Bigl({[}\partial_1\hat{p}{]}\mathcal{D}_0^{\beta}V_2^+\\ &
-{[}\partial_1\widehat{H}_{\tau}{]}
\widehat{H}_N^+\mathcal{D}_0^{\beta}V_3^+
+{[}\partial_1\widehat{H}_{\tau}{]}\widehat{H}_2^+\mathcal{D}_0^{\beta}V_2^+\Bigr)
+\mathcal{D}_0^{\beta}V_2^+\bigl[\mathcal{D}_0^{\beta},{[}\partial_1\hat{p}{]}\bigr]\varphi \\ &-
\left(\widehat{H}_N^+\mathcal{D}_0^{\beta}V_3^+-\widehat{H}_2^+\mathcal{D}_0^{\beta}V_2^+\right)
\bigl[\mathcal{D}_0^{\beta},{[}\partial_1\widehat{H}_{\tau}{]}\bigr]\varphi
\biggr\}\biggr|_{x_1=0}  .
\end{split}
\label{Q0'}
\end{equation}

The terms of the quadratic form $Q_0$ in \eqref{Q0'} are of two types:
\begin{equation}
{\rm coeff}\,\mathcal{D}_0^{\beta}V_i^+\partial_{\rm tan}^{\beta}V_j^+\qquad\mbox{and}\qquad {\rm coeff}\,\mathcal{D}_0^{\beta}V_i^+\partial_{\rm tan}^{\mu}\varphi,
\label{types}
\end{equation}
where coeff is a coefficient depending on the basic state, $i,j=2,3$ and $|\mu|\leq s-1$. We estimate the terms of the first type in \eqref{types} by passing to a volume integral and integrating by parts. As an example, we consider the term $\bigl(\hat{c}\mathcal{D}_0^{\beta}V_3^+\partial_{\rm tan}^{\beta}V_2^+\bigr)\bigr|_{x_1=0}$, where $\hat{c}=-2\widehat{H}_N^+(\partial_1\widehat{H}_{\tau}^++\partial_1\widehat{H}_{\tau}^-) $. Recalling the definition of $\mathcal{D}_0^{\beta}$ and the material derivative $\partial_0$, we decompose it as
\[
\bigl(\hat{c}\mathcal{D}_0^{\beta}V_3^+\partial_{\rm tan}^{\beta}V_2^+\bigr)\bigr|_{x_1=0}=\bigl(\hat{c}\partial_t(\partial_{\rm tan}^{\beta}V_3^+)\,\partial_{\rm tan}^{\beta}V_2^+\bigr)\bigr|_{x_1=0} +\bigl(\hat{c}\hat{v}_2^+\partial_2(\partial_{\rm tan}^{\beta}V_3^+)\,\partial_{\rm tan}^{\beta}V_2^+\bigr)\bigr|_{x_1=0},
\]
where the second term in the right-hand side is treated exactly as in \eqref{Ktild}. Regarding the first term, an additional integral connected with the integration over the time interval $[0,t]$ appears for it when we integrate by parts:
\[
\begin{split}
 -\int\limits_{\partial\Omega_t}\bigl(\hat{c} & \partial_t(\partial_{\rm tan}^{\beta}V_3^+)\,\partial_{\rm tan}^{\beta}V_2^+\bigr)\bigr|_{x_1=0}\,dx_2d\tau \\ &
= \int\limits_{\Omega_t}\partial_1\bigl\{\hat{c}\,\partial_t(\partial_{\rm tan}^{\beta}V_3^+)\,\partial_{\rm tan}^{\beta}V_2^+\bigr\}\,dxd\tau  = \int\limits_{\Omega_t}\bigl\{\partial_1\hat{c}\,\partial_{\rm tan}^{\beta}\partial_tV_3^+
\,\partial_{\rm tan}^{\beta}V_2^+
 \\
&\qquad +\hat{c}\,\partial_{\rm tan}^{\beta}\partial_tV_3^+\,\partial_{\rm tan}^{\beta}\partial_1V_2^+
-\partial_t\hat{c}\;\partial_{\rm tan}^{\beta} \partial_1V_3^+\;\partial_{\rm tan}^{\beta}V_2^+ \\[6pt]
&\qquad
-\hat{c}\; \partial_{\rm tan}^{\beta}\partial_1V_3^+\,\partial_{\rm tan}^{\beta}\partial_tV_2^+\bigr\}\,dx d\tau +\int\limits_{\Omega}\hat{c}\,\partial_{\rm tan}^{\beta}\partial_1V_3^+\,\partial_{\rm tan}^{\beta}V_2^+\,dx \\
&  \leq C(K)\mathcal{ M}(t)+\int\limits_{\Omega}\hat{c}\,\partial_{\rm tan}^{\beta}\partial_1V_3^+\,\partial_{\rm tan}^{\beta}V_2^+\,dx
\\
&  \leq C(K)\left\{\mathcal{M}(t) +\hat{\varepsilon}\nt V(t)\nt^2_{H^{s}(\Omega)}+\frac{1}{\hat{\varepsilon}} \,\nt V(t)\nt^2_{H^{s-1}(\Omega)} \right\}
\\
&  \leq C(K)\left\{\mathcal{M}(t) +\hat{\varepsilon}\nt V(t)\nt^2_{H^{s}(\Omega)}+\frac{1}{\hat{\varepsilon}} \,\|V\|^2_{H^{s}(\Omega_t)} \right\}.
\end{split}
\]
For estimating the additional integral over $\Omega$ we exploited the Young inequality and the elementary inequality \eqref{elin}. Moreover, here we use the same positive constant $\hat{\varepsilon}$ that was defined just after \eqref{c39"}.

Regarding the terms of the second type in \eqref{types} we omit detailed calculations but just notice that they are estimated by integrating by parts and using the trace theorem, the Young inequality and  inequality \eqref{elin}. We thus obtain the estimate
\[
-2\int\limits_{\partial\Omega_t}{Q}_0{d}x_2d\tau \leq C(K)\left\{\mathcal{M}(t) +\hat{\varepsilon}\nt V(t)\nt^2_{H^{s}(\Omega)}+\frac{1}{\hat{\varepsilon}} \left(\|V\|^2_{H^{s}(\Omega_t)} +\|\varphi \|^2_{H^{s}(\partial\Omega_t)}\right)\right\}
\]
giving together with \eqref{c39.0'} the estimate
\begin{equation}
\begin{split}
\sum\limits_{\pm}\int\limits_{\Omega}& (\mathcal{ A}_0^{\pm}\mathcal{D}_0^{\beta}V^{\pm},  \mathcal{D}_0^{\beta}V^{\pm} ) {d}x \\  & \leq C(K)\left\{\mathcal{M}(t) +\hat{\varepsilon}\nt V(t)\nt^2_{H^{s}(\Omega)}+\frac{1}{\hat{\varepsilon}} \left(\|V\|^2_{H^{s}(\Omega_t)} +\|\varphi \|^2_{H^{s}(\partial\Omega_t)}\right)
\right\}.
\label{c39.0^}
\end{split}
\end{equation}
Taking into account that
\[
\|\partial_tu (t)\|^2_{L^2(\Omega )} \leq  \|\partial_0u (t)\|^2_{L^2(\Omega)} +C(K)\|\partial_2u (t)\|^2_{L^2(\Omega )}
\]
and $\mathcal{A}_0^{\pm}>0$, the combination of \eqref{c39^} and \eqref{c39.0^} yields (we omit detailed simple arguments)
\begin{equation}
\begin{split}
\nt V(t)\nt^2_{{\rm tan},s} & +\nt \partial_2\varphi (t)\nt^2_{H^{s-1}(\partial\Omega)} \\ & \leq
C(K)\left\{\mathcal{M}(t) +\hat{\varepsilon}\nt V(t)\nt^2_{H^{s}(\Omega)}+\frac{1}{\hat{\varepsilon}^3} \,\|V\|^2_{H^{s}(\Omega_t)}
+\frac{1}{\hat{\varepsilon}}\|\varphi \|^2_{H^{s}(\partial\Omega_t)}
\right\},
\label{Vtan}
\end{split}
\end{equation}
where
\[
\nt u(t)\nt^2_{{\rm tan},m}:=\sum_{|\alpha|\leq m} \| \partial^{\alpha}_{\rm tan} u (t)\|^2_{L^2(\Omega )},
\]
and $C(K)=C(K,\bar{p},\epsilon )$ depends on the fixed constants $\bar{p}$ and $\epsilon$ from \eqref{5.1} and \eqref{RT1} respectively but does not depend on the constant $\hat{\varepsilon}$ which will be chosen after estimating normal derivatives of $V$. Moreover, adding to \eqref{Vtan}  inequality \eqref{elin} for $u=\varphi$ and $D=\partial\Omega$, we get
\begin{equation}
\begin{split}
\nt V(t)\nt^2_{{\rm tan},s}+& \nt \varphi (t)\nt^2_{H^{s-1}(\partial\Omega)}+\nt \partial_2\varphi (t)\nt^2_{H^{s-1}(\partial\Omega)} \\[6pt]  &\leq
C(K)\left\{\mathcal{M}(t) +\hat{\varepsilon}\nt V(t)\nt^2_{H^{s}(\Omega)}+\frac{1}{\hat{\varepsilon}^3} \,\|V\|^2_{H^{s}(\Omega_t)}
+\frac{1}{\hat{\varepsilon}}\|\varphi \|^2_{H^{s}(\partial\Omega_t)}
\right\}.
\end{split}
\label{Vtan'}
\end{equation}

\subsection{Estimation of normal derivatives through tangential ones}

In view of assumption \eqref{cdass} and the continuity of the basic state \eqref{a21}, there exists such a small but fixed constant $\delta >0$ depending on $\kappa$ that
\begin{equation}
|\widehat{H}_N^{\pm}|\geq \frac{\kappa}{4}>0 \quad\mbox{in}\ \Omega_t^{\delta},
\label{cdass"}
\end{equation}
where $\Omega_t^{\delta}=[0 ,t]\times \Omega_{\delta}$  and $\Omega_{\delta} =\{x_1\in(0,\delta),\ x_2\in\mathbb{T}\}$ is the $\delta$-neighbourhood of the boundary $x_1=0$. In $\Omega_t^{\delta}$ we may consider the matrices
\[
\mathcal{B}_1^{\pm }= \pm\begin{pmatrix}
0 & 1 & \frac{\widehat{H}_2^{\pm}}{\widehat{H}_N^{\pm}} & 0 & 0 & 0\\
1 & 0 & 0 & 0 & 0 & 0\\
\frac{\widehat{H}_2^{\pm}}{\widehat{H}_N^{\pm}} & 0 & 0 & 0 & - \frac{1}{\widehat{H}_N^{\pm}} & 0\\
0 & 0 & 0 & 0 & 0 & 0\\
0 & 0 & - \frac{1}{\widehat{H}_N^{\pm}} & 0 & 0 & 0\\
0 & 0 & 0 & 0 & 0 & 0
\end{pmatrix}.
\]
Then, it follow from \eqref{c29} and \eqref{c30} that
\begin{equation}
\begin{split}
(\partial_1V_1^{\pm}& ,\partial_1 V_2^{\pm}, \partial_1V_3^{\pm},0,\partial_1V_5^{\pm},0) \\  & =
\mathcal{B}_1^{\pm }\left( \mathcal{F}^{\pm} -\mathcal{A}_0^{\pm}\partial_t{V}^{\pm}- \mathcal{A}_2^{\pm}\partial_2{V}^{\pm}-\mathcal{A}_3^{\pm}{V}^{\pm}-\mathcal{A}^{\pm}_{(0)}\partial_1V^{\pm}\right)
 \qquad \mbox{in}\ \Omega_t^{\delta}.
\end{split}
\label{c29^}
\end{equation}

Applying to \eqref{c29^} the operator $\partial_{\rm tan}^{\beta}$ with $|\beta |\leq s-1$,
using standard decompositions like
\[
\partial_{\rm tan}^{\beta}(B\partial_i V)=B\partial_{\rm tan}^{\beta}\partial_i V +
[\partial_{\rm tan}^{\beta},B]\partial_i V,
\]
taking into account the fact that $\mathcal{ A}^{\pm}_{(0)}|_{x_1=0}=0$,  and employing counterparts of the calculus inequalities \eqref{c40}, \eqref{c41} and \eqref{c40'} for the ``layerwise'' norms $\nt (\cdot )(t)\nt$ (see \cite{Sch}) as well as inequality \eqref{elin}, we estimate the normal derivative of $\partial_{\rm tan}^{\beta}V_n^{\pm}$ (see \eqref{V_n}):
\begin{equation}
\begin{split}
\|\partial_1\partial_{\rm tan}^{\beta}V_n (t) & \|^2_{L^2(\Omega_{\delta})} \\ \leq C(K)\Bigl\{ & \nt V(t)\nt^2_{{\rm tan},s}
+\|\sigma \partial_1\partial_{\rm tan}^{\beta} V(t)\|^2_{L^2(\Omega )} \\ &+ \nt V(t) \nt^2_{H^{s-1}(\Omega )}
+ \nt F (t)\nt^2_{H^{s-1}(\Omega )} \\ &+
\left(\|{U}\|^2_{W^1_{\infty}(\Omega_T)}+\|F\|^2_{L_{\infty}(\Omega_T)}\right)\left( 1+\nt{\rm coeff} (t) \nt^2_{s}\right)\Bigr\}
\\ \leq C(K)\Bigl\{ & \nt V(t)\nt^2_{{\rm tan},s}
+\|\sigma \partial_1\partial_{\rm tan}^{\beta} V(t)\|^2_{L^2(\Omega )} +\mathcal{M}(t)\Bigr\} ,
\end{split}
\label{45}
\end{equation}
where $V_n=(V_n^+,V_n^-)$, the constant $C(K)=C(K,\kappa )$ depends on the fixed constant $\kappa$ from \eqref{cdass}, and $\sigma=\sigma (x_1)\in C^{\infty}(\mathbb{R}_+)$ is a monotone increasing function such that $\sigma (x_1)=x_1$ in a neighborhood of the origin and $\sigma (x_1)=1$ for $x_1$ large enough. Since $\sigma |_{x_1=0}=0$, we do not need to use boundary conditions to estimate $\sigma\partial_1^j\partial_{\rm tan}^{\gamma}V$, with $j+|\gamma |\leq s$,  and we easily get the inequality
\begin{equation}
\| \sigma\partial_1^j\partial_{\rm tan}^{\gamma}V(t)\|^2_{L^2(\Omega )}\leq C(K)\mathcal{M}(t).
\label{46}
\end{equation}
It follows from \eqref{45} and \eqref{46} for $j=1$ that
\begin{equation}
\|\partial_1\partial_{\rm tan}^{\beta}V_n (t)  \|^2_{L^2(\Omega_{\delta})}  \leq C(K)\left\{  \nt V(t)\nt^2_{{\rm tan},s} +\mathcal{M}(t)\right\} .
\label{45'}
\end{equation}

Since the weight $\sigma$ in \eqref{46} is not zero outside the boundary, it follows from estimate \eqref{46} that
\begin{equation}
\| \partial_1^j\partial_{\rm tan}^{\gamma}V(t)\|^2_{L^2(\Omega \setminus \Omega_{\delta} )}\leq C(K)\mathcal{M}(t),
\label{46'}
\end{equation}
where the constant $C(K)$ depends, in particular, on $\delta$ and so on $\kappa$. Combining \eqref{45'} and \eqref{46'} for $j=1$, we get
\begin{equation}
\sum_{i=1}^{k}\sum_{|\alpha|\leq s-i}\|\partial_1^i\partial_{\rm tan}^{\alpha}V_n (t)\|^2_{L_2(\Omega )}\leq C(K)\left\{  \nt V(t)\nt^2_{{\rm tan},s} +\mathcal{M}(t)\right\},
\label{45^}
\end{equation}
with $k=1$. Estimate \eqref{45^} for $k=s$ is easily proved by finite induction. The combination of \eqref{Vtan'} and \eqref{45^} for $k=s$ yields
\begin{equation}
\begin{split}
\nt V(t)\nt^2_{{\rm tan},s} & + \nt V_n (t)\nt^2_{H^{s}(\Omega)} +\nt \varphi (t)\nt^2_{H^{s-1}(\partial\Omega)}+\nt \partial_2\varphi (t)\nt^2_{H^{s-1}(\partial\Omega)} \\[6pt]  & \leq
C(K)\left\{\mathcal{M}(t) +\hat{\varepsilon}\nt V(t)\nt^2_{H^{s}(\Omega)}+\frac{1}{\hat{\varepsilon}^3} \,\|V\|^2_{H^{s}(\Omega_t)}
+\frac{1}{\hat{\varepsilon}}\|\varphi \|^2_{H^{s}(\partial\Omega_t)}
\right\},
\end{split}
\label{Vtan^}
\end{equation}
where the constant $C=C(K)$ depends also on the fixed constants $\epsilon$ and $\kappa$ from \eqref{cdass} and \eqref{RTL}: $C=C(K,\epsilon ,\kappa )$.

Missing normal derivatives in \eqref{Vtan^} for the ``characteristic'' parts $(V_4^{\pm},V_6^{\pm})=(H_N^{\pm},S^{\pm})$ of the unknowns $V^{\pm}$ can be estimated from the last equations in \eqref{c29} (or \eqref{b48b}),
\begin{equation}
\partial_tV_6^{\pm}+\frac{1}{\partial_1\widehat{\Phi}^{\pm}} \left((\hat{w}^{\pm} ,\nabla V_6^{\pm} ) +\bigl(\partial_1\widehat{S}^{\pm}\bigr)V_2^{\pm}+ \bigl(\partial_1\widehat{\Phi}^{\pm}\partial_2\widehat{S}^{\pm}\bigr)V_3^{\pm} \right) =F_6^\pm \quad\mbox{in}\ \Omega_T,
\label{entropy}
\end{equation}
and the equations for the linearized divergences of the magnetic fields
\[
\xi^{\pm}=\div h^{\pm} =\partial_1V_4^{\pm} +\partial_2\left(\frac{\partial_1\widehat{\Phi}^{\pm}}{1+(\partial_2\widehat{\Psi}^{\pm})^2}\bigl( V_5^{\pm}- V_4^{\pm}\partial_2\widehat{\Psi}^{\pm}\bigr)\right)
\]
obtained by applying the div operator to the systems for $h^{\pm}$ following from \eqref{b48b} (see \cite{T09}),
\begin{equation}
\partial_t \left(\frac{\xi^{\pm}}{\partial_1\widehat{\Phi}^{\pm}}\right)+ \frac{1}{\partial_1\widehat{\Phi}^{\pm}}\left\{ \left(\hat{w}^{\pm} ,\nabla \left(\frac{\xi^{\pm}}{\partial_1\widehat{\Phi}^{\pm}}\right)\right) + \frac{{\rm div}\,\hat{u}^{\pm}}{\partial_1\widehat{\Phi}^{\pm}}\,\xi^{\pm}\right\}=\frac{\div F_h^{\pm}}{\partial_1\widehat{\Phi}^{\pm}}+r^{\pm}\quad \mbox{in}\ \Omega_T,
\label{xi}
\end{equation}
where ${h}^{\pm}=({H}_{N}^{\pm},{H}_2^{\pm}\partial_1\widehat{\Phi}^{\pm})$, ${F}_{h}^{\pm}=(F_4^{\pm}-F_5^{\pm}\partial_2\widehat{\Psi}^{\pm} ,\partial_1\widehat{\Phi}^{\pm}F_5^{\pm})$, and $r^{\pm}$ are sums of lower-order terms which are ${\rm coeff}V_k^{\pm}$ or ${\rm coeff}\partial_jV_k^{\pm}$ ($k=2,3$, $j=1,2$) with coeff being proportional to $\div \hat{h}^\pm$.\footnote{Clearly,  $r^{\pm}\equiv 0$ if, as in \cite{Cont1}, we assume that the basic state satisfies the fourth and fifth equations of systems in \eqref{11.1} as well as the divergence constraints \eqref{14} at $t=0$, to be exact, $\div \hat{h}|_{t\leq 0}=0$ (see Remark \ref{r3}).}
Both equations \eqref{entropy} and \eqref{xi} do not need boundary conditions because, in view of \eqref{a12'}, the first component of the vector $\hat{w}$ is zero on the boundary $x_1=0$. Therefore, omitting detailed simple arguments of the energy method supplemented  with the application of the calculus inequality \eqref{c40'}, we deduce the estimate
\[
\nt (V_4^{\pm},V_6^{\pm}) (t)\nt^2_{H^{s}(\Omega)}\leq C(K)\left\{  \nt V(t)\nt^2_{{\rm tan},s} +\mathcal{M}(t)\right\}
\]
whose combination with \eqref{Vtan^} implies
\begin{equation}
\begin{split}
\nt V (t)\nt^2_{H^{s}(\Omega)} & +\nt \varphi (t)\nt^2_{H^{s-1}(\partial\Omega)}+\nt \partial_2\varphi (t)\nt^2_{H^{s-1}(\partial\Omega)} \\[6pt]  & \leq
C(K)\left\{\mathcal{M}(t) +\hat{\varepsilon}\nt V(t)\nt^2_{H^{s}(\Omega)}+\frac{1}{\hat{\varepsilon}^3} \,\|V\|^2_{H^{s}(\Omega_t)}
+\frac{1}{\hat{\varepsilon}}\|\varphi \|^2_{H^{s}(\partial\Omega_t)}
\right\}.
\end{split}
\label{Vfull}
\end{equation}

By applying the differential operator $\partial_{\rm tan}^{\beta}$ with $|\beta |\leq s-1$ to the boundary condition \eqref{b50b.4'} and using the trace theorem as well as the calculus inequality \eqref{c40'} and inequality \eqref{elin} (for $u=\varphi$ and $D=\partial\Omega$), one gets
\[
\nt \partial_t\varphi (t)\nt^2_{H^{s-1}(\partial\Omega)}\leq C(K)\left\{  \nt V(t)\nt^2_{H^{s}(\Omega)} + \nt \partial_2\varphi (t)\nt^2_{H^{s-1}(\partial\Omega)} +\mathcal{M}(t)\right\}.
\]
The last estimate and \eqref{Vfull} yield
\begin{equation}
\begin{split}
\nt V (t)\nt^2_{H^{s}(\Omega)} & +\nt \varphi (t)\nt^2_{H^{s}(\partial\Omega)} \\[6pt]  & \leq
C(K)\left\{\mathcal{M}(t) +\hat{\varepsilon}\nt V(t)\nt^2_{H^{s}(\Omega)}+\frac{1}{\hat{\varepsilon}^3} \,\|V\|^2_{H^{s}(\Omega_t)}
+\frac{1}{\hat{\varepsilon}}\|\varphi \|^2_{H^{s}(\partial\Omega_t)}
\right\}.
\end{split}
\label{Vfull'}
\end{equation}
At last, choosing $\hat{\varepsilon}$ to be small enough, we get the ``closed'' energy inequality
\[
\mathcal{ I}(t)\leq C(K)\mathcal{M}(t)=C(K)\biggl\{\mathcal{ N}(T)+\int\limits_0^t\mathcal{ I}(\tau )\,{d}\tau \biggr\}.
\]
Applying then Gronwall's lemma, we obtain the energy a priori estimate
\begin{equation}
\mathcal{ I}(t)\leq C(K)\,e^{C(K)T}\mathcal{ N}(T).
\label{enaprest}
\end{equation}

\subsection{Proof of the tame estimate (\ref{38'})}

Integrating \eqref{enaprest} over the interval $[0,T]$, we come to the estimate
\begin{equation}
\|V\|^2_{H^s(\Omega_T )}+\|\varphi\|^2_{H^{s}(\partial\Omega_T)}\leq C(K)Te^{C(K)T}\mathcal{ N}(T).
\label{c51}
\end{equation}
Recall that ${U}^{\pm}=J^{\pm}V^{\pm}$ (see \eqref{J}). Taking into account the decompositions $J^{\pm}(\hat{\varphi})= I +J_0^{\pm}(\hat{\varphi})$ and $J_0^{\pm}(0)=0$, using \eqref{c40} and \eqref{c41},
and the inequality
\[
\|u\|^2_{H^m([0,T]\times D)}\leq  T\|u\|^2_{H^{m+1}([0,T]\times D)}
\]
following from the integration of \eqref{elin} over the time interval $[0,T]$,
we obtain
\begin{equation}
\begin{split}
\|{U}\|^2_{H^s(\Omega_T )} & = \sum_{\pm}\|V^{\pm} + J_0^{\pm}V^{\pm}\|^2_{H^s(\Omega_T )} \\
 & \leq C(K)\bigl(
\|V\|^2_{H^s(\Omega_T )} +\|{U}\|^2_{L_{\infty}(\Omega_T)}\|{\rm coeff}\|_{s+1}^2\bigr) \\ &
\leq C(K)\|V\|^2_{H^s(\Omega_T )}+
TC(K)\|{U}\|^2_{L_{\infty}(\Omega_T)}\|{\rm coeff}\|_{s+2}^2.
\end{split}
\label{c52}
\end{equation}
Inequalities \eqref{c51} and \eqref{c52} imply
\begin{equation}
\|{U}\|^2_{H^s(\Omega_T )}+\|\varphi\|^2_{H^{s}(\partial\Omega_T)}\leq C(K)Te^{C(K)T}\mathcal{ N}(T).
\label{c53}
\end{equation}

Applying Sobolev's embeddings, from (\ref{c53}) with $s\geq 3$, we get
\begin{equation}
\begin{split}
\|{U} & \|_{H^s(\Omega_T )}+\|\varphi\|_{H^{s}(\partial\Omega_T)} \\
&
\begin{split}
\leq C(K)T^{1/2}e^{C(K)T}\Bigl\{ \|F\|_{H^s(\Omega_T)}
+&\left( \|{U}\|_{H^3(\Omega_T)}+\|\varphi\|_{H^3(\partial\Omega_T)} +\|F\|_{H^3(\Omega_T)}\right) \\ &\,\cdot
\bigl(
\|\widehat{U}\|_{H^{s+3}(\Omega_T )}+\|\hat{\varphi}\|_{H^{s+3}(\partial\Omega_T)}\bigr)\Bigr\},
\end{split}
\end{split}
\label{c54}
\end{equation}
where we have absorbed some norms $\|{U}\|_{H^3(\Omega_T)}$ and $\| \varphi\|_{H^3(\partial\Omega_T)}$ in the left-hand side by choosing $T$ small enough. Considering \eqref{c54} for $s=3$ and using (\ref{37}), we obtain for $T$ small enough that
\begin{equation}
\|{U}\|_{H^3(\Omega_T )}+\|\varphi\|_{H^{3}(\partial\Omega_T)}\leq C(K_0)\|F\|_{H^3(\Omega_T)}.
\label{c55}
\end{equation}
It is natural to assume that $T<1$ and, hence, we can suppose that the constant $C(K_0)$ does not depend on $T$.
Inequalities \eqref{c54} and \eqref{c55} imply the tame estimate \eqref{38'}.

\subsection{Proof of Theorem \ref{t3.1}}

Using the Moser-type calculus inequalities, inequality \eqref{tildU'} and Sobolev's embeddings, from \eqref{a87''} we get the estimate
\[
\begin{split}
\|F & \|_{H^s(\Omega_T)}\\ &\leq C(K_0)\left\{\|f\|_{H^s(\Omega_T)} +\|g\|_{H^{s+1}(\partial\Omega_T)}+
\| g\|_{H^{4}(\partial\Omega_T)} \bigl(
\|\widehat{U}\|_{H^{s+1}(\Omega_T )}+\|\hat{\varphi}\|_{H^{s+1}(\partial\Omega_T)}\bigr)\right\},
\end{split}
\]
for $F$ which together with \eqref{38'} and \eqref{a87'} (recall that the indices $^{\natural}$ were dropped) gives the desired tame a priori  estimate \eqref{38}.

Formally, the existence of solutions $(\dot{U},\varphi )\in H^1(\Omega_T)\times H^1(\partial\Omega_T)$ to problem \eqref{28}--\eqref{30} was proved in \cite{Cont1} (see Theorem \ref{t1L}). However, we can here omit a formal proof of the existence of solutions having an arbitrary degree of smoothness and suppose that the existence result of Theorem \ref{t1L} is also valid for the function spaces $H^s(\Omega_T)\times H^s(\partial\Omega_T)$ with $s\geq 1$ because the arguments in \cite{Cont1} towards the proof of existence are easily extended to these function spaces under the same assumptions about the regularity of the basic state $(\widehat{U},\hat{\varphi})$ as in Theorem \ref{t3.1}. The proof of Theorem \ref{t3.1} is thus complete.

\section{Compatibility conditions and approximate solution}
\label{s4}

\setcounter{subsubsection}{0}

\subsection{The compatibility conditions for the initial data (\ref{13.1})}

To use the tame estimate \eqref{38} for the proof of convergence of the Nash-Moser iteration, we should reduce our nonlinear problem \eqref{11.1}--\eqref{13.1} on $\Omega_T^+$ to that on $\Omega_T$ (see \eqref{OmegaT}) whose solutions vanish in the past. This is achieved by the classical argument suggesting to absorb the initial data into the interior equations by constructing a so-called \textit{approximate solution}. Before constructing the approximate solution we have to define \textit{compatibility conditions} for the initial data \eqref{13.1}.

Suppose we are given initial data $((U_0^+,U_0^-),\varphi_0)$ that satisfy all of the assumptions of Theorem \ref{t1}.
Let
\[
U^{\pm (0)}=(p^{\pm (0)},v_1^{\pm (0)},v_2^{\pm (0)},H_1^{\pm (0)},H_2^{\pm (0)},S^{\pm (0)}):=U_0^{\pm}\quad\mbox{and}\quad \varphi^{(0)}:=\varphi_0.
\]
Let also $\Psi^{\pm (i )}:=\chi (\pm x_1)\varphi^{(i )}(t,x_2)$ and $v^{\pm (i )}:=\bigl(v_1^{\pm (i )},v_2^{\pm (i )}\bigr)$, $H^{\pm (i)}:=\bigl(H_1^{\pm (i )},H_2^{\pm (i )}\bigr)$, where $i =0$ but below these notations will be used with indices $i\geq 0$. Assuming that the hyperbolicity condition \eqref{5.1'} is satisfied, we rewrite systems \eqref{11.1} in the form
\begin{equation}
\partial_t U^{\pm} = -\left(A_0(U^{\pm}+ \bar{U}^{\pm})\right)^{-1}\left( \widetilde{A}_1(U^{\pm}+ \bar{U}^{\pm} , {\Psi}^{\pm})\partial_1U^{\pm} +
A_2(U^{\pm}+ \bar{U}^{\pm} )\partial_2U^{\pm} \right).
\label{56}
\end{equation}
The traces
\[
U^{\pm (j)}=(p^{\pm (j)},v_1^{\pm (j)},v_2^{\pm (j)},H_1^{\pm (j)},H_2^{\pm (j)},S^{\pm (j)})=\partial_t^jU^{\pm}|_{t=0}\quad \mbox{and}\quad \varphi^{(j)}=\partial_t^j\varphi|_{t=0},
\]
with $j\geq 1$, are recursively defined by the formal application of the differential operator $\partial_t^{j-1}$
to the boundary condition
\begin{equation}
\partial_t\varphi =\left.\left(v_1^+-v_2^+\partial_2\varphi \right)\right|_{x_1=0}
\label{57}
\end{equation}
and \eqref{56} and evaluating $\partial_t^j\varphi$ and $\partial_t^jU$ at $t=0$. Moreover,
\[
\Psi^{\pm (j)}=\chi (\pm x_1)\varphi^{(j)}(t,x_2),\quad H_N^{\pm (j)}=\partial_t^jH_N^{\pm}|_{t=0}
=H_1^{\pm (j)}-\sum\limits_{i=0}^{j}C_j^iH_2^{\pm (j-i)}\partial_2\Psi^{\pm (i)},
\]
\[
H_{\tau}^{\pm (j)}=\partial_t^jH_{\tau}^{\pm}|_{t=0}
=\sum\limits_{i=0}^{j}C_j^iH_1^{\pm (j-i)}\partial_2\Psi^{\pm (i)}+H_2^{\pm (j)}.
\]

We naturally define the zeroth-order compatibility conditions as
$\bigl[p^{(0)}\bigr]=0$, $\bigl[v^{(0)}\bigr]=0$, $\bigl[H_{\tau}^{(0)}\bigr]=0$. Evaluating \eqref{57} at $t=0$, we get
\begin{equation}
\varphi^{(1)} =\bigl(v_1^{+(0)}-v_2^{+(0)}\partial_2\varphi^{(0)} \bigr)\bigr|_{x_1=0},
\label{58}
\end{equation}
and then, with $\partial_t\Psi^{\pm}|_{t=0}=\chi(\pm x_1)\varphi^{(1)}(x_2)$, from \eqref{56} evaluated at $t=0$ we define $U^{\pm(1)}$. The first-order compatibility conditions $\bigl[p^{(1)}\bigr]=0$, $\bigl[v^{(1)}\bigr]=0$, $\bigl[H_{\tau}^{(1)}\bigr]=0$ depend on $\varphi^{(0)}$ and $\varphi^{(1)}$. Knowing $\varphi^{(1)}$ and $U^{\pm(1)}$ we can then find $\varphi^{(2)}$, $U^{\pm(2)}$, etc.
The following lemma is the analogue of lemma 4.2.1 in \cite{Met}, lemma 19 in \cite{ST} and lemma 4.1 in \cite{Tcpam}.

\begin{lemma}
Let $\mu\in\mathbb{N}$, $\mu \geq 3$, and $((U_0^+,U_0^-),\varphi_0)\in H^{\mu +1/2}(\Omega)\times H^{\mu +1/2}(\partial\Omega )$. Then, the procedure described above determines $U^{\pm(j)}\in H^{\mu +1/2 -j}(\Omega )$ and $\varphi^{(j)}\in H^{\mu +1/2-j}(\partial\Omega)$ for $j= 1,\ldots ,\mu$. Moreover,
\begin{equation}
\sum_{j=1}^{\mu}\left( \bigl\|U^{+(j)}\bigr\|_{H^{\mu +1/2-j}(\Omega )}+\bigl\|U^{-(j)}\bigr\|_{H^{\mu +1/2-j}(\Omega )}+\bigl\|\varphi^{(j)}\bigr\|_{H^{\mu +1/2-j}(\partial\Omega)} \right)
\leq CM_0,
\label{59}
\end{equation}
where
\begin{equation}
M_0=\|U_0^+\|_{H^{\mu +1/2}(\Omega)}+\|U_0^-\|_{H^{\mu +1/2}(\Omega)}+\|\varphi_0\|_{H^{\mu +1/2}(\partial\Omega)},
\label{60}
\end{equation}
the constant $C>0$ depends only on $\mu$, $\|U_0^{\pm}\|_{W^{1}_{\infty}(\Omega)}$, and
$\|\varphi_0\|_{W^1_{\infty}(\partial\Omega)}$.
\label{l4.1}
\end{lemma}

The proof is almost evident and based on the multiplicative properties of Sobolev spaces.

\begin{definition}
Let $\mu\in\mathbb{N}$, $\mu \geq 3$. The initial data $((U_0^+,U_0^-),\varphi_0)\in H^{\mu +1/2}(\Omega)\times H^{\mu +1/2}(\partial\Omega )$ are said to be compatible up to order $\mu$ when $((U^{+(j)},U^{-(j)}), \varphi^{(j)})$ satisfy
\begin{equation}
\bigl[p^{(j)}\bigr]=0,\quad \bigl[v^{(j)}\bigr]=0, \quad\bigl[H_{\tau}^{(j)}\bigr]=0
\label{61}
\end{equation}
for $j=0,\ldots ,\mu$.
\label{d1}
\end{definition}

\begin{lemma}
The compatibility conditions \eqref{61} imply
\begin{equation}
\bigl[H_N^{ (j)}\bigr]=0,
\label{61'}
\end{equation}
i.e.,
\begin{equation}
\bigl[H^{(j)}\bigr]=0
\label{61"}
\end{equation}
for $j=0,\ldots ,\mu$. Moreover, \eqref{61} imply that the initial data satisfy the jump condition \eqref{jc2}:
\begin{equation}
\bigl[\partial_1{v}^{(0)}\bigr]=0.
\label{jc1a}
\end{equation}
\label{l4.2}
\end{lemma}

\begin{proof}
We prove \eqref{61'} by finite induction. Condition \eqref{61'} for $j=0$ holds because it is just one of the assumptions of Theorem \ref{t1} for the initial data. Considering systems \eqref{11.1} on the boundary $x_1=0$, using the last boundary condition in \eqref{12'}, and omitting detailed calculations (see also appendix A in \cite{T09}), we obtain
\[
\partial_t H_{N}^{\pm}|_{x_1=0}= -\left.\left(v_2^{\pm}\partial_2H_{N}^{\pm}+H_{\rm N}^{\pm}\partial_2v_2^{\pm}\right)\right|_{x_1=0} ,
\]
and then
\begin{equation}
\partial_t^{j+1} H_{N}^{\pm}|_{x_1=0}= -\sum\limits_{i=0}^{j}C_j^i\left.\left( (\partial_t^{j-i}v_2^{\pm})\,\partial_2\partial_t^iH_{N}^{\pm}
+(\partial_2\partial_t^{j-i}v_2^{\pm})\,\partial_t^{i}H_{\rm N}^{\pm}\right)\right|_{x_1=0}.
\label{A2a}
\end{equation}

That is, for $U^{\pm(j)}$ we can deduce the following counterpart of \eqref{A2a}:
\begin{equation}
H_{N}^{\pm (j+1)}\bigl|_{x_1=0}= -\sum\limits_{i=0}^{j}C_j^i\left.\left( v_2^{\pm (j-i)}\partial_2H_{N}^{\pm (i)}
+H_{\rm N}^{\pm (i)}\partial_2v_2^{\pm (j-i)}\right)\right|_{x_1=0}.
\label{A2a'}
\end{equation}
Let \eqref{61'} be satisfied for  $j=0,\ldots ,k$. Then, using $\bigl[H_{N}^{\pm (i)}\bigr]=0$ and the compatibility condition $\bigl[v_2^{(i)}\bigr]=0$ for $i=0,\ldots ,k$, from \eqref{A2a'} for $j=k$ we get $\bigl[H_{N}^{\pm (k+1)}\bigr]=0$. This completes the proof of \eqref{61'} which, together with the last compatibility condition in \eqref{61}, implies \eqref{61"}.

As in the proof of Proposition \ref{p1} (see \cite{Cont1}), we consider the equations for $p^+$ and $p^-$ contained in \eqref{11.1}, but now we consider them not only on the boundary $x_1=0$ but also at $t=0$. Taking into account \eqref{58}, we obtain
\begin{equation}
\frac{1}{\gamma (\bar{p}+p^{\pm (0)})}\bigl( p^{\pm (1)} +v_2^{\pm (0)}\partial_2p^{\pm (0)}\bigr)\pm \partial_1v_N^{\pm (0)} +\partial_2v_2^{\pm (0)}=0\quad\mbox{on}\ x_1=0,
\label{p(1)}
\end{equation}
where $v_N^{\pm (0)}=v_1^{\pm (0)}-v_2^{\pm (0)}\partial_2\Psi^{\pm (0)}$.
It follows from \eqref{61} and \eqref{p(1)} that
\begin{equation}
\bigl[ \partial_1v_N^{(0)} \bigr]=0
\label{vN(1)}
\end{equation}
(recall that the jump of a normal derivative is defined in \eqref{norm_jump}).

Considering the fourth and fifth equations  contained in \eqref{11.1} at $x_1=0$ and $t=0$, using \eqref{61} and \eqref{61"}, and passing then to the jump, we get
\begin{equation}
H_N^{+(0)}\bigl[\partial_1v^{(0)}\bigr]=H^{+(0)}\bigl[\partial_1v_N^{(0)}\bigr] \quad\mbox{on}\ x_1=0.
\label{2017}
\end{equation}
In view of \eqref{vN(1)} and assumption \eqref{mf.1} on the initial data from Theorem \ref{t1} ($H_{N|x_1=0}^{+ (0)}\neq 0$), this gives \eqref{jc1a}.
\end{proof}

\subsection{Construction of the approximate solution to problem (\ref{11.1})--(\ref{13.1})}

\begin{lemma}
Suppose the initial data \eqref{13.1} are compatible up to order $\mu$ and satisfy all of the assumptions of Theorem \ref{t1} (i.e.,
\eqref{5.1'} and \eqref{14} for all $x\in\Omega$ and \eqref{mf.1}, \eqref{RT1} and \eqref{15} for all $x\in \partial\Omega$). Then there exists a vector-function $((U^{a+},U^{a-}),\varphi^a)\in H^{\mu +1}(\Omega_T)\times H^{\mu +1}(\partial\Omega_T)$
that is further called the approximate solution to problem \eqref{11.1}--\eqref{13.1} such that
\begin{equation}
\partial_t^j\mathbb{L}(U^{a\pm},\Psi^{a\pm} )|_{t=0}=0 \quad\mbox{in}\ \Omega\ \mbox{for}\ j=0,\ldots , \mu -1,
\label{62}
\end{equation}
and it satisfies the boundary conditions \eqref{12.1} and the initial data \eqref{13.1}:
\begin{equation}
\mathbb{B}(U^{a+},U^{a-},\varphi^a )=0\quad\mbox{on}\ \partial\Omega_T,\label{12.1"}
\end{equation}
\begin{equation}
U^{a+}|_{t=0}=U^+_0,\quad U^{a-}|_{t=0}=U^-_0\quad\mbox{in}\ \Omega,
\qquad \varphi^a |_{t=0}=\varphi_0\quad \mbox{on}\ \partial\Omega,\label{13.1"}
\end{equation}
where ${\Psi}^{a\pm} =\chi (\pm x_1)\varphi^a$. The approximate solution  obeys the estimate
\begin{equation}
\|U^{a+}\|_{H^{\mu +1}(\Omega_T)}+\|U^{a-}\|_{H^{\mu +1}(\Omega_T)}+\|\varphi^a\|_{H^{\mu +1}(\partial\Omega_T)}\leq C_1(M_0)
\label{63}
\end{equation}
where $C_1=C_1(M_0)>0$ is a constant depending on $M_0$ (see \eqref{60}). Moreover, the approximate solution satisfies the boundary constraint \eqref{15} and the jump condition \eqref{jc2}. It also satisfies the hyperbolicity conditions \eqref{5.1'} in $\Omega_T^+$ and requirement \eqref{mf.1} together with the Rayleigh-Taylor sign condition \eqref{RT1} on $\partial\Omega_T^+$.
\label{l2}
\end{lemma}

\begin{proof}
Given the initial data, let us take $U^{\pm (j)}$ and $\varphi^{(j)}$ as in Lemma \ref{l4.1}. We first take $U^{a\pm}=(p^{a\pm}, v^{a\pm},H^{a\pm},S^{a\pm})\in H^{\mu +1}(\mathbb{R}\times\Omega)$ and $\varphi^a \in H^{\mu +1}(\mathbb{R}\times\partial\Omega)$ such that
\begin{align*}
\partial_t^jU^{a\pm}|_{t=0} &=U^{\pm (j)}\in H^{\mu +1/2-j}(\Omega),   \\
\partial_t^j\varphi^a|_{t=0} &=\varphi^{(j)}\in H^{\mu +1/2-j}(\partial\Omega )
\end{align*}
for $j=0,\ldots ,\mu $.  We can choose $U^{a\pm}$ and $\varphi^a$ that they satisfy the boundary conditions \eqref{12.1}/\eqref{12.1"} together with the boundary constraint \eqref{15} and the jump condition \eqref{jc2}:
\begin{equation}
\bigl[p^{a}\bigr]=0,\quad \bigl[v^{a}\bigr]=0,\quad \bigl[H^{a}\bigr]=0, \quad \partial_t\varphi^a=v_{N}^{a+}, \quad
\bigl[\partial_1v^{a}\bigr]=0 \quad \mbox{on}\ \mathbb{R}\times\partial\Omega ,
\label{12.1a}
\end{equation}
where $v_{N}^{a\pm}=v_1^{a\pm}-v_2^{a\pm}\partial_2\Psi^{a\pm}$. Such a lifting is possible thanks to the compatibility conditions  \eqref{61} (see also \eqref{58}) and their consequence \eqref{jc1a} (see, e.g., \cite{Maj}). Clearly, we can restrict such constructed functions to the time interval $(-\infty, T]$, i.e., $U^{a\pm}\in H^{\mu +1}(\Omega_T)$ and $\varphi^a \in H^{\mu +1}(\partial\Omega_T)$

That is, we have already proved that the approximate solution satisfies \eqref{12.1"}, \eqref{13.1"}, \eqref{15} and \eqref{jc2}. Applying Sobolev's embeddings, we rewrite estimate \eqref{59} as
\begin{equation}
\sum_{j=1}^{\mu}\left( \bigl\|U^{+(j)}\bigr\|_{H^{\mu +1/2 -j}(\Omega )}+\bigl\|U^{-(j)}\bigr\|_{H^{\mu +1/2-j}(\Omega )}+\bigl\|\varphi^{(j)}\bigr\|_{H^{\mu +1/2-j}(\partial\Omega)} \right)
\leq C(M_0),
\label{64}
\end{equation}
where $C=C(M_0)>0$ is a constant depending on $M_0$. Estimate \eqref{63} follows from \eqref{64} and the continuity of the lifting operators from the hyperplane $t=0$ to $\mathbb{R}\times\Omega$. Conditions \eqref{62} hold thanks to the properties of $U^{\pm (j)}$ and $\varphi^{(j)}$ given by Lemma \ref{l4.1}. At last, since, thanks to the assumptions on the initial data, $U^{a\pm}$, $\varphi^a$ satisfy
the hyperbolicity conditions  \eqref{5.1'}, requirement \eqref{mf.1} and  the Rayleigh-Taylor sign condition \eqref{RT1} at the initial time $t=0$, in the above procedure we can choose $U^{a\pm}$, $\varphi^a$ that satisfy \eqref{5.1'}, \eqref{mf.1} and   \eqref{RT1} for all $t\in [0,T]$.
\end{proof}

Without loss of generality we can suppose that
\begin{equation}
\|U_0^+\|_{H^{\mu +1/2}(\Omega)}+\|U_0^-\|_{H^{\mu +1/2}(\Omega)}+\|\varphi_0\|_{H^{\mu +1/2}(\partial\Omega)}\leq 1,\quad
\|\varphi_0\|_{H^{\mu +1/2}(\partial\Omega )}\leq 1/2.
\label{65}
\end{equation}
Then for a sufficiently short time interval $[0,T]$ the smooth solution whose existence we are going to prove satisfies
$\|\varphi\|_{L_{\infty}([0,T]\times\partial\Omega)}\leq 1$, which implies $\partial_1\Phi^+\geq 1/2$ and $\partial_1\Phi^-\leq -1/2$
(recall that $\|\chi'\|_{L_{\infty}(\mathbb{R})}<1/2$, see Section \ref{s1}). Let $\mu$ is an integer number that will appear in the regularity assumption for the initial data in the existence theorem for problem \eqref{11.1}--\eqref{13.1}. Running ahead, we take $\mu=m+8$, with $m\geq 6$ (see Theorem \ref{t1}). In the end of the next section  we will see that this choice is suitable. Taking into account \eqref{65}, we rewrite \eqref{63} as
\begin{equation}
\|U^{a+}\|_{H^{m +9}(\Omega_T)}+\|U^{a-}\|_{H^{m +9}(\Omega_T)}+\|\varphi^a\|_{H^{m +9}(\partial\Omega_T)}\leq C_*,
\label{66}
\end{equation}
where $C_*=C_1(1)$.

Let us introduce
\begin{equation}
f^{a\pm}:=\left\{ \begin{array}{lr}
- \mathbb{L}(U^{a\pm},{\Psi}^{a\pm} ) & \quad \mbox{for}\ t>0,\\
0 & \ \mbox{for}\ t<0.\end{array}\right.
\label{67}
\end{equation}
Since $((U^{a+},U^{a-}),\varphi^a)\in H^{m +9}(\Omega_T)\times H^{m +9}(\partial\Omega_T )$, taking into account \eqref{62}, we get $f^{a\pm} \in H^{m +8}(\Omega_T)$ and
\begin{equation}
\|f^{a\pm}\|_{H^{m + 8}(\Omega_T)}\leq \delta_0 (T),
\label{68}
\end{equation}
where the constant $\delta_0(T)\rightarrow 0$ as $T\rightarrow 0$. To prove estimate \eqref{68} we use the Moser-type and embedding inequalities and the fact that $f^{a+}$ and $f^{a-}$ vanish in the past.

Given the approximate solution defined in Lemma \ref{l2}, $((U^+,U^-) ,\varphi)= ((U^{a+},U^{a-}),\varphi^a)+ ((\widetilde{U}^+, \widetilde{U}^-) ,\tilde{\varphi})$ is a solution of the original problem \eqref{11.1}--\eqref{13.1}
on $[0,T]\times\Omega$ if $((\widetilde{U}^+, \widetilde{U}^-) ,\tilde{\varphi})$ satisfies the following problem on $\Omega_T$ (tildes are dropped):
\begin{align}
 \mathcal{ L}(U^{\pm} ,{\Psi}^{\pm})=f^{a\pm} &\quad\mbox{in}\ \Omega_T, \label{69'}\\[3pt]
 \mathcal{ B}(U^+,U^- ,\varphi )=0 &\quad\mbox{on}\ \partial\Omega_T,\label{70'}
\\[3pt]
 ((U^+,U^-),\varphi )=0 &\quad  \mbox{for}\ t<0,\label{71'}
\end{align}
where
\begin{align*}
 &\mathcal{ L}(U^{\pm} ,{\Psi}^{\pm} ):=\mathbb{L}(U^{a\pm} +U^{\pm}  ,{\Psi}^{a\pm}+{\Psi}^{\pm} ) -
\mathbb{L}(U^{a\pm} ,{\Psi}^{a\pm}),\\
 &\mathcal{ B}(U^+,U^- ,\varphi):=\mathbb{B}(U^{a+} +U^+ ,U^{a-} +U^-,\varphi^a+\varphi ).
\end{align*}
Below it will be convenient to use the  notations
\begin{equation}
U:=\begin{pmatrix} U^+ \\ U^- \end{pmatrix},\quad \mathcal{ L}(U ,{\Psi}):= \begin{pmatrix} \mathcal{ L}(U^+ ,{\Psi}^+) \\ \mathcal{ L}(U^- ,{\Psi}^-)\end{pmatrix},\quad f^a:=\begin{pmatrix} f^{a+} \\ f^{a-} \end{pmatrix}.
\label{notations}
\end{equation}
With these notations problem \eqref{69'}--\eqref{71'} reads:
\begin{align}
 \mathcal{ L}(U ,{\Psi})=f^{a} &\quad\mbox{in}\ \Omega_T, \label{69}\\[3pt]
\mathcal{ B}(U ,\varphi )=0 &\quad\mbox{on}\ \partial\Omega_T,\label{70}
\\[3pt]
(U,\varphi )=0 &\quad  \mbox{for}\ t<0.\label{71}
\end{align}
From now on we concentrate on the proof of the existence of solutions to problem \eqref{69}--\eqref{71}.

\section{Nash-Moser iteration}
\label{s5}

\setcounter{subsubsection}{0}

\subsection{Iteration scheme for solving problem (\ref{69})--(\ref{71})}

We solve problem \eqref{69}--\eqref{71} by a suitable Nash-Moser-type iteration scheme. In short, this scheme is a modified Newton's scheme, and at each Nash-Moser iteration step we smooth the coefficient $u_n$ of a corresponding linear problem for $\delta u_n =u_{n+1}-u_n$. Errors of a classical Nash-Moser iteration are the ``quadratic'' error of Newton's scheme and the ``substitution'' error caused by the application of smoothing operators $S_{\theta}$  (see, e.g., \cite{Al,Herm,Sec16} and references therein). As, for example, in \cite{ST,T09,Tcpam}, in our case the Nash-Moser procedure is not completely standard and we have the additional error caused by the introduction of an intermediate (or modified) state $u_{n+1/2}$ satisfying some constraints. In our case, these constraints are  \eqref{a5}--\eqref{jc1'}, and the Rayleigh-Taylor sign condition \eqref{RTL}, which were required to be fulfilled for the basic state \eqref{a21}. Also, the additional error is caused by dropping the zeroth-order terms in $\Psi^{\pm}$ in the  linearized interior equations written in terms of the ``good unknown'' (see \eqref{24}).

The tame a priori estimate is the main tool for proving the convergence of the Nash-Moser iteration scheme and so it specifies main features of the scheme for a concrete problem. Since our tame a priori estimate \eqref{38} for the linearized problem is internally the same as that in \cite{Tcpam}, the realization of the Nash-Moser procedure for problem \eqref{69}--\eqref{71} below is almost the same as in \cite{Tcpam} for the free boundary problem for the compressible Euler equations with a vacuum boundary condition. Therefore, we may be very brief here and just refer to \cite{Tcpam} where it is possible and where it is convenient for the reader. The only place which requires special attention is the construction of the modified state $(U_{n+1/2},\varphi_{n+1/2})$ because our constraints on it (especially, \eqref{jc1'}) are more involved as those in  \cite{Tcpam}.

Now, following \cite{Tcpam}, we describe the iteration scheme for problem \eqref{69}--\eqref{71}. We first list the important properties of smoothing operators \cite{Al,Herm,Sec16}.

\begin{proposition}
There exists such a family $\{S_{\theta}\}_{\theta\geq 1}$ of smoothing operators in $H^s(\Omega_T)$ acting on the class of functions vanishing in the past that
\begin{align}
  \|   S_{\theta}u\|_{H^{\beta}(\Omega_T) }\leq C\theta^{(\beta-\alpha )_+}\|u\|_{H^{\alpha}(\Omega_T) },\quad &    \alpha ,\beta \geq 0,
\label{72}\\
  \|  S_{\theta}u-u\|_{H^{\beta}(\Omega_T) }\leq C\theta^{\beta-\alpha }\|u\|_{H^{\alpha}(\Omega_T) },\quad & 0 \leq \beta \leq \alpha ,  \label{73}\\
  \Bigl\|  \frac{d}{d\theta}S_{\theta}u\Bigr\|_{H^{\beta}(\Omega_T) }\leq C\theta^{\beta-\alpha -1}\|u\|_{H^{\alpha}(\Omega_T) },\quad    & \alpha  ,\beta \geq 0, \label{74}
\end{align}
where $C>0$ is a constant, and $(\beta-\alpha )_+:=\max (0,\beta -\alpha )$.  Moreover, there is another family of smoothing operators (still denoted
$S_{\theta}$) acting on functions defined on the boundary $\partial\Omega_T$ and meeting properties \eqref{72}--\eqref{74} with the norms $\|\cdot \|_{H^{\alpha}(\partial\Omega_T)}$.
\label{p1a}
\end{proposition}

We choose
\[
U_0 =0,\quad \varphi_0=0
\]
and assume that $(U_k, \varphi_k)$ are already given for $k=0,\ldots ,n$. Moreover, let $(U_k,\varphi_k)$ vanish in the past, i.e., they satisfy \eqref{71}.  We define
\[
U_{n+1}=U_n+\delta U_n,\quad \varphi_{n+1}=\varphi_n+\delta \varphi_n,
\]
where the differences $\delta{U}_n$ and $\delta \varphi_n$ solve the linear problem
\begin{equation}
\left\{\begin{array}{lr}
\mathbb{L}'_e({U}^a +{U}_{n+1/2} ,{\Psi}^a+{\Psi}_{n+1/2})\delta\dot{{U}}_n={f}_n &\quad\mbox{in}\ \Omega_T, \\[6pt]
\mathbb{B}'_{n+1/2}(\delta\dot{{U}}_n,\delta {\varphi}_n)={g}_n  & \quad\mbox{on}\ \partial\Omega_T,\\[6pt]
(\delta\dot{{U}}_n,\delta \varphi_n)=0 &\quad \mbox{for}\ t<0.
\end{array}\right.
\label{75}
\end{equation}
Here
\begin{equation}
\delta\dot{U}_n:= \delta{U}_n-\frac{\delta\Psi_n}{\partial_1(\Phi^a+\Psi_{n+1/2})}\,\partial_1({U}^a+{U}_{n+1/2})
\label{76}
\end{equation}
is the ``good unknown'' (cf. \eqref{b23}),
\[
\mathbb{B}'_{n+1/2}:=\mathbb{B}'_e(({U}^a +{U}_{n+1/2})|_{x_1=0} ,\varphi^a+\varphi_{n+1/2}),
\]
the operators $\mathbb{L}'_e$ and $\mathbb{B}'_e$ are defined in \eqref{24}--\eqref{27}, and $({U}_{n+1/2},\varphi_{n+1/2})$ is a smooth modified state such that $({U}^a +{U}_{n+1/2},\varphi^a+\varphi_{n+1/2})$ satisfies constraints  \eqref{a5}--\eqref{jc1'}, \eqref{RTL} ($\Psi_n$, ${\Psi}_{n+1/2}$, and $\delta\Psi_n$ are associated to $\varphi_n$, $\varphi_{n+1/2}$, and $\delta\varphi_n$ like ${\Psi}$ is associated to $\varphi$; $\Psi =(\Psi^+,\Psi^-)$, $\Psi_n=(\Psi_n^+,\Psi_n^-)$, $U^a=(U^{a+},U^{a-})$,  $U_n=(U_n^{+},U_n^{-})$, etc.). Moreover, here (see \eqref{76}) and below for short we use the notations like
\[
\frac{a}{\partial_1(\Phi^a+\Psi_{n+1/2})}\,b:=\begin{pmatrix} \displaystyle \frac{a^+}{\partial_1(\Phi^{a+}+\Psi^+_{n+1/2})}\,b^+\\ \displaystyle\frac{a^-}{\partial_1(\Phi^{a-}+\Psi^-_{n+1/2})}\,b^-\end{pmatrix}
\]
for some vectors $a=(a^+,a^-)$ and $b=(b^+,b^+)$ (we hope that such ``strange'' notations are clear from the context).
The right-hand sides ${f}_n$ and ${g}_n$  are defined through the accumulated errors at the step $n$.

The errors of the iteration scheme are defined from the following chains of decompositions:
\begin{align*}
& \mathcal{ L}({U}_{n+1} ,{\Psi}_{n+1})-\mathcal{ L}({U}_{n} ,{\Psi}_{n})\\
& \quad =\mathbb{L}'({U}^a +{U}_{n} ,{\Psi}^a+{\Psi}_{n})(\delta{{U}}_n,\delta{\Psi}_{n})+{e}'_n\\
& \quad =\mathbb{L}'({U}^a +S_{\theta_n}{U}_{n} ,{\Psi}^a+S_{\theta_n}{\Psi}_{n})(\delta{{U}}_n,\delta{\Psi}_{n})+{e}'_n+{e}''_n
\\
&\quad= \mathbb{L}'({U}^a +{U}_{n+1/2} ,{\Psi}^a+{\Psi}_{n+1/2})(\delta{{U}}_n,\delta{\Psi}_{n})+{e}'_n+{e}''_n+{e}'''_n\\
& \quad =\mathbb{L}'_e({U}^a +{U}_{n+1/2},{\Psi}^a+{\Psi}_{n+1/2})\delta\dot{{U}}_n+{e}'_n+{e}''_n+{e}'''_n+{D}_{n+1/2}\delta{\Psi}_{n}
\end{align*}
and
\begin{align*}
&\mathcal{ B}({U}_{n+1}|_{x_1=0},\varphi_{n+1})-\mathcal{ B}({U}_{n}|_{x_1=0},\varphi_{n})
\\ &\quad =\mathbb{B}'(({U}^a +{U}_{n})|_{x_1=0},\varphi^a+\varphi_{n})(\delta{{U}}_n|_{x_1=0},\delta \varphi_{n})+\tilde{e}'_n
\\ &\quad =\mathbb{B}'(({U}^a +S_{\theta_n}{U}_{n})|_{x_1=0},\varphi^a+S_{\theta_n}\varphi_{n})(\delta{{U}}_n|_{x_1=0},\delta \varphi_{n})+\tilde{e}'_n+\tilde{e}''_n
\\ &\quad =\mathbb{B}'_{n+1/2}(\delta\dot{{U}}_n,\delta\varphi_n)+\tilde{e}'_n+\tilde{e}''_n+\tilde{e}'''_n,
\end{align*}
where $S_{\theta_n}$ are smoothing operators enjoying the properties of Proposition \ref{p1a}, with the sequence $(\theta_n)$ defined by
\[
\theta_0\geq 1,\quad \theta_n=\sqrt{\theta_0+n} ,
\]
and we use the notation
\[
{D}_{n+1/2}:= \frac{1}{\partial_1(\Phi^a+\Psi_{n+1/2})}\,\partial_1\left\{ \mathbb{L}(U^a +U_{n+1/2} ,{\Psi}^a+{\Psi}_{n+1/2})\right\}
\]
as well as the notations like
\[
\mathbb{L}(U ,{\Psi}):=
\left(
\begin{array}{c}
\mathbb{L}({U}^{+} ,{\Psi}^{+})\\[3pt]
\mathbb{L}({U}^{-} ,{\Psi}^{-})
\end{array}
\right),
\]
\[
\mathbb{L}'({U}^a +{U}_{n} ,{\Psi}^a+{\Psi}_{n})(\delta{{U}}_n,\delta{\Psi}_{n}):=
\left(
\begin{array}{c}
\mathbb{L}'({U}^{a+} +{U}^+_{n} ,{\Psi}^{a+}+{\Psi}^+_{n})(\delta{{U}}^+_n,\delta{\Psi}^+_{n})\\[3pt]
\mathbb{L}'({U}^{a-} +{U}^-_{n} ,{\Psi}^{a-}+{\Psi}^-_{n})(\delta{{U}}^-_n,\delta{\Psi}^-_{n})
\end{array}
\right),
\]
etc.
The errors ${e}'_n$ and  $\tilde{e}'_n$ are the usual quadratic errors of Newton's method, and ${e}''_n$,
$\tilde{ e}''_n$ and ${e}'''_n$, $\tilde{e}'''_n$ are the first and the second substitution errors
respectively.

Let
\begin{equation}
{e}_n:={e}'_n+{e}''_n+{e}'''_n+{D}_{n+1/2}\delta{\Psi}_{n}, \quad
\tilde{e}_n:= \tilde{e}'_n+\tilde{e}''_n+\tilde{e}'''_n,
\label{77}
\end{equation}
then the accumulated errors at the step $n\geq 1$ are
\begin{equation}
{E}_n=\sum_{k=0}^{n-1}{e}_k,\quad \widetilde{E}_n=\sum_{k=0}^{n-1}\tilde{e}_k,
\label{78}
\end{equation}
with ${E}_0:=0$ and $\widetilde{E}_0:=0$. The right-hand sides ${f}_n$ and ${g}_n$ are recursively computed from the equations
\begin{equation}
\sum_{k=0}^{n}{f}_k+S_{\theta_n}{E}_n=S_{\theta_n}{f}^a,\quad
\sum_{k=0}^{n}{g}_k+S_{\theta_n}\widetilde{E}_n=0,
\label{79}
\end{equation}
where ${f}_0:=S_{\theta_0}{f}^a$ and ${g}_0:=0$. Since $S_{\theta_N}\rightarrow I$ as $N\rightarrow \infty$, one can show that we formally obtain
the solution to problem \eqref{69}--\eqref{71} from $\mathcal{ L}(U_{N} ,{\Psi}_{N})\rightarrow {f}^a$ and
$\mathcal{ B}(U_{N}|_{x_1=0},\varphi_{N})\rightarrow 0$, provided that $({e}_N,\tilde{e}_N)\rightarrow 0$.

Below we closely follow the plan of \cite{Tcpam}. Let us first formulate our inductive hypothesis, which is actually the same as in \cite{Tcpam}.

\subsection{Inductive hypothesis}

Given a small number $\delta >0$,\footnote{We use the same Greek letter $\delta $ as in the differences $\delta{U}_n$ and $\delta \varphi_n$ above. But we hope that this will not lead to confusion because from the context it is always clear that $\delta $ written before ${U}_n$ or $\varphi_n$ is not a multiplier.} the integer $\alpha :=m+1$, and an integer $\tilde{\alpha}$, our inductive hypothesis reads:
\[
(H_{n-1})\quad \left\{
\begin{array}{ll}
{\rm a})\;& \forall\, k=0,\ldots , n-1,\quad \forall s\in [3,\tilde{\alpha}]\cap\mathbb{N},\\[3pt]
 & \|\delta U_k\|_{H^s(\Omega_T)} +\|\delta \varphi_k\|_{H^s(\partial\Omega_T)}\leq \delta\theta_k^{s-\alpha -1}\Delta_k,\\[6pt]
{\rm b}) & \forall\, k=0,\ldots , n-1,\quad \forall s\in [3,\tilde{\alpha}-2]\cap\mathbb{N},\\[3pt]
 & \|\mathcal{ L}(U_k,{\Psi}_k)-f^a\|_{H^s(\Omega_T)}\leq 2\delta\theta_k^{s-\alpha -1},\\[6pt]
{\rm c}) & \forall\, k=0,\ldots , n-1,\quad \forall s\in [4,\alpha ]\cap\mathbb{N},\\[3pt]
 & \|\mathcal{ B}(U_k|_{x_1=0},\varphi_k)\|_{H^s(\partial\Omega_T)}\leq \delta\theta_k^{s-\alpha -1},
\end{array}\right.
\]
where $\Delta_k= \theta_{k+1}-\theta_k$. Note that the sequence $(\Delta_n)$ is decreasing and tends to zero, and
\[
\forall\, n\in\mathbb{N},\quad \frac{1}{3\theta_n}\leq\Delta_n=\sqrt{\theta_n^2+1} -\theta_n\leq \frac{1}{2\theta_n}.
\]
Recall that $(U_k,\varphi_k)$ for $k=0,\ldots ,n$ are also assumed to satisfy (\ref{71}). Looking a few steps forward, we observe that we will need to use inequalities (\ref{66}) and (\ref{68}) with $m=\tilde{\alpha}-5$. That is, as in \cite{Tcpam}, we now choose $\tilde{\alpha}=m+5$, i.e., $\tilde{\alpha}=\alpha +4$.\footnote{The reader going inside \cite{Tcpam} should take into account that there is a misprint there: it is written that $\tilde{\alpha}=m+4$ but not $\tilde{\alpha}=m+5$ (whereas the formula $\tilde{\alpha}=\alpha +4$ is correct in \cite{Tcpam}).}
Our goal is to prove that ($H_{n-1}$) implies ($H_n$) for a suitable choice of parameters $\theta_0\geq 1$ and $\delta >0$, and for a sufficiently short time $T>0$. After that we shall prove ($H_0$). From now on we assume that ($H_{n-1}$) holds.  As in \cite{Tcpam}, we have the following consequences of ($H_{n-1}$).

\begin{lemma}
If $\theta_0$ is big enough, then for every $k=0,\ldots ,n$ and for every integer $s\in [3,\tilde{\alpha}]$ we have
\begin{align}
\|{U}_k \|_{H^s(\Omega_T)}+\| \varphi_k\|_{H^{s}(\partial\Omega_T)}\leq\delta\theta_k^{(s-\alpha )_+}, &\quad \alpha\neq s,\label{80}\\[3pt]
\|{U}_k \|_{H^{\alpha}(\Omega_T)}+\| \varphi_k\|_{H^{\alpha}(\partial\Omega_T)}\leq\delta\log \theta_k, & \label{81}
\end{align}
\begin{equation}
\|(I-S_{\theta_k}){U}_k \|_{H^s(\Omega_T)}+\|(1-S_{\theta_k}) \varphi_k\|_{H^{s}(\partial\Omega_T)}\leq C\delta\theta_k^{s-\alpha }.\label{82}
\end{equation}
For every $k=0,\ldots ,n$ and for every integer $s\in [3,\tilde{\alpha}+4]$ we have
\begin{align}
\|S_{\theta_k}{U}_k \|_{H^s(\Omega_T)}+\| S_{\theta_k} \varphi_k\|_{H^{s}(\partial\Omega_T)}\leq C\delta\theta_k^{(s-\alpha )_+}, &\quad \alpha\neq s,\label{83}\\[3pt]
\|S_{\theta_k}{U}_k \|_{H^{\alpha}(\Omega_T)}+\| S_{\theta_k}\varphi_k\|_{H^{\alpha}(\partial\Omega_T)}\leq C\delta\log \theta_k. &  \label{84}
\end{align}
Moreover, \eqref{82} and \eqref{83} hold for every integer $s\geq 3$.
\label{l3}
\end{lemma}

Note that estimates \eqref{82}--\eqref{84} follow from \eqref{80}, \eqref{81}, and Proposition \ref{p1a}. We also have the following important consequence of ($H_{n-1}$) connected with the jump condition \eqref{jc2}.

\begin{lemma}
Let $\alpha \geq 6$. There exist $\delta >0$ and $T>0$ sufficiently small, and $\theta_0 \geq 1$ sufficiently large such that
for every $k=0,\ldots ,n-1$, and for every integer $s\in [3,{\alpha}-1]$, we have
\begin{equation}
\left\|[\partial_1{\bf v}_k]\right\|_{H^s(\partial\Omega_T)}\leq C\delta\theta_k^{s-\alpha},
\label{Hn-1dv1}
\end{equation}
\begin{equation}
\left\|[H_{N,k}]\right\|_{H^s(\partial\Omega_T)}\leq C\delta\theta_k^{s-\alpha},
\label{Hnk}
\end{equation}
where ${\bf v}_{k}^{\pm}= (v_{1,k}^{\pm},v_{2,k}^{\pm})$, $[\partial_1{\bf v}_{k}]= \partial_1{\bf v}_{k}^+|_{x_1=0}+\partial_1{\bf v}_{k}^-|_{x_1=0}$ (cf. \eqref{norm_jump}), $H_{N,k}^{\pm}:=H_{1,k}^{\pm}+H_1^{a\pm}-(H_{2,k}^{\pm}+H_2^{a\pm})\partial_2(\Psi_k^{\pm}+\Psi^{a\pm})$,\footnote{In view of \eqref{12.1a}, $[ H_{N,k}]= [H_{1,k}]-[H_{2,k}]\partial_2(\varphi_k+\varphi^a)$.} and $v_{j,k}^{\pm}$ and $H_{j,k}^{\pm}$ ($j=1,2$) are the components of the velocities and the magnetic fields at the step $k$ respectively, i.e., $U_k^{\pm}=(p_k^{\pm},v_{1,k}^{\pm},v_{2,k}^{\pm},H_{1,k}^{\pm},H_{2,k}^{\pm},S_{k}^{\pm})$.
\label{ldv1}
\end{lemma}

\begin{proof}
Let $\mathbb{P}(U^{\pm},\Psi^{\pm})$ be the first lines of $\mathbb{L}(U^{\pm},\Psi^{\pm})$. We need them on the boundary:
\[
\mathbb{P}(U^{\pm},\Psi^{\pm})|_{x_1=0}=\left.\left(
\frac{1}{\gamma (\bar{p}+p^{\pm })}\left\{ \partial_tp^{\pm} \pm (v_N^{\pm}- \partial_t\varphi )\partial_1p^{\pm} +v_2^{\pm }\partial_2p^{\pm }\right\}\pm \partial_1v_N^{\pm } +\partial_2v_2^{\pm }\right)\right|_{x_1=0}.
\]
Let $r_k^{\pm}:=\mathbb{P}(U^{a\pm}+U_k^{\pm},\Psi^{a\pm}+\Psi_k^{\pm})|_{x_1=0}$. Point b) of ($H_{n-1}$) implies
\begin{equation}
\|r_k^{\pm}\|_{H^s(\partial\Omega_T)}\leq C \|\mathbb{P}(U^{a\pm}+U_k^{\pm},\Psi^{a\pm}+\Psi_k^{\pm})\|_{H^{s+1}(\Omega_T)}\leq C\delta\theta_k^{s-\alpha},\quad \forall s\in [3,{\alpha}+1]\cap\mathbb{N}
\label{rk}
\end{equation}
for all $k=0,\ldots ,n-1$ (recall that $\tilde{\alpha}=\alpha +4$). Taking into account \eqref{12.1a}, we have
\begin{equation}
\begin{split}
[ \partial_1v_{N,k}] = &[  r_k]-\partial_2[v_{2,k}]
-\biggl[
\frac{1}{\gamma (\bar{p}+p^a+p_k)}\bigl\{
\partial_t(p^a+p_k) \pm \bigl(v_{1,k}- (v_2^a+v_{2,k})\partial_2\varphi_k \\  & -v_{2,k}\partial_2\varphi^a-\partial_t\varphi_k \bigr)\partial_1(p^a+p_k)  +(v_2^a+v_{2,k})\partial_2(p^a+p_k)\bigr\}\biggr],
\end{split}
\label{p1vNk}
\end{equation}
where $v_{N,k}^{\pm}:=v_{1,k}^{\pm}+v_1^{a\pm}-(v_{2,k}^{\pm}+v_2^{a\pm})\partial_2(\Psi_k^{\pm}  +\Psi^{a\pm}$) (i.e., in view of \eqref{12.1a}, $[ \partial_1v_{N,k}]= [\partial_1v_{1,k}]-[\partial_1v_{2,k}]\partial_2(\varphi_k+\varphi^a)$). Using point c) of ($H_{n-1}$), we easily estimate the second term in the right-hand side of \eqref{p1vNk}:
\begin{equation}
\|\partial_2[v_{2,k}] \|_{H^s(\partial\Omega_T)}\leq  \|\mathcal{ B}(U_k|_{x_1=0},\varphi_k)\|_{H^{s+1}(\partial\Omega_T)}\leq \delta\theta_k^{s-\alpha},\quad \forall s\in [3,{\alpha}-1]\cap\mathbb{N}.
\label{pv2k}
\end{equation}
Note that in point b) of ($H_{n-1}$) it is supposed by default that the hyperbolicity conditions $\bar{p}+p^{a\pm}+p_k^{\pm} \geq\bar{p}/2>0$ are satisfied (cf. \eqref{a5}). This is true for a sufficiently short time $T$ thanks to the conditions $p^{a\pm}>-\bar{p}/4$ (see Lemma \ref{l2}) and the fact that $U_k$ vanish in the past. That is, the denominator $\bar{p}+p^a+p_k$ appearing in \eqref{p1vNk} is strictly positive for a sufficiently short time $T$.

Consider, for example,  the last term in the right-hand side of \eqref{p1vNk}. Using \eqref{12.1a}, we can decompose it as follows:
\begin{equation}
-\biggl[
\frac{v_2^a+v_{2,k}}{\gamma (\bar{p}+p^a+p_k)} \partial_2(p_k+p^a)\biggr]= b_1(y_k,z_k)[v_{2,k}]+b_2(y_k,z_k)[p_k]+b_3(y_k)\partial_2[p_k],
\label{p1vNk'}
\end{equation}
where $y_k=(U^a+U_k)|_{x_1=0}$, $z_k=\partial_2(U^a+U_k)|_{x_1=0}$, and $b_i$ ($i=1,2,3$) are rational functions which can be easily written down. Note that the graphs of these functions pass through the origin ($b_1(0,0)=b_2(0,0)=b_3(0)=0$) and, hence, we can use the calculus inequality \eqref{c41} when estimating their Sobolev's norms. Using \eqref{66}, point c) of ($H_{n-1}$), inequalities \eqref{80} and  \eqref{81}, the calculus inequalities \eqref{c40} and \eqref{c41}, Sobolev's embeddings and the trace theorem, we now estimate the first term in the right-hand side of \eqref{p1vNk'} (in comparison with other terms, for its estimation we need the most restrictive requirement on $\alpha$):
\[
\begin{split}
\bigl\| & b_1(y_k,z_k)[v_{2,k}]\bigr\|_{H^s(\partial\Omega_T)}  \\ & \leq
CC_0(M)\left\{ \|[v_{2,k}]\|_{H^4(\partial\Omega_T)}\|U_k +U^a\|_{H^{s+2}(\Omega_T)} +
\|[v_{2,k}]\|_{H^{s+1}(\partial\Omega_T)}\| U_k +U^a\|_{H^{3}(\Omega_T)}\right\} \\
& \leq C\left\{ \delta \theta_k^{3-\alpha} \big(C_*+\delta \theta_k^{\ell (s,\alpha )}\big) +\delta\theta_k^{s-\alpha }(C_*+\delta )\right\}
\leq C\delta \theta_k^{s-\alpha}
\end{split}
\]
for $\alpha \geq 6$ and integer $s\in [3,{\alpha}-1]$, where $\ell (s,\alpha )=(s+2-\alpha)_+$ for $s\neq \alpha -2$ and $\ell (\alpha -2,\alpha )=1$, and the constant $C_0=C_0(M)$ depends on $M=\|U^a+U_k \|_{H^{3}(\Omega_T)}\leq C_*+\delta$. Note that the requirement $\alpha \geq 6$ is in agreement with
$\alpha = m+1 $ and the assumption $m\geq 6$ of Theorem \ref{t1}. We can similarly estimate the remaining terms in the right-hand side of \eqref{p1vNk'} as well as the remaining terms in the right-hand side of  \eqref{p1vNk}. Then, from the obtained estimates and \eqref{rk}--\eqref{pv2k} we get the estimate
\begin{equation}
\left\|[\partial_1v_{N,k}]\right\|_{H^s(\partial\Omega_T)}\leq C\delta\theta_k^{s-\alpha}
\label{Hn-1dv1'}
\end{equation}
for integer $s\in [3,{\alpha}-1]$.

Let $\mathbb{H}(v^{\pm},H^{\pm},\Psi^{\pm})$ be the fourth and fifth lines of $\mathbb{L}(U^{\pm},\Psi^{\pm})$ respectively.
For deriving the desired estimate \eqref{Hnk} we consider $\mathbb{H}(v^{a\pm}+v_k^{\pm},H^{a\pm}+H_k^{\pm},\Psi^{a\pm}+\Psi_k^{\pm})|_{x_1=0}$  and exploit the obtained estimate \eqref{Hn-1dv1'} together with points b) and c) of  ($H_{n-1}$). We prefer here to omit detailed arguments and calculations because they are really similar to those towards the proof of estimate \eqref{Hn-1dv1'}. We only note that, as in the proof of Proposition \ref{p2} in \cite{Cont1} (see also \eqref{2017}), it is important that the normal components of the magnetic fields ($H^{a\pm}+H_k^{\pm}$ in our case) do not vanish at the boundary $x_1=0$. This is so for a sufficiently short time $T$,
\[
\begin{split}
\bigl|(H_1^a& +H_{1,k}^{\pm})|_{x_1=0}-(H_2^a+H_{2,k}^{\pm})|_{x_1=0}\partial_2(\varphi^a+\varphi_k )\bigr| \\ & =\bigl|H_N^{a\pm}|_{x_1=0}+H_{1,k}^{\pm}|_{x_1=0} -
(H_2^a+H_{2,k}^{\pm})|_{x_1=0}\partial_2\varphi_k-H_{2,k}^{\pm}|_{x_1=0}\partial_2\varphi^a\bigr| \geq\kappa/2 >0,
\end{split}
\]
thanks to the requirements $|H_N^{a\pm}|_{x_1=0}|\geq\kappa >0$ (see \eqref{mf.1} and Lemma \ref{l2}) satisfied for the approximate solution and the fact that $(U_k,\varphi_k)$ vanish in the past.

As in the proof of Proposition \ref{p1} in \cite{Cont1}, we consider $\mathbb{H}(v^{\pm},H^{\pm},\Psi^{\pm})$ on the boundary:
\[
\mathbb{H}(v^{\pm},H^{\pm},\Psi^{\pm})|_{x_1=0}=
\left.\left(\partial_t H_{N}^{\pm} \pm (v_N^{\pm}- \partial_t\varphi )\partial_1H_{N}^{\pm}+v_2^{\pm}\partial_2H_{N}^{\pm}+
H_{N}^{\pm}\partial_2v_2^{\pm} \right)\right|_{x_1=0}.
\]
Similarly to \eqref{rk}, we estimate $\tilde{r}_k^{\pm}:=\mathbb{H}(v^{a\pm}+v_k^{\pm},H^{a\pm}+H_k^{\pm},\Psi^{a\pm}+\Psi_k^{\pm})|_{x_1=0}$:
\[
\|\tilde{r}_k^{\pm}\|_{H^s(\partial\Omega_T)}\leq C\delta\theta_k^{s-\alpha}
\]
for  $s\in [3,{\alpha}+1]$. We obtain that the jumps $b_k=[H_{N,k}]$ for $k=0,\ldots ,n-1$  satisfy the linear equations
\[
\partial_tb_k+ (v_2^{a+}+v_{2,k}^+)\partial_2b_k +b_k\partial_2(v_2^{a+}+v_{2,k}^+)=z_k
\]
with the right-hand sides
\[
z_k=[\tilde{r}_k]-[v_{2,k}]\partial_2H_{N,k}^- - H_{N,k}^-\partial_2 [v_{2,k}]-\bigl[\bigl(v_1^a+v_{1,k} - (v_2^a+v_{2,k})\partial_2(\varphi_k +\varphi^a)-\partial_t(\varphi_k+\varphi^a) \bigr) \partial_1H_{N,k}\bigr]
\]
and vanish in the past (we recall that $H_{N,k}^{\pm}:=H_{1,k}^{\pm}+H_1^{a\pm}-(H_{2,k}^{\pm}+H_2^{a\pm})\partial_2(\Psi_k^{\pm}+\Psi^{a\pm})$).
Using point c) of ($H_{n-1}$) and arguments similar to those towards the proof of estimate \eqref{Hn-1dv1}, we first estimate the right-hand sides $z_k$ and then get  a priori estimates of the solutions $b_k$ of the above linear equations with zero initial data. These a priori estimates are the desired estimates \eqref{Hnk}.
\end{proof}

\subsection{Estimate of the quadratic errors}

The quadratic errors
\[
\begin{split}
{e}'_k &= \mathcal{ L}(U_{k+1} ,{\Psi}_{k+1})-\mathcal{ L}(U_{k} ,{\Psi}_{k})
-\mathcal{ L}'(U_{k} ,{\Psi}_{k})(\delta{U}_k,\delta{\Psi}_{k}),\\
\tilde{e}'_k & = \bigl(\mathcal{ B}(U_{k+1},\varphi_{k+1})-\mathcal{ B}(U_{k},\varphi_{k})
-\mathcal{ B}'(U_{k},\varphi_{k})(\delta{U}_k,\delta \varphi_{k})\bigr)|_{x_1=0}
\end{split}
\]
can be rewritten as
\begin{align}
&{e}'_k=\int_0^1(1-\tau )\mathbb{L}''({U}^a+{U}_k +\tau\delta{U}_k, {\Psi}^a+{\Psi}_k
+\tau\delta{\Psi}_k)\bigl((\delta{U}_k,\delta{\Psi}_k),
(\delta{U}_k,\delta{\Psi}_k)\bigr) d\tau ,
\label{85} \\
&\tilde{e}'_k=\frac{1}{2}\,\mathbb{B}''\bigl(
(\delta{U}_k|_{x_1=0},\delta \varphi_k),(\delta{U}_k|_{x_1=0},\delta \varphi_k)\bigr)
\label{86}
\end{align}
by using the second derivatives of the operators $\mathbb{L}$ and $\mathbb{B}$:
\[
\begin{split}
\mathbb{L}''(\widehat{U},\widehat{{\Psi}})((U',{\Psi}'),(U'',{\Psi}'')):= &
\frac{\rm d}{{\rm d}\varepsilon}\mathbb{L}'(U_{\varepsilon},{\Psi}_{\varepsilon})(U',{\Psi}')|_{\varepsilon =0},\\
\mathbb{B}''((W',\varphi'),(W'',\varphi'')):=&
\frac{\rm d}{{\rm d}\varepsilon}\mathbb{B}'(W_{\varepsilon},\varphi_{\varepsilon})(W',\varphi')|_{\varepsilon =0},
\end{split}
\]
where $U_{\varepsilon}=\widehat{U}+\varepsilon U''$, $W_{\varepsilon}=\widehat{U}|_{x_1=0}+\varepsilon W''$,
$\varphi_{\varepsilon}=\hat{\varphi}+{\varepsilon}\varphi''$,
\[
\mathbb{L}'(\widehat{U},\widehat{{\Psi}})(U'',{\Psi}'')=\frac{\rm d}{{\rm d}\varepsilon}\mathbb{L}
(U_{\varepsilon},{\Psi}_{\varepsilon}),\quad
\mathbb{B}'(\widehat{U}|_{x_1=0},\hat{\varphi})(W'',{\varphi}'')=\frac{\rm d}{{\rm d}\varepsilon}\mathbb{B}(W_{\varepsilon},{\varphi}_{\varepsilon}) ,
\]
and ${\Psi}'$ and ${\Psi}''$ are associated to
$\varphi'$ and $\varphi''$ respectively like ${\Psi}$ is associated to $\varphi$. Moreover, $U'=(U'^+,U'^-)$, $U''=(U''^+,U''^-)$, etc. We easily compute the explicit form of $\mathbb{B}''$, that do not depend on the state $(\widehat{U},\hat{\varphi})$:
\begin{equation}
\mathbb{B}''((W',\varphi'),(W'',\varphi''))=
\begin{pmatrix}
0\\
0\\
0\\
[H'_1]\partial_2\varphi''+[H''_1]\partial_2\varphi'\\
v'^+_2\partial_2\varphi''+v''^+_2\partial_2\varphi'
\end{pmatrix}.
\label{B"}
\end{equation}
To estimate the quadratic errors by utilizing representations (\ref{85}) and (\ref{86}) we need estimates for
$\mathbb{L}''$ and $\mathbb{B}''$. They can easily be obtained from the explicit forms of $\mathbb{L}''$ and $\mathbb{B}''$ by applying the Moser-type and embedding inequalities. Omitting detailed calculations, as in \cite{Tcpam}, we get the following result:

\begin{proposition}
Let $T>0$ and $s\in\mathbb{N}$, with $s\geq 3$. Assume that $(\widehat{U} ,\hat{\varphi})\in
H^{s+1}(\Omega_T )\times H^{s+1}(\partial\Omega_T)$ and
\[
\|\widehat{U}\|_{H^3((\Omega_T)} +\|\hat{\varphi} \|_{H^{3}(\partial\Omega_T)}\leq \widetilde{K}.
\]
Then there exists a positive constant $\widetilde{K}_0$, that does not depend on $s$ and $T$, and there exists a constant $C(\widetilde{K}_0) >0$ such that, if $\widetilde{K}\leq \widetilde{K}_0$ and
$(U' ,\varphi'),\, (U'' ,\varphi'')\in H^{s+1}(\Omega_T )\times H^{s+1}(\partial\Omega_T)$, then
\[
\begin{split}
\|\mathbb{L}'' & (\widehat{U},\widehat{{\Psi}})((U',{\Psi}'),(U'',{\Psi}''))\|_{H^s(\Omega_T)}\\
&
\begin{split}
\leq C(\widetilde{K}_0)\bigl\{ &{\nl(\widehat{{U}},\hat{\varphi})\nr}_{s+1}{\nl({U}',\varphi')\nr}_{3}{\nl({U}'',\varphi'')\nr}_{3}\\
&
+{\nl({U}',\varphi')\nr}_{s+1}{\nl({U}'',\varphi'')\nr}_{3}
+ {\nl({U}'',\varphi'')\nr}_{s+1}{\nl({U}',\varphi')\nr}_{3}\bigr\},
\end{split}
\end{split}
\]
where
\[
{\nl({U} ,\varphi )\nr}_{\ell}:=\|{U} \|_{H^{\ell}(\Omega_T)} +\| \varphi\|_{H^{\ell}(\partial\Omega_T)}=
\|{U}^+ \|_{H^{\ell}(\Omega_T)} +\|{U}^- \|_{H^{\ell}(\Omega_T)} +\| \varphi\|_{H^{\ell}(\partial\Omega_T)}.
\]

If $({W}' ,\varphi'),\, ({W}'' ,\varphi'')\in H^{s}(\partial\Omega_T )\times H^{s+1}(\partial\Omega_T)$, then
\[
\begin{split}
\|\mathbb{B}'' & (({W}',\varphi'),({W}'',\varphi''))\|_{H^s(\partial\Omega_T)} \\
 &
\begin{split}
 \leq C(\widetilde{K}_0)\bigl\{ &
\|{W}'\|_{H^s(\partial\Omega_T)}\|\varphi''\|_{H^3(\partial\Omega_T)}
+\|{W}'\|_{H^3(\partial\Omega_T)}\|\varphi''\|_{H^{s+1}(\partial\Omega_T)} \\
 &+\|{W}''\|_{H^s(\partial\Omega_T)}\|\varphi'\|_{H^{3}(\partial\Omega_T)} +\|{W}''\|_{H^3(\partial\Omega_T)}\|\varphi'\|_{H^{s+1}(\partial\Omega_T)}\bigr\}.
\end{split}
\end{split}
\]
\label{p2a}
\end{proposition}

Without loss of generality we assume that the constant $\widetilde{K}_0=2C_*$, where $C_*$ is the constant from
(\ref{66}). By using (\ref{80}), (\ref{85}), (\ref{86}), and Proposition \ref{p2a}, as in \cite{Tcpam}, we obtain the following result:

\begin{lemma}
Let $\alpha \geq 4$. There exist $\delta >0$ sufficiently small, and $\theta_0 \geq 1$ sufficiently large, such that
for all $k=0,\ldots n-1$, and for all integer $s\in [3,\widetilde{\alpha}-1]$, we have the estimates
\begin{align}
 \|{e}'_k\|_{H^s(\Omega_T)}\leq & C\delta^2\theta_k^{L_1(s)-1}\Delta_k,\label{87}\\
 \|\tilde{e}'_k\|_{H^s(\partial\Omega_T)}\leq & C\delta^2\theta_k^{L_1(s)-1}\Delta_k,\label{88}
\end{align}
where $L_1(s)=\max \{ (s+1-\alpha )_+ +4-2\alpha ,s+2-2\alpha \}$.
\label{l4}
\end{lemma}

The proof of Lemma \ref{l4} is absolutely the same as that of lemma 4.8 in \cite{Tcpam}.

\subsection{Estimate of the first substitution errors}

The first substitution errors can be rewritten as follows:
\begin{equation}
\begin{split}
{e}''_k &= \mathcal{ L}'({U}_{k} ,{\Psi}_{k})(\delta{{U}}_k,\delta{\Psi}_{k})-
\mathcal{ L}'(S_{\theta_k}{U}_{k} ,S_{\theta_k}{\Psi}_{k})(\delta{{U}}_k,\delta{\Psi}_{k})
\\
 &
 \begin{split}
 =\int\limits_0^1 &\mathbb{L}''\bigl({U}^a+S_{\theta_k}{U}_k +\tau (I-S_{\theta_k}){U}_k, {\Psi}^a+S_{\theta_k}{\Psi}_k
 \\ &+ \tau (I -  S_{\theta_k}){\Psi}_k\bigr)\bigl((\delta{U}_k,\delta{\Psi}_k),
((I-S_{\theta_k}) {U}_k,(I-S_{\theta_k}) {\Psi}_k)\bigr) d\tau ,
\end{split}
\end{split}
\label{89}
\end{equation}
\begin{equation}
\begin{split}
\tilde{e}''_k & =
\bigl(\mathcal{ B}'({U}_{k},\varphi_{k})(\delta{{U}}_k,\delta \varphi_{k})-\mathcal{ B}'(S_{\theta_k}{U}_{k},S_{\theta_k}\varphi_{k})(\delta{{U}}_k,\delta \varphi_{k})\bigr)|_{x_1=0}\\
 & =\mathbb{B}''\bigl(
(\delta{U}_k|_{x_1=0},\delta \varphi_k),(({U}_k-S_{\theta_k}{U}_k)|_{x_1=0},\varphi_k- S_{\theta_k}\varphi_k)\bigr).
\end{split}
\label{90}
\end{equation}

Using \eqref{89} and \eqref{90} as well as (\ref{66}), ($H_{n-1}$), (\ref{82}), (\ref{83}), (\ref{84}) and Proposition \ref{p2a}, we get the following result.

\begin{lemma}
Let $\alpha \geq 4$. There exist $\delta >0$ sufficiently small, and $\theta_0 \geq 1$ sufficiently large, such that
for all $k=0,\ldots n-1$, and for all integer $s\in [6,\widetilde{\alpha}-2]$, one has
\begin{align}
& \|{e}''_k\|_{H^s(\Omega_T)}\leq C\delta^2\theta_k^{L_2(s)-1}\Delta_k,\nonumber\\
& \|\tilde{e}''_k\|_{H^s(\partial\Omega_T)}\leq C\delta^2\theta_k^{L_2(s)-1}\Delta_k,\nonumber
\end{align}
where $L_2(s)=\max \{ (s+1-\alpha )_+ +6-2\alpha ,s+5-2\alpha \}$.
\label{l5}
\end{lemma}

We can again refer to \cite{Tcpam} because the proof of Lemma \ref{l5} is the same as that of lemma 4.9 in \cite{Tcpam}.

\subsection{Construction and estimate of the modified state}
Since the approximate solution satisfies the hyperbolicity conditions \eqref{5.1'}, requirement \eqref{mf.1}  and the Rayleigh-Taylor sign condition \eqref{RT1} (see Lemma \ref{l2}) and since we shall require that the smooth modified state vanishes in the past, the state $(U^a+U_{n+1/2},\varphi^a+\varphi_{n+1/2})$ will satisfy \eqref{a5}, \eqref{cdass} and \eqref{RTL} (the ``relaxed'' versions of \eqref{5.1'}, \eqref{mf.1}  and \eqref{RT1})  for a sufficiently short time $T>0$. Therefore, while constructing the modified state we may focus only on constraints \eqref{a12'} and  \eqref{jc1'}.

\begin{proposition}
Let $\alpha \geq 6$. The exist some functions $U_{n+1/2}$ and $\varphi_{n+1/2}$, that vanish in the past, and such that
$(U^a+U_{n+1/2},\varphi^a+\varphi_{n+1/2})$ satisfies \eqref{a5}--\eqref{jc1'} and \eqref{RTL} for a sufficiently short time $T$. Moreover, these functions satisfy
\begin{equation}
\varphi_{n+1/2}=S_{\theta_n}\varphi_n, \quad S^{\pm}_{n+1/2}=S_{\theta_n}S^{\pm}_n,
 \label{93}
\end{equation}
and
\begin{equation}
\|U_{n+1/2} - S_{\theta_n}U_n\|_{H^s(\Omega_T)}\leq C\delta\theta_n^{s+1-\alpha}\quad \mbox{for}\ s\in [3,\tilde{\alpha}+3]
\label{94}
\end{equation}
for sufficiently small $\delta>0$ and $T>0$, and a sufficiently large $\theta_0\geq 1$.
\label{p3}
\end{proposition}

\begin{proof}
Note that estimate \eqref{94} can be proved for every $s\geq 3$ but below we will need it only for $s\in [3,\tilde{\alpha}+3]$.
Let $\varphi_{n+1/2}$ and the entropies $S^{\pm}_{n+1/2}$ be defined by (\ref{93}). We now define the pressures $p^{\pm}_{n+1/2}$ as follows:
\[
p^{\pm}_{n+1/2}:=S_{\theta_n}p^{\pm}_{n}\mp \frac{\chi}{2} S_{\theta_n}[p_{n}],
\]
where $\chi =\chi (x_1)$ is the same $C_0^{\infty}$ function which was used in \eqref{change2}. Clearly, $[p_{n+1/2}]=0$ that is in agreement with the boundary conditions \eqref{a12'} written for $(U^a+U_{n+1/2},\varphi^a+\varphi_{n+1/2})$. To get the estimate of $p^{\pm}_{n+1/2}-S_{\theta_n}p^{\pm}_{n}$ we first estimate the jump $[p_{n}]$ by using points a) and c) of the induction assumption:
\begin{equation}
\begin{split}
\|[p_{n}]\|_{H^s(\partial\Omega_T)} &\leq \|[p_{n-1}]\|_{H^s(\partial\Omega_T)}+\|[\delta p_{n-1}]\|_{H^s(\partial\Omega_T)} \\
 & \leq  \|\mathcal{ B}(U_{n-1}|_{x_1=0},\varphi_{n-1})\|_{H^s(\partial\Omega_T)} +C \|\delta U_{n-1}\|_{H^{s+1}(\Omega_T)} \leq C\delta\theta_n^{s-\alpha -1}
\end{split}
\label{n-1}
\end{equation}
for all integer $s\in[4,\alpha ]$. Using then the properties of the smoothing operators $S_{\theta_n}$ (see Proposition \ref{p1a}), we obtain
\[
\|p^{\pm}_{n+1/2}-S_{\theta_n}p^{\pm}_{n}\|_{H^s(\Omega_T)}\leq
\left\{ \begin{array}{ll} C\theta_n^{s-\alpha }\|[p_{n}]\|_{H^{\alpha}(\partial\Omega_T)}&\quad \mbox{for}\ s\in [\alpha,\tilde{\alpha}+3],\\
C\|[p_{n}]\|_{H^{s+1}(\partial\Omega_T)}&\quad \mbox{for}\ s\in [3,\alpha -1]
\end{array}\right.  \leq C\delta\theta_n^{s-\alpha}
\]
for $s\in [3,\tilde{\alpha}+3]$.

We first define the second components of the velocities $v_{n+1/2}^{\pm}$:
\begin{equation}
v^{\pm}_{2,n+1/2}:=  S_{\theta_n}v^{\pm}_{2,n}\mp \frac{\chi}{2} S_{\theta_n}[v_{2,n}] - \frac{\zeta}{2} S_{\theta_n}[\partial_1 v_{2,n-1}] + \frac{\zeta}{2}\mathcal{R}_n,
\label{201}
\end{equation}
where
\[
\mathcal{R}_n= S_{\theta_n}(\partial_1v_{2,n-1}^+ +\partial_1v_{2,n-1}^-) -\partial_1(S_{\theta_n}v_{2,n}^+ + S_{\theta_n}v_{2,n}^-) ,
\]
$\zeta =\zeta (x_1)= x_1\chi (x_1) \in C^{\infty}_0(\mathbb{R}^+)$, and we recall that $[\partial_1 v_{2,n-1}]=\partial_1 v^+_{2,n-1}|_{x_1=0}+\partial_1 v^-_{2,n-1}|_{x_1=0}$ (see \eqref{norm_jump}). Since $\chi (0)=1$, $\chi '(0)=0$, $\zeta (0)=0$ and $\zeta ' (0)=1$, we easily check that $[v_{2,n+1/2}]=[\partial_1v_{2,n+1/2}]=0$ that is in agreement with \eqref{a12'} and \eqref{jc1'} written for $(U^a+U_{n+1/2},\varphi^a+\varphi_{n+1/2})$ (we take \eqref{12.1a} into account). Using the same arguments as above, we get
\begin{equation}
\left\|\chi S_{\theta_n}[v_{2,n}]\right\|_{H^s(\Omega_T)}\leq C\|S_{\theta_n}[v_{2,n}]\|_{H^s(\partial\Omega_T)}\leq C\delta\theta_n^{s-\alpha}
\label{202}
\end{equation}
for $s\in [3,\tilde{\alpha}+3]$. In view of estimates \eqref{72} and \eqref{Hn-1dv1}, we obtain
\begin{equation}
\begin{split}
\left\|{\zeta} S_{\theta_n}[\partial_1 v_{2,n-1}]\right\|_{H^s(\Omega_T)} &\leq C\|S_{\theta_n}[\partial_1 v_{2,n-1}]\|_{H^s(\partial\Omega_T)} \\ & \leq
\left\{ \begin{array}{ll} C\theta_n^{s-\alpha +1}\|[\partial_1 v_{2,n-1}]\|_{H^{\alpha -1}(\partial\Omega_T)}&\quad \mbox{for}\ s\in [\alpha,\tilde{\alpha}+3],\\
C\|[\partial_1 v_{2,n-1}]\|_{H^{s}(\partial\Omega_T)}&\quad \mbox{for}\ s\in [3,\alpha -1]
\end{array}\right. \\ & \leq C\delta\theta_n^{s-\alpha}
\end{split}
\label{203}
\end{equation}
for $s\in [3,\tilde{\alpha}+3]$.

For estimating $\zeta\mathcal{R}_n$ we decompose $\mathcal{R}_n$ as
\[
\mathcal{R}_n=\sum_{\pm}\left( (S_{\theta_n} -I)(\partial_1v_{2,n}^{\pm}) -\partial_1\bigl\{(S_{\theta_n} -I)v_{2,n}^{\pm}\bigr\}
-S_{\theta_n}\partial_1(\delta v_{2,n-1}^{\pm})\right).
\]
For $s\in [\alpha , \tilde{\alpha} +3]$ one has
\[
\begin{split}
\|(S_{\theta_n} & -I)(\partial_1v_{2,n}^{\pm}) -\partial_1\bigl\{(S_{\theta_n} -I)v_{2,n}^{\pm}\bigr\}\|_{H^s(\Omega_T)}\\
 &\leq C \bigl\{
\|\partial_1(S_{\theta_n}v_{2,n}^{\pm}) \|_{H^s(\Omega_T)}+\|S_{\theta_n}(\partial_1v_{2,n}^{\pm}) \|_{H^s(\Omega_T)} \bigr\}\\ &
\leq C\bigl\{
\|S_{\theta_n}v_{2,n}^{\pm} \|_{H^{s+1}(\Omega_T)}
+\theta_n^{s-\alpha}\|v_{2,n}^{\pm}\|_{H^{\alpha+1}(\Omega_T)}\bigr\}\leq C\delta\theta_n^{s+1-\alpha},
\end{split}
\]
while for $s\in [3 , {\alpha} -1]$ we obtain (recall that $\tilde{\alpha}=\alpha +4$)
\[
\begin{split}
& \bigl\|\partial_1\bigl\{(S_{\theta_n} -I)v_{2,n}^{\pm}\bigr\} \bigr\|_{H^s(\Omega_T)}\leq C\delta\theta_n^{s+1-\alpha},
\\
& \|(S_{\theta_n}  -I)(\partial_1v_{2,n}^{\pm})\|_{H^s(\Omega_T)}\leq C\theta_n^{s-\alpha}\| v_{2,n}^{\pm}\|_{H^{\alpha+1}(\Omega_T)}\leq C\delta\theta_n^{s+1-\alpha}.
\end{split}
\]
Here we have, in particular, used Lemma \ref{l3}. Using \eqref{72} and point a) of ($H_{n-1}$), we get
\[
\begin{split}
\|S_{\theta_n} &\partial_1(\delta v_{2,n-1}^{\pm})\|_{H^s(\Omega_T)} \\
& \leq C\theta_n^{s-2}\|\delta v_{2,n-1}^{\pm}\|_{H^3(\partial\Omega_T)}\leq C\theta_n^{s-2}\delta\theta_{n-1}^{2-\alpha }
\theta_{n-1}^{-1}\leq C\delta\theta_n^{s-\alpha -1}
\end{split}
\]
for $s\in [3,\tilde{\alpha} +3]$. That is, we obtain
\begin{equation}
\|{\zeta}\mathcal{R}_n\|_{H^s(\Omega_T)}\leq C\|\mathcal{R}_n\|_{H^s(\Omega_T)}\leq C\delta\theta_n^{s+1-\alpha }
\label{204}
\end{equation}
for $s\in [3,\tilde{\alpha} +3]$. It follows from \eqref{201}--\eqref{204} that
\[
\|v^{\pm}_{2,n+1/2}-S_{\theta_n}v^{\pm}_{2,n}\|_{H^s(\Omega_T)}\leq C\delta\theta_n^{s+1-\alpha }
\]
for $s\in [3,\tilde{\alpha} +3]$. Moreover, considering \eqref{201} at $x_1=0$ and using \eqref{202}, we get one more estimate
\begin{equation}
\bigl\|(v^{\pm}_{2,n+1/2}-S_{\theta_n}v^{\pm}_{2,n})|_{x_1=0}\bigr\|_{H^s(\partial\Omega_T)}\leq C\delta\theta_n^{s-\alpha }
\label{301}
\end{equation}
for $s\in [3,\tilde{\alpha} +3]$ which will be used below.

We now define the first components of the velocities $v_{n+1/2}^{\pm}$:
\begin{equation}
\begin{split}
v^{\pm}_{1,n+1/2}: = & S_{\theta_n}v^{\pm}_{1,n} -\chi (S_{\theta_n}v_{1,n}^{\pm})|_{x_1=0}
+\chi (v_2^{a\pm}+v^{\pm}_{2,n+1/2})|_{x_1=0}\,\partial_2\varphi_{n+1/2}
\\ & +\chi (v^{\pm}_{2,n+1/2})\bigr|_{x_1=0}\,\partial_2\varphi^{a} +\chi\partial_t\varphi_{n+1/2}- \frac{\zeta}{2} S_{\theta_n}[\partial_1 v_{1,n-1}] + \frac{\zeta}{2}\widetilde{\mathcal{R}}_n,
\end{split}
\label{205}
\end{equation}
where
\[
\widetilde{\mathcal{R}}_n= S_{\theta_n}(\partial_1v_{1,n-1}^+ +\partial_1v_{1,n-1}^-) -\partial_1(S_{\theta_n}v_{1,n}^+ + S_{\theta_n}v_{1,n}^-)
\]
coincides with $\mathcal{R}_n$ if $v_{1,n-1}^{\pm}$ and $v_{1,n}^{\pm}$ are replaced with $v_{2,n-1}^{\pm}$ and $v_{2,n}^{\pm}$ respectively.
We have
\[
v^{\pm}_{1,n+1/2}=(v_2^{a\pm}+v^{\pm}_{2,n+1/2})\partial_2\varphi_{n+1/2} +v^{\pm}_{2,n+1/2}\partial_2\varphi^{a} +\partial_t\varphi_{n+1/2}\quad\mbox{at }x_1=0,
\]
i.e., in view of \eqref{12.1a}, we get the last boundary condition in \eqref{a12'} written for $(U^a+U_{n+1/2},\varphi^a+\varphi_{n+1/2})$. Moreover, taking $[v_{2,n+1/2}]=0$ into account, we obtain $[v_{1,n+1/2}]=0$. It can be also easily seen that $[\partial_1v_{1,n+1/2}]=0$.

For estimating $v^{\pm}_{1,n+1/2}- S_{\theta_n}v^{\pm}_{1,n}$ we rewrite \eqref{205} as
\begin{equation}
v^{\pm}_{1,n+1/2}- S_{\theta_n}v^{\pm}_{1,n}=\mathcal{V}_{n} +\chi\mathcal{P}_{n}^\pm +\chi\mathcal{G}_n^\pm ,
\label{205'}
\end{equation}
where
\[
\begin{split}
\mathcal{V}_{n}  & = - \frac{\zeta}{2} S_{\theta_n}[\partial_1 v_{1,n-1}] + \frac{\zeta}{2}\widetilde{\mathcal{R}}_n ,\\
\mathcal{P}_{n}^\pm &=  (v^{\pm}_{2,n+1/2}-S_{\theta_n}v^{\pm}_{2,n})|_{x_1=0}\,\partial_2(\varphi^a+S_{\theta_n}\varphi_n ),\\
\mathcal{G}_{n}^\pm &=\partial_t\varphi_{n+1/2} -(S_{\theta_n}v_{1,n}^{\pm})|_{x_1=0} + (v_2^{a\pm}+S_{\theta_n}v^{\pm}_{2,n})|_{x_1=0}\,\partial_2\varphi_{n+1/2} + (S_{\theta_n}v^{\pm}_{2,n})|_{x_1=0}\,\partial_2\varphi^{a}.
\end{split}
\]
Exactly as in \eqref{203}, we get
\[
\left\|{\zeta}S_{\theta_n}[\partial_1 v_{1,n-1}]\right\|_{H^s(\Omega_T)} \leq C\delta\theta_n^{s-\alpha}
\]
for $s\in [3,\tilde{\alpha}+3]$. By virtue of the counterpart of \eqref{204} for $\widetilde{\mathcal{R}}_n$, we obtain the estimate
\begin{equation}
\|\mathcal{V}_{n}\|_{H^s(\Omega_T)}\leq C\delta\theta_n^{s+1-\alpha}
\label{206}
\end{equation}
for $s\in [3,\tilde{\alpha}+3]$.

Applying the Moser-type calculus inequalities and Sobolev's embeddings and using \eqref{66}, \eqref{83}, \eqref{84} and \eqref{301},  we estimate the second term in the right-hand side of \eqref{205'}:
\begin{equation}
\begin{split}
\|\chi\mathcal{P}_{n}^{\pm}\|_{H^s(\Omega_T)} & \leq C\|\mathcal{P}_{n}^{\pm}\|_{H^s(\partial\Omega_T)} \\
& \;
\begin{split}
\leq C \bigl\{ & \|(v^{\pm}_{2,n+1/2}-S_{\theta_n}v^{\pm}_{2,n})|_{x_1=0}\|_{H^s(\partial\Omega_T)}\|\varphi^a+S_{\theta_n}\varphi_n \|_{H^3(\partial\Omega_T)} \\
& +\|(v^{\pm}_{2,n+1/2}-S_{\theta_n}v^{\pm}_{2,n})|_{x_1=0}\|_{H^3(\partial\Omega_T)}\|\varphi^a+S_{\theta_n}\varphi_n \|_{H^{s+1}(\partial\Omega_T)}\bigr\}
\end{split}\\
& \leq C\bigl\{ \delta \theta_n^{s-\alpha} (C_*+\delta\theta_n^{(3-\alpha)_+})+ \delta \theta_n^{3-\alpha} (C_*+\delta\theta_n^{\ell (s,\alpha )})\bigr\} \\
& \leq  C \delta \theta_n^{s-\alpha}
\end{split}
\label{207}
\end{equation}
for $s\in [3,\tilde{\alpha}+3]$, where $\ell (s,\alpha )=(s+1-\alpha)_+$ for $s\neq \alpha -1$ and $\ell (\alpha -1,\alpha )=1$.
Note that the estimate $\|\varphi^a\|_{H^{s+1}(\partial\Omega_T)}$ used above for $s\in [3,\tilde{\alpha}+3]$ is in agreement with assumption \eqref{66} (recall that $\tilde{\alpha}+3=\alpha +7=m+8$).  In \eqref{207} we also took into account the assumption $\alpha \geq 6$. Regarding the terms $\mathcal{G}_n^{\pm}$ (see \eqref{205'}), up to dropping the superscripts $\pm$ by $v_{1,n}^\pm$ and $v_{2,n}^\pm$ they have exactly the same form as $\mathcal{G}$ appearing in the proof of proposition 4.10 in \cite{Tcpam}. That is, here we can just refer to \cite{Tcpam} and write down the estimate
\begin{equation}
\|\chi\mathcal{G}_{n}^{\pm}\|_{H^s(\Omega_T)} \leq C\|\mathcal{G}_{n}^{\pm}\|_{H^s(\partial\Omega_T)}\leq C \theta_n^{s+1-\alpha}
\label{208}
\end{equation}
for $s\in [3,\tilde{\alpha}+3]$. It follows from \eqref{205'}--\eqref{208} that
\[
\|v^{\pm}_{1,n+1/2}- S_{\theta_n}v^{\pm}_{1,n}\|_{H^s(\Omega_T)} \leq C\theta_n^{s+1-\alpha}
\]
for $s\in [3,\tilde{\alpha}+3]$.

At last,  we define the magnetic fields  in the same way as the pressures $p^{\pm}_{n+1/2}$ above:
\[
H^{\pm}_{n+1/2}:=S_{\theta_n}H^{\pm}_{n}\mp \frac{\chi}{2} S_{\theta_n}[H_{n}].
\]
Using then the arguments as in \eqref{n-1}, points a) and c) of the induction assumption (for estimating the ``tangential'' jump $[H_{\tau ,n-1}]$) and  estimate \eqref{Hnk} for the ``normal'' jump $[H_{N ,n-1}]$, we obtain
\[
\|H^{\pm}_{\tau ,n+1/2}-S_{\theta_n}H^{\pm}_{\tau ,n}\|_{H^s(\Omega_T)}\leq C\delta\theta_n^{s-\alpha}
\]
(exactly as for $p^{\pm}_{n+1/2}$) and
\[
\|H^{\pm}_{N ,n+1/2}-S_{\theta_n}H^{\pm}_{N ,n}\|_{H^s(\Omega_T)}
\leq
\left\{ \begin{array}{ll} C\theta_n^{s-\alpha +1}\|[H_{N,n}]\|_{H^{\alpha -1}(\partial\Omega_T)}&\ \mbox{for}\ s\in [\alpha,\tilde{\alpha}+3],\\
C\|[H_{N,n}]\|_{H^{s}(\partial\Omega_T)}&\ \mbox{for}\ s\in [3,\alpha -1]
\end{array}\right.
\leq C\delta\theta_n^{s-\alpha}
\]
for $s\in [3,\tilde{\alpha}+3]$.
\end{proof}

\subsection{Estimate of the second substitution errors}

Since the estimate \eqref{94} for the modified state is absolutely the same as the corresponding one in \cite{Tcpam}, the rest of the proof of Theorem \ref{t1} is the same as that in \cite{Tcpam}. But, for the reader's convenience we will sketch here the remaining part of the proof with references to \cite{Tcpam} for detailed arguments and technical calculations.

The second substitution errors
\[
{e}'''_k = \mathcal{ L}'(S_{\theta_k}{U}_{k} ,S_{\theta_k}{\Psi}_{k})(\delta{{U}}_k,\delta{\Psi}_{k})- \mathcal{ L}'({U}_{k+1/2} ,{\Psi}_{k+1/2})(\delta{{U}}_k,\delta{\Psi}_{k})
\]
and
\[
\tilde{e}'''_k =
\bigl(\mathcal{ B}'(S_{\theta_k}{U}_{k},S_{\theta_k}\varphi_{k})(\delta{{U}}_k,\delta \varphi_{k})-
\mathcal{ B}'({U}_{k+1/2},\varphi_{k+1/2})(\delta{{U}}_k,\delta \varphi_{k})
\bigr)|_{x_1=0}
\]
can be written as
\begin{equation}
\begin{split}
{e}'''_k =\int_0^1 \mathbb{L}''\bigl( & {U}^a+{U}_{k+1/2} +\tau (S_{\theta_k}{U}_k-{U}_{k+1/2}),{\Psi}^a
\\ &
+S_{\theta_k}{\Psi}_k) \bigl((\delta{U}_k,\delta{\Psi}_k),
(S_{\theta_k} {U}_k -{U}_{k+1/2},0) \bigr) d\tau ,
\end{split}
\label{95}
\end{equation}
\begin{equation}
\tilde{e}'''_k
 =\mathbb{B}''\bigl(
(\delta{U}_k|_{x_1=0},\delta \varphi_k),((S_{\theta_k}{U}_k-{U}_{k+1/2})|_{x_1=0},0)\bigr).
\label{96}
\end{equation}
Employing (\ref{95}) and (\ref{96}) as well as Lemma \ref{l3} and Proposition \ref{p3}, we get the following result.

\begin{lemma}
Let $\alpha \geq 6$. There exist $\delta >0$, $T>0$ sufficiently small, and $\theta_0 \geq 1$ sufficiently large, such that for all $k=0,\ldots n-1$, and for all integer $s\in [3,\widetilde{\alpha}-1]$, one has
\begin{align}
& \|{e}'''_k\|_{H^s(\Omega_T)}\leq C\delta^2\theta_k^{L_3(s)-1}\Delta_k,\label{97} \\
& \|\tilde{e}'''_k\|_{H^s(\partial\Omega_T)}\leq C\delta^2\theta_k^{L_3(s)-1}\Delta_k,\label{97'}
\end{align}
where $L_3(s)=\max \{ (s+1-\alpha )_+ +8-2\alpha ,s+5-2\alpha \}$.
\label{l6}
\end{lemma}

\begin{proof}
For the proof of estimate \eqref{97} we refer to \cite{Tcpam} (see there the proof of lemma 4.11). Our assumption $\alpha \geq 6$ in Lemma \ref{l6} is more restrictive than the assumption  $\alpha \geq 4$ in lemma 4.11 in \cite{Tcpam} because it was necessary for the proof of  Proposition \ref{p3}. Using the explicit form of $\mathbb{B}''$ in \eqref{B"}, one has
\[
\tilde{e}'''_k=
\begin{pmatrix}
0\\
0\\
0\\
[S_{\theta_k}{H}_{1,k}-{H}_{1,k+1/2}]\partial_2(\delta\varphi_k)\\
(S_{\theta_k}{v}_{2,k}-{v}_{2,k+1/2})|_{x_1=0}\partial_2(\delta\varphi_k)
\end{pmatrix}.
\]
Applying then point a) of ($H_{n-1}$), Proposition \ref{p2a}, estimate \eqref{301} and the analogous estimate which can be obtained for $(H^{\pm}_{1,n+1/2}-S_{\theta_n}H^{\pm}_{1,n})|_{x_1=0}$, we get the estimate
\[
\|\tilde{e}'''_k\|_{H^s(\partial\Omega_T)}\leq C\delta^2\theta_k^{s+3-2\alpha}\Delta_k.
\]
Roughening the last estimates gives \eqref{97'} (the roughened estimate \eqref{97'} is enough for the proof of Lemma \ref{l8} below).
\end{proof}

\subsection{Estimate of the last error term}

We now estimate the last error term
\[
{ D}_{k+1/2}\delta\Psi_k=\frac{\delta\Psi_k}{\partial_1(\Phi^a+\Psi_{n+1/2})}\,\partial_1\left\{ \mathbb{L}({U}^a +{U}_{k+1/2} ,{\Psi}^a+{\Psi}_{k+1/2})\right\}.
\]
Note that
\[
|\partial_1(\Phi^{a\pm}+\Psi^{\pm}_{n+1/2})|=|\pm 1+\partial_1(\Psi^{a\pm}+\Psi^{\pm}_{n+1/2})|\geq 1/2,
\]
provided that $T$ and $\delta$ are small enough.

Referring again for a detailed proof to \cite{Tcpam} (see there lemma 4.12), here we just formulate the following result.

\begin{lemma}
Let $\alpha \geq 6$. There exist $\delta >0$, $T>0$ sufficiently small, and $\theta_0 \geq 1$ sufficiently large, such that for all $k=0,\ldots n-1$, and for all integer $s\in [3,\widetilde{\alpha}-2]$, one has
\[
\|{D}_{k+1/2}\delta{\Psi}_k\|_{H^s(\Omega_T)}\leq C\delta^2\theta_k^{L(s)-1}\Delta_k,
\]
where $L(s)=\max \{ (s+2-\alpha )_+ +8-2\alpha ,(s+1-\alpha)_++9-2\alpha ,s+6-2\alpha \}$.
\label{l7}
\end{lemma}

\subsection{Convergence of the iteration scheme}

Lemmas \ref{l4}--\ref{l7} yield the estimate of ${e}_n$ and $\tilde{e}_n$ defined in (\ref{77}) as the sum of all the errors of the $n$th step.

\begin{lemma}
Let $\alpha \geq 6$. There exist $\delta >0$, $T>0$ sufficiently small, and $\theta_0 \geq 1$ sufficiently large, such that
for all $k=0,\ldots n-1$, and for all integer $s\in [3,\widetilde{\alpha}-2]$, one has
\begin{equation}
\|{e}_k\|_{H^s(\Omega_T)}+\|\tilde{e}_k\|_{H^s(\partial\Omega_T)}\leq C\delta^2\theta_k^{L(s)-1}\Delta_k,
\label{103}
\end{equation}
where $L(s)$ is defined in Lemma \ref{l7}.
\label{l8}
\end{lemma}

Lemma \ref{l8} gives the estimate of the accumulated errors ${E}_n$ and $\widetilde{E}_n$.

\begin{lemma}
Let $\alpha \geq 7$.  There exist $\delta >0$, $T>0$ sufficiently small, and $\theta_0 \geq 1$ sufficiently large, such that
\begin{equation}
\|{E}_n\|_{H^{\alpha +2}(\Omega_T)}+\|\widetilde{E}_n\|_{H^{\alpha +2}(\partial\Omega_T)}\leq C\delta^2\theta_n.
\label{104}
\end{equation}
\label{l9}
\end{lemma}

\begin{proof} One can check that $L(\alpha +2)\leq 1$ if $\alpha \geq 7$, where $L(s)$ is defined in Lemma \ref{l7}. It follows from (\ref{103}) that
\[
{\nl ({E}_n,\widetilde{E}_n)\nr}_{\alpha +2}\leq \sum_{k=0}^{n-1}{\nl ({e}_k,\widetilde{e}_k)\nr}_{\alpha +2}\leq \sum_{k=0}^{n-1}C\delta^2\Delta_k\leq C\delta^2\theta_n
\]
for $\alpha \geq 7$ and $\alpha +2\in [3,\tilde{\alpha}-2]$, i.e., $\tilde{\alpha}\geq \alpha +4$. The minimal possible
$\tilde{\alpha}$ is $\alpha +4$, i.e., our choice $\tilde{\alpha}= \alpha +4$ is suitable.
\end{proof}

Referring to \cite{Tcpam} for the proof (with the help of (\ref{72}), (\ref{74}), (\ref{103}), and (\ref{104})), below we just write down
the estimates of the source terms ${f}_n$ and ${g}_n$ defined in (\ref{79}).

\begin{lemma}
Let $\alpha \geq 7$.  There exist $\delta >0$, $T>0$ sufficiently small, and $\theta_0 \geq 1$ sufficiently large, such that for all integer $s\in [3,\widetilde{\alpha}+1]$, one has
\begin{align}
& \|{f}_n\|_{H^{s}(\Omega_T)}\leq  C\Delta_n\bigl\{ \theta_n^{s-\alpha -2}\left( \|{f}^a\|_{H^{\alpha +1}(\Omega_T)}+\delta^2\right)+\delta^2\theta_n^{L(s)-1} \bigr\},\nonumber\\[6pt]
 & \|{g}_n\|_{H^{s}(\partial\Omega_T)}\leq  C\delta^2\Delta_n\bigl( \theta_n^{L(s)-1}+\theta_n^{s-\alpha -2}\bigr).
\nonumber
\end{align}
\label{l10}
\end{lemma}

We are now in a position to obtain the estimate of the solution to problem (\ref{75}) by employing the tame estimate
(\ref{38}). Then the estimate of $(\delta U_n,\delta\varphi_n)$ follows from formula (\ref{76}).

\begin{lemma}
Let $\alpha \geq 7$.  There exist $\delta >0$, $T>0$ sufficiently small, and $\theta_0 \geq 1$ sufficiently large, such that for all integer $s\in [3,\widetilde{\alpha}]$, one has
\begin{equation}
\|\delta U_n\|_{H^{s}(\Omega_T)}+\|\delta \varphi_n\|_{H^{s}(\partial\Omega_T)}\leq  \delta\theta_n^{s-\alpha -1}\Delta_n.
\label{107}
\end{equation}
\label{l11}
\end{lemma}

The proof of Lemma \ref{l11} is absolutely the same as the proof of lemma 4.17 in \cite{Tcpam}.

\begin{remark}
{\rm
While applying the tame estimate (\ref{38}) in the proof of  Lemma \ref{l11} we need to estimate the approximate solution in the $H^{s+3}$ norm. Since $s+3 =\widetilde{\alpha} +3 =\alpha +7=m+8$ for $s=\widetilde{\alpha}$, here the boundedness of the approximate solution in the $H^{m+8}$ norm is enough. But, we recall that for getting estimates \eqref{207} and \eqref{208} we needed the more restrictive assumption \eqref{66}.
}
\end{remark}

Inequality (\ref{107}) is point a) of ($H_n$). The lemma below gives us points b) and c) of ($H_n$).

\begin{lemma}
Let $\alpha \geq 7$.  There exist $\delta >0$, $T>0$ sufficiently small, and $\theta_0 \geq 1$ sufficiently large, such that for all integer $s\in [3,\widetilde{\alpha}-2]$
\begin{equation}
\|\mathcal{ L}({U}_n,{\Psi}_n)-{f}^a\|_{H^s(\Omega_T)}\leq 2\delta\theta_n^{s-\alpha -1}.
\label{193}
\end{equation}
Moreover, for all integer $s\in [4,{\alpha}]$ one has
\begin{equation}
\|\mathcal{ B}({U}_n|_{x_1=0},\varphi _n)\|_{H^s(\partial\Omega_T)}\leq \delta\theta_n^{s-\alpha -1}.
\label{194}
\end{equation}
\label{l12}
\end{lemma}

Again, we can just refer to \cite{Tcpam} for the proof of Lemma \ref{l12} (see the proof of lemma 4.19 in \cite{Tcpam}). As follows from Lemmas \ref{l11} and \ref{l12}, we have proved that $(H_{n-1})$ implies $(H_{n})$, provided that
$\alpha \geq 7$, $\tilde{\alpha}=\alpha +4$, the constant $\theta_0\geq 1$ is large enough, and $T>0$, $\delta >0$ are small enough. Fixing now the constants $\alpha$, $\delta$, and $\theta_0$, exactly as in \cite{Tcpam} (see there the proof of lemma 4.20), we can prove that $(H_0)$ is true. That is, we have

\begin{lemma}
If the time $T>0$ is sufficiently small, then $(H_0)$ is true.
\label{l13}
\end{lemma}

\subsection{The proof of Theorem \ref{t1}}

We consider initial data $({U}_0,\varphi_0)\in H^{m+17/2}(\Omega)\times H^{m+17/2}(\partial\Omega)$ satisfying all the assumptions of Theorem \ref{t1}. In particular, they satisfy the compatibility conditions up to order $\mu=m+8$ (see Definition \ref{d1}). Then, thanks to Lemmas \ref{l4.1} and \ref{l2} we can construct an approximate solution $({U}^{a},\varphi^a)\in H^{m +9}(\Omega_T)\times H^{m +9}(\partial\Omega_T)$ that satisfies (\ref{66}). As follows from Lemmas \ref{l11}--\ref{l13}, $(H_n)$ holds for all integer $n\geq 0$, provided that $\alpha \geq 7$, $\tilde{\alpha}=\alpha +4$, the constant $\theta_0\geq 1$ is large enough, and the time $T>0$ and the constant $\delta >0$ are small enough.  In particular, ($H_n$) implies
\[
\sum_{n=0}^{\infty}\left\{ \|\delta{U}_n\|_{H^m(\Omega_T)} +\|\delta \varphi_n\|_{H^m(\partial\Omega_T)}\right\} < \infty.
\]
Hence, the sequence $({U}_n,\varphi_n)$ converges in $H^{m}(\Omega_T)\times H^{m}(\partial\Omega_T)$ to some
limit $({U} ,\varphi )$. Recall that $m=\alpha -1 \geq 6$. Passing to the limit in (\ref{193}) and (\ref{194}) with
$s=m$, we obtain (\ref{69})--(\ref{71}). Consequently, ${U} :={U} +{U}^a$, $\varphi := \varphi +\varphi^a$ is a solution of problem \eqref{11.1}--\eqref{13.1}. As was already noted in Section \ref{s1}, this solution is unique. This completes the proof of Theorem \ref{t1}.

\appendix
\section{MHS equilibria in a gravitational field}\label{appA}

If we take into account the gravitational field $\mathcal{G}\in \mathbb{R}^2$, then in the right-hand side of the second (vector) equation in \eqref{3} we have the term $\rho \mathcal{G}$ and the MHS equilibria
\[
\bar{U}^\pm (x) =(\bar{p}^\pm (x),0,\bar{H}^\pm (x), \bar{S}^\pm (x))
\]
obey the systems
\[
\nabla^\pm \bar q^\pm =(\bar H^\pm\cdot\nabla^\pm )\bar H^\pm +\bar \rho^\pm \mathcal{G}
\]
in the half-plane $\mathbb{R}^2_+$ (recall that $\nabla^\pm:= (\pm\partial_1,\partial_2)$). Let
\[
\mathcal{G}=(G,0),\quad \bar{p}^\pm =\bar{p}^\pm (x_1),\quad \bar{H}_1^\pm=\bar{H}_1^\pm (x_1),
 \]
 \[
 \bar{H}_2^\pm =\bar{H}_2^0={\rm const},\quad \mbox{and}\quad \bar{S}^\pm ={\rm const}\quad (\bar{S}^+\neq \bar{S}^-),
  \]
where
\[
\bar{p}^+(0)=\bar{p}^-(0)=\bar{p}_0,\quad \bar{H}_1^+(0)=\bar{H}_1^-(0)=\bar{H}_1^0,
\]
and $G$ denotes Newton's gravitational constant. Then the functions $\bar{H}_1^\pm (x_1)$ can be arbitrary functions satisfying \eqref{mf.1} (with $\varphi =0$), i.e., $\bar{H}_1^0\neq 0$, whereas the pressures $\bar{p}^\pm$ solve the equations
\begin{equation}
\pm\frac{d\bar{p}^\pm}{d x_1}=B^\pm (\bar{p}^\pm )^{\frac{1}{\gamma}}\qquad \mbox{for}\ x_1>0,
\label{pbar}
\end{equation}
where $B^\pm = G A \,e^{-\frac{\bar{S}^\pm}{\gamma}}>0$ (see \eqref{pg}).
The solutions of \eqref{pbar} are
\[
\bar{p}^\pm = \left(\pm \frac{\gamma -1}{\gamma}B^\pm x_1+\bar{p}_0^{\frac{\gamma -1}{\gamma}} \right)^{\frac{\gamma}{\gamma -1}}
\]
which satisfy \eqref{RT1} if $\bar{S}^+<\bar{S}^-$, i.e., $\bar{\rho}^+>\bar{\rho}^-$ meaning that the heavier plasma lies below the lighter plasma and the classical Rayleigh-Taylor instability does not happen. The consideration of these solutions on the interval $(0,a)$, with
\[
a<\frac{\gamma\, \bar{p}_0^{\frac{\gamma -1}{\gamma}}}{(\gamma -1)B^-},
\]
guarantees the strict positivity (\eqref{5.1} holds) and boundedness of the pressures $\bar{p}^\pm$.

At the right end point of the interval $(0,a)$ we should set the boundary conditions $\bar{H}_1^\pm (a)=0$.
If in \eqref{change2} we choose, without loss of generality, such a cut-off function $\chi$ that $\chi (\pm a)=0$, then before the change of variables \eqref{change2} we had the boundary conditions
\[
\bar{H}_1^+ =0\quad \mbox{at}\ x_1=a\qquad \mbox{and}\qquad \bar{H}_1^- =0\quad \mbox{at}\ x_1=-a.
\]
For the original nonstationary and nonstatic ($v\not\equiv 0$) problem \eqref{4}, \eqref{bcond}, \eqref{indat} (with the gravitational term $\rho \mathcal{G}$) this corresponds to the consideration of the contact discontinuity between the perfectly conducting rigid walls $x_1= a$ and $x_1= -a$ with the usual boundary conditions
\begin{equation}
v_1^\pm=H_1^\pm = 0\quad \mbox{on}\ [0,T]\times \{x_1=\pm a\}\times\mathbb{R}.
\label{wall}
\end{equation}
In this case the reference domains
\[
\Omega^+(t)=\{\varphi (t,x_2) <x_1<a\}\quad\mbox{and} \quad \Omega^-(t)=\{-a<x_1<\varphi (t,x_2)\}.
\]

Since gravity just contributes with the lower-order term $\rho \mathcal{G}$ in the MHD system, using the results of \cite{Sec95a,YM} for the MHD system with a perfectly conducting wall boundary condition, we can easily extend the results of \cite{Cont1} and the present paper to the case of the above reference domains and the boundary conditions \eqref{wall} on their outer boundaries $x_1=\pm a$. However, since, unlike the situation with the characteristic free boundary $x_1=\varphi (t,x_2)$ of contact discontinuity,  the loss of derivatives in the normal direction to the characteristic boundaries $x_1=\pm a$ cannot be compensated (see \cite{Sec95a,YM}), the functional setting has to be provided by the anisotropic weighted Sobolev spaces $H^m_*$ (see, e.g., \cite{Sec95a,ST,T09,YM}) but not the usual Sobolev spaces $H^m$ as in \cite{Cont1} or the present paper.

\end{document}